%% file: HigherdimSubspaces.tex
\documentclass{amsart}
\input{preamble}

\usepackage{tikz-cd}

\newcommand{\StSpin}{\mathbf{H}} 
\newcommand{\StSpinpt}{\mathbf{H}^{\mathrm{pt}}}
\newcommand{\StSO}{\mathbf{H}'}

\newcommand{\StDiag}{\mathbf{\Delta H}}

\newcommand{\bStDiag}{\mathbf{\Delta \bar{H}}}

\newcommand{\bStDiagcpt}{\Delta H^{\mathrm{cpt}}}

\newcommand{\gQ}{\eta_Q}

\newcommand{\shapesp}[1]{\mathcal{S}_{#1}}

\newcommand{\kfree}[2]{#1^{[#2]}}

\newcommand{\qsubscr}[1]{q_{{ }_{#1}}}
\newcommand{\isog}{\rho_Q}

\begin{document}

\title{Equidistribution of rational subspaces and their shapes}

\author{Menny Aka}
\address{Department of Mathematics, ETH Z{\"u}rich, 
Ramistrasse 101, Z{\"u}rich, Switzerland}

\author{Andrea Musso}
\address{D-GESS, ETH Z{\"u}rich, 
Ramistrasse 101, Z{\"u}rich, Switzerland}

\author{Andreas Wieser}
\address{Einstein Institute of Mathematics, Givat Ram, Hebrew University of Jerusalem,  Jerusalem, Israel}
\thanks{A.W. was supported by  ERC grant HomDyn, ID 833423, SNF grant 178958, and the SNF Doc. Mobility grant 195737.}





\begin{abstract}
To any $k$-dimensional subspace of $\Q^n$ one can naturally associate a point in the Grassmannian $\Grk(\R)$ and two shapes of lattices of rank $k$ and $n-k$ respectively. 
These lattices originate by intersecting the $k$-dimensional subspace and its orthogonal with the lattice $\Z^n$. Using unipotent dynamics we prove simultaneous equidistribution of all of these objects under congruence conditions when $(k,n) \neq (2,4)$.
\end{abstract}

\maketitle


\input{Sections/intro.tex}

\newpage
\part{Homogeneous results}
For an overview of the contents of this part, we refer the reader to \S\ref{sec:organization}.
\input{Sections/groups.tex}

\input{Sections/dynamical.tex}
\input{Sections/dynamical2.tex}
\newpage
\part{From equidistribution of orbits to the main theorems}
For the contents of this part we refer the reader to the overview of this article in \S\ref{sec:organization}.
\input{Sections/quadraticforms.tex}
\input{Sections/moduli.tex}

\input{Sections/equidistr_moduli.tex}

\newpage
\begin{appendix}
\input{Sections/non-empty.tex}
\input{Sections/primitive.tex}
\end{appendix}

\bibliographystyle{amsalpha}
\bibliography{Bibliography}

\end{document}

%% file: preamble.tex
\usepackage{amsmath,amssymb,mathrsfs,hyperref, amsthm, mathtools,commath,comment}
\usepackage{xfrac} 
\usepackage{bbm}
\usepackage{xcolor}
\usepackage{tikz-cd}
\usepackage{csquotes}
\usepackage[shortlabels]{enumitem}
\usepackage[nobysame,abbrev,alphabetic]{amsrefs}

\usepackage{blindtext}

\makeatletter
\renewcommand\part{%
   \if@noskipsec \leavevmode \fi
   \par
   \addvspace{4ex}%
   \@afterindentfalse
   \secdef\@part\@spart}

\def\@part[#1]#2{%
    \ifnum \c@secnumdepth >\m@ne
      \refstepcounter{part}%
      \addcontentsline{toc}{part}{\thepart\hspace{1em}#1}%
    \else
      \addcontentsline{toc}{part}{#1}%
    \fi
    {\parindent \z@ \raggedright
     \interlinepenalty \@M
     \normalfont
     \ifnum \c@secnumdepth >\m@ne
       \Large\bfseries \partname\nobreakspace\thepart
       \par\nobreak
     \fi
     \huge \bfseries #2%
     \par}%
    \nobreak
    \vskip 3ex
    \@afterheading}
\def\@spart#1{%
    {\parindent \z@ \raggedright
     \interlinepenalty \@M
     \normalfont
     \huge \bfseries #1\par}%
     \nobreak
     \vskip 3ex
     \@afterheading}
\makeatother

\newtheorem{Thm}{Theorem}[section]
\newtheorem{Prop}[Thm]{Proposition}
\newtheorem{Lem}[Thm]{Lemma}
\newtheorem{Cor}[Thm]{Corollary}
\newtheorem{Conj}[Thm]{Conjecture}
\newtheorem*{claim*}{Claim}

\theoremstyle{remark}
\newtheorem{remark}[Thm]{Remark}
\newtheorem{Ex}[Thm]{Example}

\numberwithin{equation}{section}

\newcommand{\de}{\, \mathrm{d}}

\newcommand{\Mat}{\mathrm{Mat}}
\newcommand{\SL}{\mathrm{SL}}
\newcommand{\GL}{\mathrm{GL}}
\newcommand{\PGL}{\mathrm{PGL}}
\newcommand{\Spin}{\mathrm{Spin}}
\newcommand{\SO}{\mathrm{SO}}
\newcommand{\Orth}{\mathrm{O}}

\newcommand{\SPnk}{\mathbf{P}_{n,k}}

\newcommand{\SDnk}{\mathbf{D}_{n,k}}

\newcommand{\G}{\mathbf{G}}
\newcommand{\SG}{\mathbf{G}}

\newcommand{\bSG}{\bar{\G}}

\newcommand{\Qp}{\mathbb{Q}_p}
\newcommand{\Zp}{\mathbb{Z}_p}
\newcommand{\R}{\mathbb{R}}
\newcommand{\Q}{\mathbb{Q}}
\newcommand{\Z}{\mathbb{Z}}
\newcommand{\N}{\mathbb{N}}
\newcommand{\A}{\mathbb{A}}

\newcommand{\RZ}{\R \times \widehat{\Z}}

\newcommand{\RkD}{\mathcal{R}^{n,k}_Q(D)}

\newcommand{\HkQD}{\mathcal{H}^{n,k}_Q(D)}

\newcommand{\U}{\mathcal{U}}

\newcommand{\Gr}[1]{\mathrm{Gr}_{n,#1}}
\newcommand{\Gror}[1]{\mathrm{Gr}^+_{n,#1}}
\newcommand{\Grk}{\mathrm{Gr}_{n,k}}

\newcommand{\disc}{\mathrm{disc}}

\newcommand{\Lie}{\mathrm{Lie}}

\newcommand{\spn}{\mathrm{span}}

\newcommand{\Pack}{\mathrm{P}}

\newcommand{\ord}{\mathrm{ord}}
\newcommand{\Grids}{\mathcal{Y}}
\newcommand{\Gridsor}{\mathcal{Y}^+}

\newcommand{\g}{\mathfrak{g}}
\newcommand{\h}{\mathfrak{h}}

\newcommand\rquot[2]{
  \mathchoice
  {
    \text{\raise0.5ex\hbox{$#1$}\big/\lower0.5ex\hbox{$#2$}}%
  }
  {
    #1\,/\,#2
  }
  {
    #1\,/\,#2
  }
  {
    #1\,/\,#2
  }
}

\newcommand\lquot[2]{
  \mathchoice
  {
    \text{\lower0.5ex\hbox{$#1$}\big\backslash\raise0.5ex\hbox{$#2$}}%
  }
  {
    #1\,\backslash\,#2
  }
  {
    #1\,\backslash\,#2
  }
  {
    #1\,\backslash\,#2
  }
}

\newcommand\lrquot[3]{
  \mathchoice
  {
    \text{\lower0.5ex\hbox{$#1$}\big\backslash\raise0.5ex\hbox{$#2$\!}\big/
      \lower0.5ex\hbox{\!\!$#3$}}%
  }
  {
    #1\,\backslash\,#2\,/\,#3
  }
  {
    #1\,\backslash\,#2\,/\,#3
  }
  {
    #1\,\backslash\,#2\,/\,#3
  }
}

%% file: Sections/intro.tex

\section{Introduction}

In this paper, we study the joint distribution of rational subspaces of a fixed discriminant (also called height by some authors) and of two naturally associated lattices: the integer lattice in the subspace and in its orthogonal complement together with some natural refinements.

Let $Q$ be a positive definite integral quadratic form on $\Q^n$ and let $L \in \Grk(\Q)$ be a rational $k$-dimensional subspace.
Here, $\Grk$ is the projective variety of $k$-dimensional subspaces of the $n$-dimensional linear space.
The discriminant $\disc_Q(L)$ of $L$ with respect to $Q$ is the discriminant of the restriction of $Q$ to the integer lattice $L(\Z) = L\cap \Z^n$. In a formula,
\begin{align*}
\disc_Q(L) = \det \begin{pmatrix}
\langle v_1,v_1\rangle_Q & \cdots & \langle v_1,v_k\rangle_Q \\
\vdots & & \vdots \\
\langle v_k,v_1\rangle_Q & \cdots & \langle v_k,v_k\rangle_Q
\end{pmatrix}
\end{align*}
where $\langle \cdot,\rangle_Q$ is the bilinear form induced by $Q$ and $v_1,\ldots,v_k$ is a basis of $L(\Z)$.
We consider the finite set
\begin{align*}
\HkQD := \{ L \in \Grk(\Q) : \disc_Q(L) = D \}.
\end{align*}

We attach to any $L \in \Grk(\Q)$ the restriction of $Q$ to $L(\Z)$ represented in a basis. This is an integral quadratic form in $k$-variables which is well-defined up to a change of basis i.e.~(in the language of quadratic forms) up to equivalence.
In particular, it defines a well-defined point -- also called the \emph{shape} of $L(\Z)$ --
\begin{align*}
[L(\Z)] \in \shapesp{k}
\end{align*}
where $\shapesp{k}$ is the space of positive definite real quadratic forms on $\R^n$ up to similarity (i.e.~up to equivalence and positive multiples).
We may identify $\shapesp{k}$ as
\begin{align*}
\shapesp{k} \simeq \lrquot{\mathrm{O}_k(\R)}{\PGL_k(\R)}{\PGL_k(\Z)}
\end{align*}
which in particular equips $\shapesp{k}$ with a probability measure $m_{\shapesp{k}}$ arising from the Haar measures of the groups on the right. We will simply call $m_{\shapesp{k}}$ the Haar probability measure on $\shapesp{k}$.

Analogously, one may define the point $[L^\perp(\Z)]\in \shapesp{n-k}$ where $L^\perp$ is the orthogonal complement of $L$ with respect to $Q$.
Overall, we obtain a triple of points $(L,[L(\Z)],[L^\perp(\Z)])$. The goal of this work is to study the distribution of these points in $\Grk(\R) \times \shapesp{k} \times \shapesp{n-k}$ as $\disc_Q(L)$ grows.
In what follows, $\Grk(\R)$ is given the unique $\SO_Q(\R)$-invariant probability measure $m_{\Grk(\R)}$.

\begin{Conj}\label{conj:ideal equidistribution result}
Let $k,n \in \N$ be integers such that $k \geq 2$ and $n-k \geq 2$.
Then the sets
\begin{align*}
\{(L,[L(\Z)],[L^\perp(\Z)]): L \in \HkQD\}
\end{align*}
equidistribute\footnote{Implicitly, we mean with respect to the product 'Haar' measure, i.e, the product measure $m_{\Grk(\R)} \otimes m_{\shapesp{k}} \otimes m_{\shapesp{n-k}}$.} 
in $\Grk(\R) \times \shapesp{k} \times \shapesp{n-k}$ as $D \to \infty$ along $D \in \N$ satisfying $\HkQD\neq \emptyset$.
\end{Conj}

\begin{remark}\label{rem:AES}
There exists an analogous conjecture for $k=1$, $n-k \geq 2$ where one only considers the pairs $(L,[L^\perp(\Z)])$ (and similarly for $n-k=1$, $k \geq 2$).
This has been studied extensively by the first named author with Einsiedler and Shapira in \cite{AES-dim3,AES-higherdim} where the conjecture is settled for $n \geq 6$ (i.e.~$n-k \geq 5$), for $n = 4,5$ under a weak congruence condition and for $n=3$ under a stronger congruence condition on $D$.
We remark that, as it is written, \cite{AES-dim3,AES-higherdim} treat only the case where $Q$ is the sum of squares (that we will sometimes call the standard form), but the arguments carry over without major difficulties.
Using effective methods from homogeneous dynamics, Einsiedler, R\"uhr and Wirth \cite{Shapes-effective} proved an effective version of the conjecture when $n=4,5$ removing in particular all congruence conditions.
The case $n=3$ relies on a deep classification theorem for joinings by Einsiedler and Lindenstrauss \cite{joiningsfinal}; effective versions of that theorem are well out of reach of current methods from homogeneous dynamics.
Assuming the Generalized Riemann Hypothesis, Blomer and Brumley \cite{BlomerBrumley} have recently removed the congruence condition in \cite{AES-dim3}.
\end{remark}

\begin{remark}\label{rem:2in4}
The case $k=2$ and $n-k=2$ of Conjecture \ref{conj:ideal equidistribution result} has been settled in \cite{2in4} by the first and the last named author together with Einsiedler under a (relatively strong) congruence condition when $Q$ is the sum of four squares.
The result in the paper is in fact stronger as it considers two additional shapes that one can naturally associate to $L$ essentially thanks to the local isomorphism between $\SO_4(\R)$ and $\SO_3(\R) \times \SO_3(\R)$. The arguments carry over without major difficulties to consider norm forms on quaternion algebras (equivalently, the forms $Q$ for which $\disc(Q)$ is a square in $\Q^\times$). In \cite{2in4General}, the first and last named author will extend the results of \cite{2in4} to treat arbitrary quadratic forms.
\end{remark}

In this article, we prove Conjecture~\ref{conj:ideal equidistribution result} in the remaining cases, partially under congruence conditions.
For integers $D,\ell$ we write $\kfree{D}{\ell}$ for the $\ell$-power free part of $D$ i.e.~the largest divisor $d$ of $D$ with $a^\ell \nmid d$ for any $a>1$.

\begin{Thm}[Equidistribution of subspaces and shapes]\label{thm:jointwithshapes}
Let $2 \leq k \leq n$ be integers with $k \leq n-k$ and $n-k >3$, and let $p$ be an odd prime with $p \nmid\disc(Q)$. 
Let $D_i\in \N$ be a sequence of integers with $\kfree{D_i}{k}\to \infty$ and $\mathcal{H}_{Q}^{n,k}(D_i)\neq \emptyset$ for every $i$.
Then the sets
\begin{align*}
\{(L,[L(\Z)],[L^\perp(\Z)]): L \in \mathcal{H}_{Q}^{n,k}(D_i)\}
\end{align*}
equidistribute in $\Grk(\R) \times \shapesp{k} \times \shapesp{n-k}$ as $i \to \infty$ assuming the following conditions:
\begin{itemize}
\item $p \nmid D_i$ if $k \in \{3,4\}$.
\item $-D_i \mod p$ is a square in $\mathbb{F}_p^\times$ if $k=2$.
\end{itemize}
Moreover, the analogous statement holds when the roles of $k$ and $n-k$ are reversed.
\end{Thm}

\begin{remark}
Maass \cite{Maass56,Maass59} in the 60's and Schmidt \cite{Schmidt-shapes} in the 90's have considered problems of this kind. They prove that the set of pairs $(L,[L(\Z)])$ equidistributes in $\Grk(\R) \times \shapesp{k}$ where $L\in \Grk(\Q)$ varies over the rational subspaces with discriminant \emph{at most} $D$. 
In this averaged setup, Horesh and Karasik \cite{HoreshKarasik} recently verified Conjecture \ref{conj:ideal equidistribution result}.
Indeed, their version is polynomially effective in $D$.
\end{remark}

\begin{remark}[Congruence conditions]\label{rem:congcond}
As in the previous works referenced in Remarks \ref{rem:AES} and \ref{rem:2in4}, our proof is of dynamical nature and follows from an equidistribution result for certain orbits in an adelic homogeneous space.
The congruence conditions at the prime $p$ assert roughly speaking that one can use non-trivial dynamics at one fixed place for all $D$.
The acting groups we consider here are (variations of) the $\Q_p$-points of
\begin{align*}
\mathbf{H}_L= \{g \in \SO_Q: g.L \subset L \}^\circ
\end{align*}
for $L \in \Grk(\Q)$.
In particular, the cases $k=2$ and $k>2$ are very different from a dynamical viewpoint:
\begin{itemize}
\item For $k>2$ the group $\mathbf{H}_L$ is semisimple. The knowledge about measures on homogeneous spaces invariant under unipotents is vast -- see Ratner's seminal works \cite{ratner91-measure,ratner-p-adic}. In our situation, we use an $S$-arithmetic version of a theorem by Mozes and Shah \cite{mozesshah}, proven by Gorodnik and Oh \cite{gorodnikoh}, which describes weak${}^\ast$-limits of measures with invariance under a semisimple group. 
Roughly speaking, the theorem implies that any sequence of orbits under a semisimple subgroup is either equidistributed or sits (up to a small shift) inside an orbit of a larger subgroup.
The flexibility that this method provides allows us to in fact prove a significantly stronger result; see Theorem~\ref{thm:equidistrmoduli} below.
\item For $k=2$ and $n-k \geq 3$, the group $\mathbf{H}_L$ is reductive. Thus, one can apply the results mentioned in the previous bullet point only to the commutator subgroup of $\mathbf{H}_L$ which is non-maximal and has intermediate subgroups.
\end{itemize}
One of the novelties of this article is a treatment of this reductive case where we use additional invariance under the center to rule out intermediate subgroups 'on average' (see \S\ref{sec:firstfactor}). 
Here as well as for the second component of the triples in Theorem~\ref{thm:jointwithshapes} we need equidistribution of certain adelic torus orbits; this is a generalized version of a theorem of Duke \cite{duke88} building on a breakthrough of Iwaniec \cite{iwaniecforduke} -- see for example \cite{Dukeforcubic,harcosmichelII,W-linnik}.
Furthermore, to prove simultaneous equidistribution of the tuples in Theorem~\ref{thm:jointwithshapes}
we apply a new simple disjointness trick -- see the following remark.
\end{remark}

\begin{remark}[Disjointness]\label{rem:disjointness}
In upcoming work, the first and last named author prove together with Einsiedler, Luethi and Michel \cite{EffDisjoint} an effective version of Conjecture \ref{conj:ideal equidistribution result} when $k\neq 2$.
This removes in particular the congruence conditions. The technique consists of a method to 'bootstrap' effective equidistribution in the individual factors to simultaneous effective equidistribution (in some situations).

In the current article, we use an ineffective analogue of this to prove Theorem~\ref{thm:jointwithshapes}, namely, the very well-known fact that mixing systems are disjoint from trivial systems (see also Lemma~\ref{lem:disjointnessfromtrivial}).
This simple trick has (to our knowledge) not yet appeared in the literature in a similar context.
It is particularly useful when $k=2$ and $n-k\geq 3$ in which case we cannot rely solely on methods from unipotent dynamics (see Remark~\ref{rem:congcond}).
\end{remark}

\begin{remark}[On the power assumption]\label{rem:powerfree}
The assumption in Theorem~\ref{thm:jointwithshapes} toward the power free part of the discriminants should only be considered a simplifying assumption. Its purpose is automatically to rule out situations where for most subspaces $L \in \HkQD$ the quadratic form $Q|_{L(\Z)}$ (or $Q|_{L^\perp(\Z)}$) is highly imprimitive (i.e. a multiple of a quadratic form of very small discriminant).
We expect that such discriminants do not exist regardless of their factorization.
A conjecture in this spirit is phrased in Appendix~\ref{sec:appendixprimitive}.
Moreover, Schmidt's work \cite{Schmidt-count} suggests that $|\HkQD| = D^{\frac{n}{2}-1+o(1)}$ in which case one could remove the assumption $\kfree{D_i}{k}\to \infty$ in Theorem~\ref{thm:jointwithshapes}.
\end{remark}

\subsection{A strengthening}\label{sec:modulispacetheorem}

In the following we present a strengthening of Conjecture~\ref{conj:ideal equidistribution result} inspired by the notion of grids introduced in \cite{AES-higherdim} and by Bersudsky's construction of a moduli space \cite{BersudskyModuli} which refines the results of \cite{AES-higherdim}.

Consider the set of pairs $(L,\Lambda)$ where $L \subset \R^n$ is a $k$-dimensional subspace and where $\Lambda \subset \R^n$ is a lattice of full rank with the property that $L \cap \Lambda$ is a lattice in $L$ ($L$ is \emph{$\Lambda$-rational}). We define an equivalence relation on these pairs by setting $(L,\Lambda)\sim (L',\Lambda')$ whenever the following conditions are satisfied:
\begin{enumerate}
\item $L=L'$,
\item There exists $g \in \GL_n(\R)$ with $\det(g) > 0$ such that $g$ acts on $L$ and $L^{\perp}$ as scalar multiplication and $g\Lambda = \Lambda'$.
\end{enumerate}
We write $[L,\Lambda]$ for the class of $(L,\Lambda)$; elements of such a class are said to be \emph{homothetic along $L$ or $L$-homothetic} to $(L,\Lambda)$. 
We refer to the set $\Grids$ of such equivalence classes as the \emph{moduli space of basis extensions}.
Indeed, one can think of a lattice $\Lambda$ such that $L\cap \Lambda$ is a lattice as one choice of complementing the lattice $L \cap \Lambda$ into a basis of $\R^n$.
The equivalence relation is not very transparent in this viewpoint, see Section~\ref{sec:moduli} for further discussion.

The moduli space $\Grids$ is designed to incorporate subspaces as well as both shapes. Clearly, we have a well-defined map
\begin{align}\label{eq:modulitoGr}
[L,\Lambda] \in \Grids	\mapsto L \in \Grk(\R)
\end{align}
The restriction of $Q$ to $L \cap \Lambda$ yields a well-defined element of $\shapesp{k}$.
Similarly, one may check that $L^\perp$ intersects the dual lattice $\Lambda^\#$ in a lattice, the second shape is given by the restriction of $Q$ to $L^\perp \cap \Lambda^\#$.

We note that there is a natural identification of $\Grids$ with a double quotient of a Lie group (cf.~Lemma~\ref{lem:ident for grids}) so that we may again speak of the 'Haar measure' on $\Grids$.

\begin{Conj}\label{conj:equidistrmoduli}
Let $k,n \in \N$ be integers such that $k\geq 3$ and $n-k \geq 3$. Then the sets
\begin{align*}
\{([L,\Z^n]: L \in \HkQD\} \subset \Grids
\end{align*}
equidistribute with respect to the Haar measure as $D \to \infty$ along $D \in \N$ satisfying $\HkQD\neq \emptyset$.
\end{Conj}

\begin{remark}[From Conjecture~\ref{conj:equidistrmoduli} to Conjecture~\ref{conj:ideal equidistribution result}]\label{rem:comparisonconj}
When $Q$ is unimodular (i.e.~$\disc(Q) = 1$), Conjecture~\ref{conj:equidistrmoduli} implies Conjecture~\ref{conj:ideal equidistribution result}. Otherwise, Conjecture~\ref{conj:equidistrmoduli} implies equidistribution of the triples $(L,[L(\Z)],[L^\perp \cap (\Z^n)^\#])$ where $(\Z^n)^\#$ is the dual lattice to $\Z^n$ under the quadratic form $Q$:
\begin{align*}
(\Z^n)^\# = \{ x \in \Q^n: \langle x,y \rangle_Q \in \Z\text{ for all }y \in \Z^n \}.
\end{align*}
This is not significantly different as the lattice $L^\perp \cap (\Z^n)^\#$ contains $L^\perp \cap \Z^n$ with index at most $\disc(Q)$;
it is nevertheless insufficient to deduce Conjecture~\ref{conj:ideal equidistribution result}. In \S\ref{sec:moduli} we introduce tuples $[L,\Lambda_L]$ which satisfy an analogue of Conjecture~\ref{conj:equidistrmoduli}; this adapted conjecture implies Conjecture~\ref{conj:ideal equidistribution result}
\end{remark}

We prove the following towards Conjecture~\ref{conj:equidistrmoduli}.

\begin{Thm}\label{thm:equidistrmoduli}
Let $k,n$ be integers with $3 \leq k \leq n-k$ and let $p$ be an odd prime with $p \nmid\disc(Q)$. 
Let $D_i\in \N$ be a sequence of integers with $\kfree{D_i}{k}\to \infty$ and $\mathcal{H}_{Q}^{n,k}(D_i)\neq \emptyset$ for every $i$.
Then the sets
\begin{align*}
\{([L,\Z^n]: L \in \mathcal{H}_{Q}^{n,k}(D_i)\}
\end{align*}
equidistribute in $\Grids$ as $i \to \infty$ assuming in addition that $p \nmid D_i$ if $k \in \{3,4\}$.
\end{Thm}

\begin{remark}\label{rem:gridsIlya}
As mentioned in Remark~\ref{rem:congcond}, the assumption $k\geq 3$ and $n-k \geq 3$ asserts that the acting group underlying the problem is semisimple.
There are instances where one could overcome this obstacle: Khayutin \cite{Khayutin-kugasato} proves equidistribution of grids 
when $(k,n) = (1,3)$ as conjectured in \cite{AES-higherdim} using techniques from geometric invariant theory.
\end{remark}

\subsection{Further refinements and questions}
For an integral quadratic form $q$ in $k$ variables a \emph{primitive representation} of $q$ by $Q$ is a $\Z$-linear map $\iota: \Z^k \to \Z^n$ such that $Q(\iota(v)) = q(v)$ for all $v \in \Z^k$ and such that $\Q\iota(\Z^k) \cap \Q^n = \iota(\Z^k)$.
One can identify primitive representations of $q$ with subspaces $L \in\Grk(\Q)$ such that $Q|_{L(\Z)}$ is equivalent to $q$.
Given this definition, one could ask about the distribution of the pairs
\begin{align}\label{eq:EVpairs}
\{(L,[L^\perp(\Z)]):L \in \Grk(\Q) \text{ and } Q|_{L(\Z)} \text{ is equivalent to }q\}
\end{align}
inside $\Grk(\R)\times \shapesp{n-k}$ when $\disc(q) \to \infty$. The condition $\disc(q) \to \infty$ here is not sufficient; for instance, when $q$ represents $1$ and $Q$ represents $1$ only on, say, $\pm v \in \Z^n$ then any primitive representation of $q$ by $Q$ must contain $\pm v$. However, the subspaces in $\Grk(\R)$ containing $\pm v$ form a Zariski closed subset.
Assuming that the minimal value represented by $q$ goes to infinity, the above question is very strongly related to results of Ellenberg and Venkatesh \cite{localglobalEV} as are indeed our techniques in this article.
In principle, these techniques should apply to show that under congruence conditions as in Theorems~\ref{thm:jointwithshapes} and~\ref{thm:equidistrmoduli} the pairs in \eqref{eq:EVpairs} are equidistributed when $q_i$ is a sequence of quadratic forms primitively representable by $Q$ whose minimal values tend to infinity.

As alluded to in Remark~\ref{rem:gridsIlya} it would be interesting to know if Khayutin's technique applies to show the analogue of Theorem~\ref{thm:equidistrmoduli} when, say, $(k,n) = (2,5),(2,4)$.
The two cases are from a dynamical perspective quite different as noted in Remark~\ref{rem:congcond}.

Furthermore, we note that this paper has various clear directions of possible generalization.
Most notably, this paper can be extended to indefinite forms.
Let $Q$ be an indefinite integral quadratic form on $\Q^n$ of signature $(r,s)$.
Here, we observe that $\SO_Q(\R)$ does not act transitively on $\Grk(\R)$. Indeed, the degenerate subspaces form a Zariski closed subset (the equation being $\disc(Q|_L) = 0$).
The complement is a disjoint union of finitely many open sets on which $\SO_Q(\R)$ acts transitively; for each tuple $(r',s')$ with $r'+s'= k$ and $r' \leq r,\ s' \leq s$ such an open set is given by the subspaces $L$ for which $Q|_L$ has signature $(r',s')$.
The analogue of the above conjectures and theorems can then be formulated by replacing $\Grk(\R)$ with one of these open sets.
The proofs generalize without major difficulties to this case; we refrain from doing so here for simplicity of the exposition.
Other directions of generalization include the number field case which is not addressed in any of the works prior to this article and is hence interesting in other dimensions as well.
\bigskip

\textbf{Acknowledgments}
The authors would like to thank Michael Bersudsky, Manfred Einsiedler and Manuel Luethi for useful discussions. 
We are also thankful towards the anonymous referee who made various valuable suggestions towards improving the exposition.

\subsection{Organization of the paper}\label{sec:organization}
This article consists of two parts.
In Part $1$ -- the 'dynamical' part -- we establish the necessary results concerning equidistribution of certain adelic orbits.
It is structured as follows:
\begin{itemize}
\item In \S\ref{SEC:stabilisers}, we prove various results concerning stabilizer subgroups of subspaces.
\item In \S\ref{sec:dimatleast4}, we prove the homogeneous analogue of Theorem~\ref{thm:equidistrmoduli}.
The key ingredient of our proof is an $S$-arithmetic extension of a theorem of Mozes and Shah \cite{mozesshah} proven by Gorodonik and Oh \cite{gorodnikoh}. The arguments used in this section only work when the dimension and codimension (i.e. $k$ and $n-k$) are at least 3.
\item[$\bullet$] In \S\ref{SEC: two_dim_thm}, we prove the homogeneous analogue of Theorem~\ref{thm:jointwithshapes} for two dimensional subspaces (i.e.~for $k=2$). 
Contrary to the case of dimension and codimension at least 3, the groups whose dynamics we use are not semisimple (see Remark~\ref{rem:congcond}). In particular, the theorem of Gorodonik and Oh \cite{gorodnikoh} is not sufficient and more subtle arguments, relying on Duke's Theorem \cite{duke88} and the trick mentioned in Remark~\ref{rem:disjointness}, are required.
\end{itemize}
In Part $2$, we deduce Theorems~\ref{thm:jointwithshapes} and \ref{thm:equidistrmoduli} from the homogeneous dynamics results proven in \S\ref{sec:dimatleast4} ($k>2$) and \S\ref{SEC: two_dim_thm} ($k=2$) of the first part. More precisely, it is structured as follows:
\begin{enumerate}
\item[$\bullet$] In \S\ref{SEC:quadratic_forms}, we prove that the discriminant of the orthogonal complement of a subspace is equal to the discriminant of the subspace up to an essentially negligible factor.
\item[$\bullet$] In \S\ref{sec:moduli},
we study the moduli space of base extensions and show that it surjects onto $\Grk(\R) \times \shapesp{k} \times \shapesp{n-k}$.
From this, we prove that a slight strengthening of Theorem~\ref{thm:equidistrmoduli} implies Theorem~\ref{thm:jointwithshapes}.
In these considerations, it is useful to consider subspaces together with an orientation.
\item[$\bullet$] In \S\ref{sec:proof arithmetic}, we finally establish Theorems~\ref{thm:jointwithshapes} and \ref{thm:equidistrmoduli}.
The technique here is by now standard -- we interpret the sets in Theorem~\ref{thm:equidistrmoduli} as projections of the adelic orbits in Part $1$ (or a slight adaptation thereof).
\end{enumerate}
In the appendix, we establish various complementary facts.
\begin{enumerate}
\item[$\bullet$] In Appendix~\ref{sec:appendix}, we discuss non-emptiness conditions for the set $\HkQD$ when the quadratic form $Q$ is the sum of squares. In particular, we prove that $\HkQD \neq \emptyset$ for all $n \geq 5$. The techniques here are completely elementary and we do not provide any counting results.
\item[$\bullet$]
In Appendix~\ref{sec:appendixprimitive}, we prove various facts complementing the discussion in \S\ref{SEC:quadratic_forms}. For instance, we prove that if $L \in \Gr{k}(\Q)$ is a subspace where $k < n-k$ then the quadratic form on the orthogonal complement $Q|_{L^\perp(\Z)}$ is primitive up to negligible factors.
\end{enumerate}

\subsection{Notation}
Let $V_\Q$ be the set of places of $\Q$ and denote by $\Q_v$ for any $v \in V_\Q$ the completion at $v$.
Given a subset $S \subset V_\Q$ we define the ring $\Q_S$ to be the restricted direct product of $\Qp$ for $p \in S$ with respect to the subgroups $\Zp$ for $p \in S\setminus\{\infty\}$.  
Moreover, we set $\Z^S := \Z[\frac{1}{p} : p \in S\setminus \{\infty\} \}$.
When $S = V_\Q$ we denote $\Q_S$ by $\A$ and call it the ring of \emph{adeles}.
When instead $S= V_\Q \setminus \{\infty\}$ we denote $\Q_S$ by $\A_f$ and call it the ring of \emph{finite adeles}. Finally, we let $\hat{\Z} = \prod_{p \in V_\Q \setminus \{\infty\}} \Zp$.

Let $\G< \SL_N$ be a connected algebraic group defined over $\Q$.
We identify $\G(\Z^S) = \G(\Q_S) \cap \SL_N(\Z^S)$ with its diagonally embedded copy in $\G(\Q_S)$.
If $\G$ has no non-trivial $\Q$-characters (for instance when the radical of $\G$ is unipotent), the Borel-Harish-Chandra Theorem (see  \cite[Thm.~5.5]{platonov}) yields that $\G(\Z^S)$ is a lattice in $\G(\Q_S)$ whenever $\infty \in S$. In particular, the quotient $\G(\Q_S)/\G(\Z^S)$ is a finite volume homogeneous space.
For $g \in \G(\Q_S)$ and $v \in S$, $g_v$ denotes the $v$-adic component of $g$.

Whenever $\G$ is semisimple, we denote by $\G(\Q_S)^+$ the image of the simply connected cover in $\G(\Q_S)$ (somewhat informally, this can be thought of as the part of $\G(\Q_S)$ which is generated by unipotents).

\subsubsection{Quadratic forms}\label{sec:quadratic forms}
In this whole article, $(V,Q)$ is a fixed non-degenerate quadratic space over $\Q$ of dimension $n$.
The induced bilinear form is denoted by $\langle \cdot,\cdot\rangle_Q$.
We assume throughout that $(V,Q)$ is positive definite.
We also identify $V$ with $\Q^n$ and suppose that $\langle \cdot,\cdot\rangle_Q$ takes integral values on $\Z^n \times \Z^n$ in which case we say that $Q$ is \emph{integral}.
Equivalently, the matrix representation $M_Q$ in the standard basis of $\Z^n$ has integral entries.

We denote by $\Orth_Q$ resp.~$\SO_Q$ the orthogonal resp.~special orthogonal group for $Q$.
Recall that $\SO_Q$ is abelian if $\dim(V) = 2$ and semisimple otherwise.
We denote by $\Spin_Q$ the spin group for $Q$ which is the simply connected cover of $\SO_Q$ if $\dim(V) >2$.
Explicitly, the spin group may be constructed from the Clifford algebra of $Q$. 
We remark that this article contains certain technicalities that will use the Clifford algebra --
we refer to \cite{knus} for a thorough discussion.
The spin group comes with an isogeny of $\Q$-groups $\isog: \Spin_Q \to \SO_Q$ which satisfies that for any field $K$ of characteristic zero we have an exact sequence (cf.~\cite[p.~64]{knus})
\[ \Spin_Q(K) \rightarrow \SO_Q(K) \rightarrow K^{\times}/ (K^{\times})^2. \]
where the second homomorphism is given by the spinor norm.
The isogeny $\isog$ induces an integral structure on $\Spin_Q$. For instance, $\Spin_Q(\Z)$ consists of elements $g \in \Spin_Q(\Q)$ for which $\isog(g) \in \SO_Q(\Z)$.
To simplify notation, we will write $g.v$ for the action of $\Spin_Q$ on a vector in $n$-dimensional linear space. Here, the action is naturally induced by the isogeny $\isog$ (and the standard representation of $\SO_Q$).

Furthermore, we let $\Grk$ denote the Grassmannian of $k$-dimensional subspaces of~$V$. 
Note that this is a homogeneous variety for $\SO_Q$ and (through the isogeny $\isog$) also for $\Spin_Q$.
If we assume that $Q$ is positive definite (as we always do), the action of $\SO_Q(\R)$ on $\Grk(\R)$ is transitive. Furthermore, in this case the spinor norm on $\SO_Q(\R)$ takes only positive values so that $\Spin_Q(\R)$ surjects onto $\SO_Q(\R)$ and in particular also acts transitively.


We denote the standard positive definite form (i.e.~the sum of $n$ squares) by $Q_0$ and write $\SO_n$ for its special orthogonal group. As $Q_0$ and $Q$ have the same signature, there exists $\gQ \in \GL_n(\R)$ with $\det(\gQ) > 0$ such that $\gQ^t \gQ = M_Q$ or equivalently
\begin{align}\label{eq:defetaQ}
Q_0(\gQ x) = Q(x)
\end{align}
holds for all $x \in \R^n$ (similarly for the induced bilinear forms). 
In particular, $\gQ$ maps pairs of vectors in $V$ which are orthogonal with respect to $Q$ onto pairs of vectors which are orthogonal with respect to $Q_0$.
Also, $\gQ^{-1} \SO_n(\R) \gQ = \SO_Q(\R)$.

\subsubsection{Quadratic forms on sublattices and discriminants}\label{sec:notation-induced qf}

For any finitely generated $\Z$-lattice $\Gamma < \Q^n$ (of arbitrary rank) the restriction of $Q$ to $\Gamma$ induces a quadratic form.
We denote by $q_\Gamma$ the representation of this form in a choice of basis of $\Gamma$.
Hence, $q_\Gamma$ is well-defined up to equivalence (and not proper equivalence) of quadratic forms (i.e.~up to change of basis).

If $\Gamma < \Z^n$, $q_\Gamma$ is an integral quadratic form and we denote by $\gcd(q_\Gamma)$ the greatest common divisor of its coefficients (which is independent of the choice of basis). Note that $\gcd(q_\Gamma)$ is sometimes also referred to as the content of $q_\Gamma$.
We write $\tilde{q}_\Gamma = \frac{1}{\gcd(q_\Gamma)}q_\Gamma$ for the primitive multiple of $q_\Gamma$.
If $L \subset \Q^n$ is a subspace, we sometimes write $q_L$ instead of $q_{L(\Z)}$ for simplicity.

The discriminant $\disc_Q(\Gamma)$ of a finitely generated $\Z$-lattice $\Gamma< \Q^n$ is the discriminant of $q_\Gamma$. 
As at the beginning of the introduction, we write $\disc_Q(L)$ instead of $\disc_Q(L(\Z))$ for any subspace $L \subset \Q^n$.
Given a prime $p$ we also define
\begin{align}\label{eq:def localdisc}
\disc_{p,Q}(L) =  \disc(Q|_{L(\Z_p)}) \in \Z_p/(\Z_p^\times)^2 
\end{align}
where $L(\Z_p) = L(\Q_p)\cap \Z_p^n$. We have the following useful identity
\begin{align}\label{eq:disc loctoglob}
\disc_Q(L) = \prod_{p} p^{\nu_p(\disc_{p,Q}(L))} 
\end{align}
where the product is taken over all primes $p$ and $\nu_p$ denotes the standard $p$-adic valuation.
Note that only primes dividing the discriminant contribute non-trivially.

\subsubsection{Choice of a reference subspace}
We fix an integer $k\leq n$ for which we always assume that one of the following holds:
\begin{itemize}
\item $k \geq 3$ and $n-k \geq 3$,
\item $k = 2$ and $n-k \geq 3$, or
\item $k \geq 3$ and $n-k=2$.
\end{itemize}
Let $L_0 \subset V$ be given by
\begin{align}\label{eq:refsubspace}
L_0 = \Q^{k} \times \{(0,\ldots,0)\} \subset V.
\end{align}

We adapt the choice of $\gQ$ to this reference subspace $L_0$ and suppose that the first $k$ column vectors in $\gQ^{-1}$ are an orthonormal basis of $L_0$.
This choice asserts that $\gQ$ maps $L_0(\R)$ to $L_0(\R)$ and hence $L_0^\perp(\R)$ to $\{(0,\ldots,0)\} \times \R^{n-k}$,

\subsubsection{Ambient groups}\label{sec:groupsdef}
The following subgroups of $\SL_n$ will be useful throughout this work:
\begin{align*}
\SPnk &= \Big\{ \begin{pmatrix}
A & B \\
0 & D
\end{pmatrix} \in \SL_n : \det(A) = \det(D) = 1 \Big\}\\
\SDnk &= \Big\{ \begin{pmatrix}
A & 0 \\
0 & D
\end{pmatrix} \in \SL_n : \det(A) = \det(D) = 1 \Big\}.
\end{align*}
where $A$ is a $k \times k$-matrix, $D$ is an $(n-k)\times (n-k)$-matrix and $B$ is a $k \times (n-k)$-matrix.
We denote by $\pi_1$ resp.~$\pi_2$ the projection of $\SPnk$ onto the upper-left resp.~bottom-right block.
We also define the group
\begin{align*}
\SG = \Spin_Q \times \SPnk.
\end{align*}
By $\bSG$ we denote the Levi subgroup of $\SG$ with $B =0$ i.e.
\begin{align*}
\bSG &= \Spin_Q \times \SDnk \simeq \Spin_Q \times \SL_k \times \SL_{n-k}.
\end{align*}

\begin{remark}\label{rem:maximality}
Concerning the aforementioned groups we will need two well known facts. Firstly, $\SDnk$ is a maximal subgroup of $\SPnk$ (meaning that there is no connected $\Q$-group $\mathbf{M}$ with $\SDnk \subsetneq \mathbf{M} \subsetneq \SPnk$) -- see for example \cite[Prop.~3.2]{emvforSld}. Secondly, for any quadratic form $q$ in $d$ variables $\SO_q$ is maximal in $\SL_d$ -- see for example \cite{LiebeckSeitzMax} for a modern discussion of maximal subgroups of the classical groups.
\end{remark}

\subsubsection{Landau notation}
In classical Landau notation, we write $f \asymp g$ for two positive functions if there exist constants $c,C>0$ with $cf \leq g\leq Cf$. If the constants depend on another quantity $a$, we sometimes write $f \asymp_a g$ to emphasize the dependence.

%% file: Sections/groups.tex
\section{Stabilizer groups}\label{SEC:stabilisers}

Recall that throughout the article $Q$ is a positive definite integral quadratic form on $V = \Q^n$.
In particular, any subspace of $\Q^n$ is non-degenerate with respect to $Q$.

\subsection{Stabilizers of subspaces}\label{sec:stabilizersintro}
For any subspace $L \subset \overline{\Q}^n$ we define the following groups:
\begin{enumerate}
\item[$\bullet$] $\StSpin_L<\Spin_Q$ is the identity component of the stabilizer group of $L$ in $\Spin_Q$ for the action of $\Spin_Q$ on $\Grk$.
\item[$\bullet$] $\StSO_L<\SO_Q$ is the identity component of the stabilizer group of $L$ in $\SO_Q$ for the action of $\SO_Q$ on $\Grk$.
\end{enumerate}
Note that we have an isogeny $\StSpin_L \to \StSO_L$.
Furthermore, the restriction to $L$ resp.~$L^\perp$ yields an isomorphism of $\Q$-groups
\begin{align}\label{eq:ConnectedComponentOfStabilizerAsProduct}
\StSO_L \to \SO_{Q|_L} \times \SO_{Q|_{L^\perp}}.
\end{align}
To see this, one needs to check that the image consists indeed of special orthogonal transformations. This follows from the fact that the determinant of the restrictions is a morphism with finite image and hence its kernel must be everything by connectedness.
In particular, we have the following cases:
\begin{itemize}
\item If $k \geq 3$ and $n-k \geq 3$, $\StSO_L$ (and hence also $\StSpin_L$) is semisimple.
\item If $k = 2$ and $n-k \geq 3$ (or $k\geq 3$ and $n-k =2$), $\StSO_L$ is reductive.
\item If $k=2$ and $n-k=2$ (which is not a case this paper covers), $\StSO_L$ is abelian.
\end{itemize}

\begin{remark}[Special Clifford groups and \eqref{eq:ConnectedComponentOfStabilizerAsProduct}]\label{rem:isogenytostabprod}
While it might seem appealing to suspect that $\StSpin_L$ is simply-connected, this is actually false.
The following vague and lengthy explanation is not needed in the sequel.
Denote by $\mathbf{M}$ the special Clifford group of $Q$ and similarly by $\mathbf{M}_1$ resp.~$\mathbf{M}_2$ the special Clifford groups of $Q|_L$ resp.~$Q|_{L^\perp}$ (for the duration of this remark) -- cf.~\cite{knus}. 
These are reductive groups whose center is a one-dimensional $\Q$-isotropic torus.
We identify $\mathbf{M}_1,\mathbf{M}_2$ as subgroups of $\mathbf{M}$ and write $\mathbf{C}$ for the center of $\mathbf{M}$ which is in fact equal to $\mathbf{M}_1 \cap \mathbf{M}_2$.
The natural map $\phi: \mathbf{M}_1 \times \mathbf{M}_2 \to \mathbf{M}$ has kernel $\{(x,y) \in \mathbf{C}\times \mathbf{C}: xy = 1\}$ so that
\begin{align*}
\rquot{\mathbf{M}_1 \times \mathbf{M}_2}{\{(x,y) \in \mathbf{C}\times \mathbf{C}: xy = 1\}} \simeq \{g \in \mathbf{M}: g \text{ preserves } L\}^\circ.
\end{align*}
Furthermore, we have the spinor norm which is a character $\chi: \mathbf{M} \to \G_m$ whose kernel is the spin group. Similarly, we have spinor norms $\chi_1,\chi_2$ for $\mathbf{M}_1$ resp.~$\mathbf{M}_2$ which are simply the restrictions of $\chi$.
The above yields that
\begin{align*}
\StSpin_L \simeq \rquot{\{(g_1,g_2) \in \mathbf{M}_1 \times \mathbf{M}_2: \chi(g_1)\chi(g_2) = 1\}}{\ker(\phi)}
\end{align*}
which is isogenous (but not isomorphic) to $\Spin_{Q|_L} \times \Spin_{Q|_{L^\perp}}$.
\end{remark}

The first result we prove states that the group $\StSpin_L$ totally determines the subspace $L$ (up to orthogonal complements). More precisely,

\begin{Prop}
\label{Proposition: Stabilizer Uniquely Determines Subspace}
Let $L_1, L_2 \leq V$ be non-degenerate\footnote{Recall that a non-trivial subspace $W \subset V$ is non-degenerate if $\disc(Q|_W)\neq 0$ or equivalently if there is no non-zero vector $w \in W$ so that $\langle w,w'\rangle = 0$ for all $w' \in W$. This notion is stable under extension of scalars.} subspaces. If $\StSpin_{L_1} = \StSpin_{L_2}$, then $L_1 = L_2$ or $L_1 = L_2^{\perp}$.
\end{Prop}

The proposition follows directly from the following simple lemma:

\begin{Lem}
Let $L \subset V$ be a non-degenerate subspace and let $W \subset V$ be a non-trivial non-degenerate subspace invariant under $\StSO_L$. Then $W \in \{L,L^\perp,V\}$.
\end{Lem}

\begin{proof}
We first observe the following: over $\bar{\Q}$, $\StSO_{L}$ acts transitively on the set of anisotropic lines in $L$ and in $L^\perp$.
Indeed, by Witt's theorem \cite[p.~20]{Cassels} the special orthogonal group in dimension at least $2$ acts transitively on vectors of the same quadratic value. In any two lines one can find vectors of the same quadratic value by taking roots.

Let $w \in W$ be anisotropic and write $w = w_1 + w_2$ for $w_1 \in L$ and $w_2 \in L^\perp$. As $w$ is anisotropic, one of $w_1$ or $w_2$ must also be anisotropic; we suppose that $w_1$ is anisotropic without loss of generality.
Let $h \in \StSO_{L}(\bar{\Q})$ be such that $hw_1 \neq w_1$ and $hw_2 = w_2$. Then
\begin{align*}
u := hw-w =  hw_1 -w_1 \in L \cap W.
\end{align*}
We claim that we can choose $h$ so that $u$ is anisotropic. Indeed, as $w_1$ is anisotropic its orthogonal complement in $L$ is non-degenerate (as $L$ is non-degenerate). We can thus choose $h$ to map $w_1$ to a vector orthogonal to it by the above variant of Witt's theorem. Then
\begin{align*}
Q(u) = Q(hw_1)+Q(w_1) = 2Q(w_1) \neq 0.
\end{align*}
Now note that $L \cap W$ is $\StSO_L$-invariant.
By a further application of the above variant of Witt's theorem and the fact that $L$ is spanned by anisotropic vectors ($L$ is non-degenerate), we obtain that $L \cap W = L$ or equivalently $L \subset W$.
Thus, we may write $W = L \oplus W'$ where $W'$ is an orthogonal complement to $L$ in $W$ and in particular contained in $L^\perp$.
The subspace $W'$ must be non-degenerate as $W$ and $L$ are and hence is trivial or contains anisotropic vectors.
If $W'$ is trivial, $W = L$ and we are done.
Otherwise, we apply the above variant of Witt's theorem and obtain that $W' = L^\perp$ and $W = V$.
\end{proof}

An analogous statement holds for the relationship between quadratic forms and their special stabilizer groups.

\begin{Prop}
\label{Proposition: Stabilizer Uniquely Determines Quadratic Form}
Let $Q_1, Q_2$ be rational quadratic forms on $V$. If $\SO_{Q_1} = \SO_{Q_2}$, then $Q_1 = r Q_2$ for some $r \in \Q$.
\end{Prop}

For a proof see \cite[Lemma 3.3]{AES-higherdim}.

\subsubsection{Maximality}
We now aim to prove that for any non-degenerate subspace $L$ the connected $\Q$-groups $\StSO_L$ and $\StSpin_L$ are maximal subgroups.
Here, maximal is meant among connected and proper subgroups (as it was in Remark~\ref{rem:maximality}).

\begin{Prop}
\label{StabilizerMaximallyConnectedInSOQ}
For any non-degenerate subspace $L \subset V$ the groups $\StSO_L$ and $\StSpin_L$ are maximal.
\end{Prop}

The result above is well-known and due to Dynkin, who classified the maximal subgroups of the classical groups in \cite{DynkinMaximal} (see also the work of Liebeck and Seitz such as \cite{LiebeckSeitzMax}).
We will give an elementary proof.

\begin{proof}
Note that it suffices to prove the statement for $\StSO_L$.
As $L$ is non-degenerate, we may choose an orthogonal basis of $V$ consisting of an orthogonal basis of $L$ and an orthogonal basis of $L^\perp$. 
Let
\[ M_Q = \begin{pmatrix}
M_1 & 0 \\
0 & M_4
\end{pmatrix} \ \text{with} \ M_1,M_4 \ \text{diagonal matrices} \]
be the matrix representation of $Q$ in this basis.
Computing the Lie algebras of $\SO_Q$ and $\StSO_L$ we obtain:
\[ \g := \Lie(\SO_Q) = \{ A \in \Mat(n) : A^TM_Q+M_QA = 0 \} \]
and 
 \[ \h := \Lie(\StSO_L) =  \Big\{ A \in \Mat(n) : A = \begin{pmatrix}
A_1 & 0 \\
0 & A_4
\end{pmatrix}
\ \text{and} \ A_i^{T}M_i + M_iA_i = 0, i = 1,4 \Big\}. \]

We may split $\g$ in a direct sum $\h \oplus \mathfrak{r}$ where $\mathfrak{r}$ is an invariant subspace under the adjoint action of $\StSO_L$ on $\g$.
Explicitly, we may set 
\[ \mathfrak{r} = 
\Big\{ \begin{pmatrix}
0 & A_2 \\
A_3 & 0 \\
\end{pmatrix} : 
A_2^TM_1 +M_4A_3 = 0 \Big\}. \]
We claim that the representation of $\StSO_L$ on $\mathfrak{r}$ is irreducible.
Note that we may as well show that the representation of $\SO_{Q|_L}  \times  \SO_{Q|_{L^{\perp}}}$ on $\Mat(k, n-k)$ given by
\[ ((\sigma_1, \sigma_2) , A) \mapsto \sigma_1 A \sigma_2^{-1}  \]
is irreducible.
Over $\bar{\Q}$ we may apply Lemma \ref{lem:irred rep of SO_k times SO_m} below from which this follows.

Now let $\mathbf{M}$ be a connected group containing $\StSO_L$ and let $\mathfrak{m}$ be its Lie algebra.
Note that $\mathfrak{m} \cap \mathfrak{r}$ is an invariant subspace under the adjoint action of $\StSO_L$ on $\mathfrak{r}$.
Since this representation is irreducible, $\mathfrak{m} \cap \mathfrak{r} = \{0\}$ or $\mathfrak{m} \cap \mathfrak{r} =\mathfrak{r}$. In the former case, we have that $\mathfrak{m} = \mathfrak{h}$ and in the latter $\mathfrak{m} = \mathfrak{g}$. It follows that $\StSO_L$ is maximal and the proof is complete.
\end{proof}

\begin{Lem}\label{lem:irred rep of SO_k times SO_m}
For any $k,m \geq 3$ the action of $\SO_k \times \SO_m$ on $\Mat(k,m)$ by right- resp.~left-multiplication is irreducible.
\end{Lem}

\begin{proof}
We write a very elementary proof for the sake of completeness. 
First, assume that $k,m \geq 3$.
Note that the standard representation of $\SO_k$ (resp.~$\SO_m$) is irreducible as\footnote{Note that whenever $k=2$ any isotropic vector is a fixed vector.}
$k \geq 3$ (resp.~$m \geq 3$).
It follows that the representation of $\SO_k \times \SO_m$ on the tensor product of the respective standard representations is also irreducible (see, for instance, \cite[Theorem 3.10.2]{TensorProduct_Irreducible}); the latter is isomorphic to the representation in the lemma.
\end{proof}

\subsection{The isotropy condition}

We establish here congruence conditions which imply isotropy of the stabilizer groups $\StSpin_L$.
Recall that a $\Q_p$-group $\G$ is \emph{strongly isotropic} if for every connected non-trivial normal subgroup $\mathbf{N} < \G$ defined over $\Qp$, the group $\mathbf{N}(\Qp)$ is not compact.
We say that a $\Q$-group $\G$ is \emph{strongly isotropic at a prime $p$} if $\G$ is strongly isotropic as a $\Q_p$-group.

\begin{Prop}
Let $(V',Q')$ be any non-degenerate quadratic space over $\Q_p$. Then $Q'$ is isotropic if and only if $\Spin_{Q'}$ is strongly isotropic.
\end{Prop}
\begin{proof}
If $Q'$ is isotropic, $V'$ contains an hyperbolic plane $H$ (see \cite[Chapter 2. Lemma 2.1]{Cassels}). 
Then $\Spin_{Q'}$ contains $\Spin_{Q'|_H}$ which is a split torus. Hence, $\Spin_{Q'}$ is isotropic. Conversely, if $Q'$ is anisotropic then $\Spin_{Q'}(\Q_p)$ is compact as the hypersurface $Q'(x) =1$ is compact.
This proves that $Q'$ is isotropic if and only if $\Spin_{Q'}$ is isotropic.
This is sufficient to prove the proposition if $\dim(V') =2$ (as the torus $\Spin_{Q'}$ is one-dimensional) and if $\dim(V') > 2$ is not equal to $4$ as $\Spin_{Q'}$ is absolutely almost simple in these cases.

Suppose that $\dim(V') = 4$. We freely use facts about Clifford algebras and spin groups from \cite{knus} (mostly Chapter $9$ therein).
Recall that $\Spin_{Q'}$ is equal to the norm one elements of the even Clifford algebra $\mathcal{C}^0$ of $Q'$. If the center $\mathcal{Z}$ of $\mathcal{C}^0$ is a field over $\Q_p$, $\mathcal{C}^0$ is a quaternion algebra over $\mathcal{Z}$ and $\Spin_{Q'}$ is simple. In this case, the proof works as in the case of $\dim(V') \neq 4$.

So suppose that the center is split which is equivalent to $\disc(Q')$ being a square in $\Q_p$.
Thus, there is a quaternion algebra $\mathcal{B}$ over $\Q_p$ such that $(V',Q')$ is similar to $(\mathcal{B},\mathrm{Nr})$ where $\mathrm{Nr}$ is the norm on $\mathcal{B}$.
Then $\Spin_{Q'} \simeq \SL_1(\mathcal{B}) \times \SL_1(\mathcal{B})$ which is a product of two $\Q_p$-simple groups. 
Note that $\mathcal{B}$ or $\SL_1(\mathcal{B})$ are isotropic if and only if $Q'$ is isotropic. This concludes the proof of the proposition.
\end{proof}

Via \eqref{eq:ConnectedComponentOfStabilizerAsProduct} we obtain the following.

\begin{Cor} \label{cor: strong-isotropicity stabilizer iff subspace}
Let $L \in \Grk(\Q)$ and $p$ be an odd prime. 
Then, $\StSpin_{L}$ is strongly isotropic at $p$ if and only if the quadratic spaces $(L,Q|_{L})$ and $(L^{\perp},Q|_{L^{\perp}})$ are isotropic over $\Q_p$.
\end{Cor}

Using standard arguments (as in \cite[Lemma 3.7]{AES-higherdim} for example) we may deduce the following explicit characterization of isotropy.

\begin{Prop}
\label{Proposition: Strong Isotropicity criterion for HL }
Let $L \in \Grk(\Q)$ be a rational subspace and let $p$ be an odd prime. Then, $\StSpin_L$ is strongly isotropic at $p$ if any of the following conditions hold:
\begin{itemize}
\item $k \geq 5$ and $n-k \geq 5$.
\item $3 \leq k < 5$, $n- k \geq 5$, and $p \nmid \disc_Q(L)$.
\item $k \geq 5$, $3 \leq n- k < 5$, and $p \nmid \disc_Q(L^{\perp})$.
\item $3 \leq k < 5$, $3 \leq n- k < 5$, $p \nmid \disc_Q(L)$, and $p \nmid \disc_Q(L^{\perp})$.
\item $k=2$, $n-k \geq 5$, and $-\disc_Q(L) \in (\mathbb{F}_p^\times)^2$ (i.e.~$-\disc_Q(L)$ is a non-zero square modulo $p$).
\item $k=2$, $3 \leq n-k < 5$, $p \nmid \disc_Q(L^{\perp})$, and $-\disc_Q(L)\in (\mathbb{F}_p^\times)^2$.
\item $k\geq 5$, $n-k = 2$, and $-\disc_Q(L^\perp)\in (\mathbb{F}_p^\times)^2$.
\item $3 \leq k < 5$, $n-k=2$, $p \nmid \disc_Q(L)$, and $-\disc_Q(L^{\perp})\in (\mathbb{F}_p^\times)^2$.
\end{itemize}
\end{Prop}

While the list is lengthy, let us note that half of it consists in interchanging the roles of $k$ and $n-k$ as well as $L$ and $L^\perp$.
Also, whenever $p \nmid \disc(Q)$ the conditions $p \nmid \disc_Q(L)$ and $p \nmid \disc_Q(L^\perp)$ are equivalent (see Proposition~\ref{Prop:Isomorphism between the quotients of piL(Rn) and piL perp (Rn)} and its corollary).
When $k=4$ or $n-k=4$ the above criteria are sufficient but not necessary. For example, the form $x_1^2+x_2^2+ x_3^2 + px_4^2$ is isotropic though its discriminant is divisible by $p$.

\subsection{Diagonal embeddings of stabilizer groups}\label{sec:diagemb}

In this section, we define a diagonally embedded copy $\StDiag_L < \Spin_Q\times \SPnk$ of the stabilizer group of any subspace $L  \in \Gr{k}(\Q)$.

With the arithmetic application in Part $2$ in mind, we must allow for any rational subspace a choice of a full rank $\Z$-lattice $\Lambda_L \subset \Q^n$ with
\begin{align*}
\Z^n \subset \Lambda_L\subset (\Z^n)^\# := \{v \in \Q^n: \langle v,w\rangle \in \Z \text{ for all } w \in \Z^n\}.
\end{align*}
If $Q$ is unimodular (i.e.~$\disc(Q) =1$) then $\Lambda_L = \Z^n=(\Z^n)^\#$.
We emphasize that for the arguments in the current Part $1$, this choice of intermediate lattice $\Lambda_L$ is inconsequential and the reader may safely assume $\Lambda_L = \Z^n$ at first.

Let $g_L \in \GL_n(\Q)$ be such that $g_L \Z^n = \Lambda_L$, its first $k$ columns are a basis of $L \cap \Lambda_L$ and such that $\det(g_L) >0$. 
In words, the columns of $g_L$ complement a basis of $L\cap \Lambda_L$ into an oriented basis of $\Lambda_L$.
We then have a well-defined morphism with finite kernel
\begin{align}\label{eq:defisogenystab}
\Psi_L: \StSpin_L \to \SPnk,\ h \mapsto g_L^{-1}\isog(h) g_L.
\end{align}
Note that the morphism depends on the choice of $\Lambda_L$, but we omit this dependency here to simplify notation.
It also depends on the choice of basis; a change of basis conjugates $\Psi_L$ by an element of $\SPnk(\Z)$.

One can restrict the action of an element of $\StSpin_L$ to $L$ and represent the so-obtained special orthogonal transformation in the basis contained in $g_L$.
This yields an epimorphism (as in \eqref{eq:ConnectedComponentOfStabilizerAsProduct})
\begin{align*}
\psi_{1,L}: \StSpin_L \to \SO_{q_{{ }_{L \cap \Lambda_L}}}.
\end{align*}
Explicitly, the epimorphism is given by
\begin{align*}
\psi_{1,L}: h \in \StSpin_L \mapsto \pi_1(g_L^{-1}\isog(h)g_L) = \pi_1 \circ \Psi_L(h) \in \SO_{q_{{ }_{L \cap \Lambda_L}}}.
\end{align*}

Similarly to the above, one can obtain an epimorphism $\StSpin_L \to \SO_{Q|_{L^\perp}}$. 
To explicit this, we would like to specify how to obtain a basis of $L^\perp \cap \Lambda_L^\#$ from $g_L$.
For this, observe first that the basis dual to the columns of $g_L$ is given by the columns of $M_Q^{-1} (g_L^{-1})^t$.
Note that the last $n-k$ columns of $M_Q^{-1} (g_L^{-1})^t$ are orthogonal to $L$ so they form a basis of $\Lambda_L^\# \cap L^\perp$.
We hence obtain an epimorphism
\begin{align*}
\psi_{2,L}: h \in \StSpin_L \mapsto \pi_2(g_L^t M_Q\isog(h)M_Q^{-1} (g_L^{-1})^t) \in \SO_{q_{L^\perp \cap \Lambda_L^\#}}.
\end{align*}
Note that
\begin{align*}
g_L^{t}M_Q \isog(h) M_Q^{-1}(g_L^{-1})^t
= g_L^t \isog(h^{-1})^t (g_L^{-1})^t = (g_L^{-1}\isog(h^{-1})g_L)^t
\end{align*}
which shows that
\begin{align*}
\psi_{2,L}(h) = \pi_2((g_L^{-1}\isog(h^{-1})g_L)^t) = \pi_2(g_L^{-1}\isog(h^{-1})g_L)^t
= \pi_2(\Psi_L(h^{-1}))^t.
\end{align*}

We define the group
\begin{align}\label{eq:def diagonal stab}
\StDiag_L 
= \{(h, \Psi_L(h)) : h \in \StSpin_L\} \subset \Spin_Q \times \SPnk = \SG.
\end{align}
By the definitions above, the morphism
\begin{align*}
\SG \to \bSG,\ (g_1,g_2) \mapsto (g_1,\pi_1(g_2),\pi_2(g_2^{-1})^t)
\end{align*}
induces a morphism
\begin{align*}
\StDiag_L \to \{ (h,\psi_{1,L}(h),\psi_{2,L}(h)): h \in \StSpin_L\}
=: \bStDiag_L \subset \bSG
\end{align*}
which is in fact an isogeny.

%% file: Sections/dynamical.tex
\section{The dynamical version of the theorem in codimension at least 3}\label{sec:dimatleast4}

As mentioned in the introduction, our aim is to translate the main theorems into a statement concerning weak${}^\ast$ limits of orbit measures on an adequate adelic homogeneous space. 
In this and the next section we shall establish these equidistribution theorems for orbit measures.
This section treats the case $k, n-k \geq 3$.

In the following we call a sequence of subspaces $L_i \in \Grk(\Q)$ \emph{admissible} if
\begin{enumerate}
\item \label{item: c1 admissibility} $\disc_Q(L_i) \to \infty$ as $i \to \infty$,
\item \label{item: c2 admissibility}$\disc(\tilde{q}_{L_i}) \to \infty$ as $i \to \infty$,
\item \label{item: c3 admissibility}$\disc(\tilde{q}_{L_i^\perp}) \to \infty$ as $i \to \infty$, and
\item \label{item: c4 admissibility}there exists a prime $p$ such that $\StSpin_{L_i}(\Q_p)$ is strongly isotropic for all $i$.
\end{enumerate}

This section establishes the following theorem.
Conjecturally, an analogous version should hold when $k=2$ or $n-k =2$ (see Remark~\ref{rem:gridsIlya}).

\begin{Thm}\label{thm:dyn-simplyconnected}
Let $L_i \in \Grk(\Q)$ be an admissible sequence of rational subspaces (with a choice of lattice $\Lambda_{L_i}$ as in \S\ref{sec:diagemb}), let $g_i \in \SG(\R)$ and let $\mu_i$ be the Haar probability measure on the closed orbit
\begin{align*}
g_i\StDiag_{L_i}(\A) \SG(\Q) \subset \rquot{\SG(\A)}{\SG(\Q)}.
\end{align*}
Then $\mu_i$ converges to the Haar probability measure on $\rquot{\SG(\A)}{\SG(\Q)}$ as $i \to \infty$.
\end{Thm}
The rest of the section is devoted to proving Theorem~\ref{thm:dyn-simplyconnected}.
We remark that the notion of admissible sequences here is an ad hoc notion which appeared in other instances (see e.g.~\cite{2in4}) to assert a similar goal. 
The assumptions \eqref{item: c1 admissibility}--\eqref{item: c3 admissibility}
in the definition of admissibility are in fact necessary for the above theorem to hold while \eqref{item: c4 admissibility} can conjecturally be removed.

\subsection{A general result on equidistribution of packets}\label{sec:gorodnikoh}

The crucial input to our results is an $S$-arithmetic extension of a theorem of Mozes and Shah \cite{mozesshah} by Gorodnik and Oh \cite{gorodnikoh}. We state a version of it here for the reader's convenience.

Let $\mathsf{G}$ be a simply-connected connected semisimple algebraic group defined over $\Q$ and $Y_\A = \mathsf{G}(\A)/\mathsf{G}(\Q)$. 
Let $W$ be a compact open subgroup of $\mathsf{G}(\A_f)$. 
We denote by $C_c(Y_\A, W)$ the set of all continuous compactly supported functions on $Y_\A$ which are $W$ invariant.
Consider a sequence $(\mathsf{H}_i)_{i\in \N}$ of connected semisimple subgroups of $\mathsf{G}$ and let $\mu_i$ denote the Haar probability measure on the orbit $\mathsf{H}_i(\A)^+ \mathsf{G}(\Q) \subset Y_\A$ where $\mathsf{H}_i(\A)^+$ is the image of the adelic points of the simply connected cover of $\mathsf{H}_i$ in $\mathsf{H}_i(\A)$.
 For given $g_i \in \mathsf{G}(\A)$ we are interested in the weak* limits of the sequence of measures $g_i\mu_i$.

\begin{Thm}[{Gorodnik-Oh \cite[Theorem 1.7]{gorodnikoh}}]
\label{Theorem: GorodnikOh}
Assume that there exists a prime $p$ such that $\mathsf{H}_i$ is strongly isotropic at $p$ for all $i \in \N$. Then, for any  weak${}^\ast$ limit of the sequence $(g_i\mu_i)$ with $\mu(Y_\A) = 1$, there exists a connected $\Q$-group $\mathsf{M}< \mathsf{G}$ such that the following hold:
\begin{enumerate}[(1)]
\item For all $i$ large enough, there exist $\delta_i \in \mathsf{G}(\Q)$ such that:
\[ \delta_i^{-1} \mathsf{H}_i \delta_i \subset \mathsf{M}. \]
\item \label{item:item2-go} For any compact open subgroup $W$ of $\mathsf{G}(\A_f)$ there exists a finite index normal subgroup $M_0 = M_0(W)$ of $\mathsf{M}(\A)$ and $g \in \mathsf{G}(\A)$ such that $\mu$ agrees with the Haar probability measure on $gM_0 \mathsf{G}(\Q)$ when restricted to $C_c(Y_\A,W)$.
Moreover, there exists $h_i \in \mathsf{H}_i(\A)^+$ such that $g_i h_i \delta_i \rightarrow g$ as $i \rightarrow \infty$.
\item If the centralizers of $\mathsf{H}_i$ are $\Q$-anisotropic for all $i \in \N$, then $\mathsf{M}$ is semisimple.
Moreover, for any compact open subgroup $W$, $M_0 = M_0(W)$ in \ref{item:item2-go} contains $\mathsf{M}(\A)^+ \mathsf{M}(\Q)$.
\end{enumerate}
\end{Thm}
We remark that the theorem as stated in \cite{gorodnikoh} does not assume that $\mathsf{G}$ is simply connected; we will however only need this case.
\subsection{Proof of Theorem~\ref{thm:dyn-simplyconnected}}\label{sec:proofdyn}
We shall prove Theorem~\ref{thm:dyn-simplyconnected} in several steps and start with a short overview. 
Note that we have a morphism
\begin{align*}
\SG \to \bSG = \Spin_Q \times \SL_k \times \SL_{n-k}
\end{align*}
given by mapping $g \in \SPnk$ to $(\pi_1(g),\pi_2(g^{-1})^t)$ and $\Spin_Q$ to itself via the identity map (see also \S\ref{sec:diagemb}).
The first step of the theorem establishes equidistribution of the projections to the respective homogeneous quotients for $\Spin_Q, \SL_k, \SL_{n-k}$ (henceforth called 'individual equidistribution'). The second step is the analogous statement  for~$\bSG$.
Note that the admissibility assumption on the sequence of subspaces $L_i$ is used for individual equidistribution and in fact, the different conditions \eqref{item: c1 admissibility}--\eqref{item: c3 admissibility} imply the corresponding individual equidistribution statements (i.e.~\eqref{item: c1 admissibility} implies equidistribution in the homogeneous quotient $\rquot{\Spin_Q(\A)}{\Spin_Q(\Q)}$ etc.).

To vaguely outline the argument here, consider a sequence of orbits 
\begin{align*}
g_i' \StSpin_{L_i}(\A)\Spin_Q(\Q) \subset \rquot{\Spin_Q(\A)}{\Spin_Q(\Q)}.
\end{align*}
As the groups $\StSpin_{L_i}$ are maximal subgroups, the theorem of Gorodnik and Oh above implies that either the orbits are equidistributed or that there exist lattice elements $\delta_i$ so that $\delta_i\StSpin_{L_i}\delta_i^{-1}$ is eventually independent of $i$.
In the latter case, we also know that the lattice elements are up to a bounded amount in the stabilizer group; this will be shown to contradict the assumption that $\disc_Q(L_i) \to \infty$.

\subsubsection{Applying Theorem~\ref{Theorem: GorodnikOh}}

Consider the subgroup $\mathbf{J} = \Spin_Q \times \SL_n$. Note that $\mathbf{J}$ is semisimple and simply connected so that we may apply Theorem~\ref{Theorem: GorodnikOh} given a suitable sequence of subgroups.

The groups $\StSpin_{L_i}$ are potentially not simply connected so that a little more care is needed in applying Theorem~\ref{Theorem: GorodnikOh} to the orbit measures $\mu_i$.
We fix for any $i$ some $h_i \in \StDiag_{L_i}(\A)$ and consider the orbit measures on $g_ih_i \StDiag_{L_i}(\A)^+\SG(\Q)$. 
In view of the theorem, it suffices to show that these converge to the Haar probability measure on $\rquot{\SG(\A)}{\SG(\Q)}$.
Indeed, by disintegration the Haar measure on $g_i\StDiag_{L_i}(\A) \SG(\Q)$ is the integral over the Haar measures on $g_ih_i \StDiag_{L_i}(\A)^+\SG(\Q)$ when $h_i$ is integrated with respect to the Haar probability measure on the compact group $\StDiag_{L_i}(\A)/\StDiag_{L_i}(\A)^+$. 
In other words, the Haar measure on $g_i\StDiag_{L_i}(\A) \SG(\Q)$ is a convex combination of the Haar measures on the orbits $g_ih_i\StDiag_{L_i}(\A)^+ \SG(\Q)$.
To simplify notation, we replace $g_i$ by $g_ih_i$ in order to omit $h_i$.
Furthermore, we abuse notation and write $\mu_i$ for these ''components'' of the original orbit measures.

We fix a compact open subgroup $W$ of $\SG(\A_f)$ in view of (2) b) in Theorem~\ref{Theorem: GorodnikOh} and an odd prime $p$ as in the definition of admissibility of the sequence $(L_i)_i$.

Let $\mu$ be any weak${}^\ast$-limit of the measures $\mu_i$. Note that $\mu$ is a probability measure. Indeed, the pushforward of the measures $\mu_i$ to $\Spin_Q(\A)/\Spin_Q(\Q)$ has to converge to a probability measure as $\Spin_Q(\A)/\Spin_Q(\Q)$ is compact.
We let $\mathbf{M} < \mathbf{J}$ be as in Theorem~\ref{Theorem: GorodnikOh}.
As $g_i \in \SG(\A)$ and $\StDiag_{L_i} < \SG$ for all $i$, the support of the measures $\mu_i$ is contained in $\SG(\A) \mathbf{J}(\Q) \simeq \rquot{\SG(\A)}{\SG(\Q)}$. Thus $\mathbf{M}< \SG$.

\begin{claim*}
It suffices to show that $\mathbf{M} = \SG$.
\end{claim*}

\begin{proof}[Proof of the claim]
Suppose that $\mathbf{M} = \SG$.
Let $M_0 = M_0(W)$ be as in Theorem~\ref{Theorem: GorodnikOh}. 
Since $\SG(\A)$ has no proper finite-index subgroups \cite[Theorem 6.7]{BorelHomAbstr}, we have $M_0 = \SG(\A)$ (independently of $W$). Therefore, for any $W$-invariant continuous compactly-supported function $f$, the integral $\mu(f)$ agrees with the integral against the Haar measure on $\SG(\A)/\SG(\Q)$. But any continuous compactly-supported function is invariant under some compact open subgroup $W$, hence the claim follows.
\end{proof}

We now focus on proving that $\mathbf{M} = \SG$.
By Theorem~\ref{Theorem: GorodnikOh} there exist $\delta_i \in \SG(\Q)$ such that $\delta_i^{-1}\StDiag_{L_i}\delta_i < \mathbf{M}$ for all $i \geq i_0$. Furthermore, we fix $g \in \SG(\A)$ as well as $\hat{h}_i = (h_i, \Psi_{L_i}(h_i))\in \StDiag_{L_i}(\A)^+$ as in Theorem~\ref{Theorem: GorodnikOh} such that
\begin{align}\label{eq:GO-convergence for latticeel}
g_i \hat{h}_i \delta_i \to g.
\end{align}

\subsubsection{Individual equidistribution of subspaces and shapes}
Consider the morphism
\begin{align}\label{eq:defLeviG}
\SG \to \bSG = \Spin_Q \times \SL_k \times \SL_{n-k}.
\end{align}
In the following step of the proof, we show that the image $\bar{\mathbf{M}}$ of the subgroup $\mathbf{M}$ via \eqref{eq:defLeviG} projects surjectively onto each of the factors of $\bSG$.

\begin{Prop}\label{prop:projontofactors of Levi}
The morphism obtained by restricting the projection of $\bSG$ onto any almost simple factor of $\bSG$ to $\mathbf{M}$ is surjective.
\end{Prop}

\begin{proof}
We prove the proposition for each factor separately.
To ease notation, $\pi$ will denote the projection of $\bSG$ onto the factor in consideration, which we extend to $\SG$ by precomposition.

\textsc{First factor:}
As $\pi(\StDiag_{L_i}) = \StSpin_{L_i}$ we have for each $i$
\begin{align*}
\pi(\delta_i)^{-1} \StSpin_{L_i} \pi(\delta_i) < \pi(\mathbf{M}).
\end{align*}
Since $\StSpin_{L_i}$ is a maximal subgroup of $\Spin_Q$ (see Proposition~\ref{StabilizerMaximallyConnectedInSOQ}), there are two options: either $\pi_{1}(\mathbf{M}) = \Spin_Q$ or $\pi(\delta_i)^{-1} \StSpin_{L_i} \pi(\delta_i) = \pi(\mathbf{M})$ for all $i\geq i_0$.

Suppose the second option holds (as we are done otherwise).
Setting $\gamma_i = \pi(\delta_i\delta_{i_0}^{-1})$ and $L = L_{i_0}$ we have
\begin{align*}
\StSpin_{\gamma_i.L} = \gamma_i\StSpin_{L}\gamma_i^{-1} = \StSpin_{L_i}.
\end{align*}
By Proposition \ref{Proposition: Stabilizer Uniquely Determines Subspace} we have $\gamma_i.L = L_i$ or $\gamma_i.L^\perp = L_i$; changing to a subsequence and increasing $i_0$ we may suppose that the former option holds for all $i \geq i_0$.
By \eqref{eq:GO-convergence for latticeel} there exist $h_i \in \StSpin_{L_i}(\A)$ such that $\pi(g_i) h_i \gamma_i \to \pi(g')$ for some $g' \in \SG(\A)$.
Roughly speaking, this implies that $L_i = h_i\gamma_i.L \to \pi(g).L$ as $\Q_p$-subspaces for any prime $p$ contradicting the discriminant condition. More precisely, let $\varepsilon_i\to e$ be such that $\pi(g_i) h_i \gamma_i= \varepsilon_i\pi(g')$. Then for any prime $p$ the local discriminant gives
\begin{align*}
\disc_{p,Q}(L_i) = \disc_{p,Q}(h_{i,p}\gamma_i.L) 
= \disc_{p,Q}(\varepsilon_{i,p}\pi(g_p').L)
\end{align*}
If $i$ is large enough such that $\varepsilon_i \in \Spin_Q(\R\times \widehat{\Z})$, we have
\begin{align*}
\disc_Q(L_i) = \prod_p p^{\nu_p(\disc_{p,Q}(L_i))}
= \prod_pp^{\nu_p(\disc_{p,Q}(\pi(g_p').L))}
\end{align*}
which is constant, contradicting $\disc_Q(L_i) \to \infty$.

\textsc{Second factor:}
The proof is very similar to the first case, so we will be brief. By maximality of special orthogonal groups (Remark~\ref{rem:maximality}) and as 
$\pi(\StDiag_{L_i}) = \SO_{\qsubscr{L_i\cap \Lambda_{L_i}}}$ we may suppose by contradiction that for all $i \geq i_0$
\begin{align*}
\pi(\delta_i)^{-1}\SO_{\qsubscr{L_i \cap \Lambda_{L_i}}}\pi(\delta_i) = \pi(\mathbf{M}).
\end{align*}
Let us simplify notation and write $q_i$ for the least integer multiple of $\qsubscr{L_i \cap \Lambda_{L_i}}$ that has integer coefficients. 
Since $L_i \cap \Lambda_{L_i}$ and $L_i(\Z)$ are commensurable with indices controlled by $\disc(Q)$, we have
$\disc(q_{i}) \asymp \disc(q_{L_i})$ and $\disc(\tilde{q}_{i}) \asymp \disc(\tilde{q}_{L_i})$.
In particular, by our assumption $\disc(\tilde{q}_{i}) \to \infty$ as $i \to \infty$.

Set $\gamma_i = \pi(\delta_i\delta_{i_0}^{-1})\in \SL_k(\Q)$ so that
\begin{align}\label{eq:equalityorthgroups}
\SO_{\gamma_i\tilde{q}_{i_0}} = \SO_{\gamma_iq_{i_0}}
= \gamma_i \SO_{\qsubscr{L_{i_0}}}\gamma_{i}^{-1} 
= \SO_{q_{i}} = \SO_{\tilde{q}_{i}}.
\end{align}
By Proposition \ref{Proposition: Stabilizer Uniquely Determines Quadratic Form} there exist coprime integers $m_i,n_i$ such that
\begin{align*}
m_i \gamma_i\tilde{q}_{i_0} = n_i \tilde{q}_{i}.
\end{align*}
Using \eqref{eq:GO-convergence for latticeel} write $\pi(g_i) h_i \gamma_i = \varepsilon_i \pi(g')$ for some $g' \in \SG(\A)$ and $\varepsilon_i \to e$. 
By \eqref{eq:equalityorthgroups}, $h_i(\gamma_i\tilde{q}_{i_0}) = \gamma_i\tilde{q}_{i_0}$. Thus, we have for any prime $p$
\begin{align*}
m_i \varepsilon_{i,p}\pi(g_p')\tilde{q}_{i_0}
= m_i h_{i,p} \gamma_i\tilde{q}_{i_0}
= n_i \tilde{q}_{i}.
\end{align*}
The form $\pi(g_p')\tilde{q}_{i_0}$ is a form over $\Q_p$ with trivial denominators for all but finitely many $p$. 
Applying $\varepsilon_{i,p}$ for large $i$ does not change this. Furthermore, $m_i$ needs to divide all denominators of $\tilde{q}_{i_0}$ over $\Z_p$ for all $i$ as $\tilde{q}_{i}$ is primitive. Hence, $m_i$ can only assume finitely many values and by reversing roles one can argue the same for $n_i$.
For any prime $p$ we have 
\begin{align*}
\disc_{p}(\tilde{q}_{i})
= p^{\ord_p(\frac{m_i}{n_i})} \disc_p(\pi(g_p')\tilde{q}_{i_0}).
\end{align*}
and hence
\begin{align*}
\disc(\tilde{q}_{i}) 
= \tfrac{m_i}{n_i} \prod_{p} p^{\ord_p(\disc_p(\pi(g_p')\tilde{q}_{i_0}))}
\end{align*}
which is a contradiction to $\disc(\tilde{q}_{i}) \to \infty$.

\textsc{Third factor:}
The proof here is the same as for the second factor. We do however point out that the morphism $\SG \to \bSG$ was constructed to satisfy that for any $h \in \StSpin_{L_i}$ we have $\pi((h,\Psi_{L_i}(h)) = \psi_{2,L_i}(h)$ and hence $\pi(\StDiag_{L_i}) = \SO_{q_{L_i^\perp \cap \Lambda_{L_i}^\#}}$.
\end{proof}

\begin{remark}\label{rem:equidistr subspaces}
We recall from the beginning of this section \S\ref{sec:proofdyn} that the first three conditions in admissibility were used in this order for the three factors in the above proof. This has a consequence: If $L_i\in \Grk(\Q)$ is any sequence of subspaces satisfying properties \eqref{item: c1 admissibility} and \eqref{item: c4 admissibility}, then for any $g_i \in \Spin_Q(\R)$ the packets
\begin{align*}
g_i \StSpin_{L_i}(\A) \Spin_Q(\Q) \subset \rquot{\Spin_Q(\A)}{\Spin_Q(\Q)}
\end{align*}
are equidistributed as $i \to \infty$. This can be used to obtain equidistribution of $\HkQD \subset \Gr{k}(\R)$ without any restrictions on the $k$-power free part of $D$ (as opposed to our main theorems in the introduction).
\end{remark}

\subsubsection{Simultaneous equidistribution of subspaces and shapes}\label{sec:simult subsp+shapes}

Proposition~\ref{prop:projontofactors of Levi} shows that the image $\bar{\mathbf{M}}$ of $\mathbf{M}$ under \eqref{eq:defLeviG} satisfies that the projection onto each simple factors of $\bSG$ is surjective.
We claim that this implies $\bar{\mathbf{M}} = \bSG$.

We first show that the projection of $\bar{\mathbf{M}}$ to $\SL_k \times \SL_{n-k}$ is surjective.
Note that any proper subgroup of $\SL_k \times \SL_{n-k}$ with surjective projections is the graph of an isomorphism $\SL_k \to \SL_{n-k}$. In particular, we are done with the intermediate claim if $k \neq n-k$. Suppose that $ k = n-k$ and choose for some $i \geq i_0$ an element $h \in \StSpin_{L_i}$ acting trivially on $L_i$ but not trivially on $L_i^{\perp}$. The projection of $g_{L_i}^{-1}\isog(h)g_{L_i}$ to the first (resp.~the second) $\SL_{k}$ is trivial (resp.~non-trivial); the projection of $\bar{\mathbf{M}}$ to $\SL_k \times \SL_{n-k}$ thus contains elements of the form $(e,g)$ with $g \neq e$. This rules out graphs under isomorphisms and concludes the intermediate claim.

Now note that $\bar{\mathbf{M}}$ projects surjectively onto $\Spin_Q$ and $\SL_k \times \SL_{n-k}$ and that the latter two $\Q$-groups do not have isomorphic simple factors. By a similar argument as above, we deduce that $\bar{\mathbf{M}} = \bSG$.

\subsubsection{Handling the unipotent radical}

We now turn to proving that $\mathbf{M} = \SG$ which concludes the proof of the theorem.
By \S\ref{sec:simult subsp+shapes} we know that $\mathbf{M}$ surjects to $\bSG$. In particular, by the Levi-Malcev theorem there exists some element in the unipotent radical of $\SPnk$
\[ y_C = 
\begin{pmatrix}
I_k & C \\
0  & I_{n-k}
\end{pmatrix} \in \SPnk(\Q)
\]
such that $\mathbf{M}$ contains $\Spin_Q \times y_C \SDnk y_C^{-1}$.
By maximality of the latter group (cf.~Remark~\ref{rem:maximality}), $\mathbf{M}$ is either equal to $\SG$ or we have
\begin{align*}
\mathbf{M} = \Spin_Q \times y_C \SDnk y_C^{-1}.
\end{align*}
Assume by contradiction the latter.
The inclusion $\delta_i^{-1}\StDiag_{L_i}\delta_i \subset \mathbf{M}$ implies that
\begin{align*}
\delta_{2,i}^{-1}g_{L_i}^{-1}\isog(h)g_{L_i} \delta_{2,i} \in y_C \SDnk y_C^{-1}
\end{align*}
where $\delta_{2,i}$ denotes the second coordinate of the element $\delta_i \in \SG(\Q) = \Spin_Q(\Q)\times \SPnk(\Q)$.
Since $y_C \SDnk y_C^{-1}$ stabilizes two subspaces, namely $y_C L_0 = L_0$ and $L' = y_C(\{(0,\ldots,0)\}\times \Q^{n-k})$, the conjugated group $g_{L_i}\delta_{i,2}y_C \SDnk y_C^{-1} \delta_{i,2}^{-1}g_{L_i}^{-1}$ fixes the subspaces
\begin{align*}
g_{L_i}\delta_{i,2}L_0 
= g_{L_i}L_0 = L_i \quad \text{and} 
\quad g_{L_i}\delta_{i,2}L'.
\end{align*}
As $\StSpin_{L_i}$ fixes exactly the subspaces $L_i,L_i^\perp$, we must have
\begin{align}\label{eq:proofdyngrids1}
L_i^\perp = g_{L_i}\delta_{i,2}L'
\end{align} 
for all $i$.
We denote by $v_{1}^i,\ldots,v_n^i$ the columns of $g_{L_i}$ which is a basis of $\Lambda_{L_i}$ and by $w_1^i,\ldots,w_n^i$ its dual basis.
Recall that $w_{k+1}^i,\ldots, w_n^i$ form a basis of $\Lambda_{L_i}^\# \cap L_i^{\perp}$.
By \eqref{eq:proofdyngrids1}, there exists a rational number $\alpha_i\in \Q^\times$ such that 
\begin{align}\label{eq:proofdyngrids2}
\alpha_i (w^{i}_{k+1}\wedge \ldots \wedge w^{i}_n) = g_{L_i}\delta_{i,2}y_C (e_{k+1} \wedge \ldots \wedge e_n). 
\end{align}
To simplify notation, we set $\eta_i = \delta_{i,2}y_C$. 

We first control the numbers $\alpha_i$.
From \eqref{eq:GO-convergence for latticeel} we know that there are $h_i \in \StSpin_{L_i}$ such that
\begin{align*}
g_{2,i} g_{L_i}^{-1} \isog(h_i) g_{L_i} \eta_i \to g'
\end{align*}
for some $g' \in \SPnk(\A)$. For $i$ large enough, there exist $\varepsilon_i \in \SPnk(\RZ)$ with $g_{2,i} g_{L_i}^{-1} \isog(h_i) g_{L_i} \eta_i= \varepsilon_i g'$.
We now fix a prime $p$ so that $\isog(h_{i,p}) g_{L_i} \eta_i= g_{L_i}\varepsilon_{i,p} g'_p$ (as $g_{2,i} \in \SG(\R)$).
Applying $\isog(h_{i,p})$ to \eqref{eq:proofdyngrids1} we obtain
\begin{align*}
\alpha_i (w^{i}_{k+1}\wedge \ldots \wedge w^{i}_n) = g_{L_i}\varepsilon_{i,p} g'_p (e_{k+1} \wedge \ldots \wedge e_n).
\end{align*}
Considering that the vectors $w^{i}_{k+1}\wedge \ldots \wedge w^{i}_n$ and $e_{k+1} \wedge \ldots \wedge e_n$ are primitive (see e.g.~\cite[Ch.~1, Lemma 2]{casselsGeoNum}) and that $g_{L_i}$ and $g_p'$ have bounded denominators, this shows that the denominators and numerators of the numbers $\alpha_i$ are bounded independently of $i$.

We now compute the discriminant of the lattice spanned by $w^{i}_{k+1}, \ldots, w^{i}_n$ in two ways. First, note that as $w^{i}_{k+1},\ldots, w^{i}_n$ is a basis of $\Lambda_{L_i}^\# \cap L_i^\perp$, the discriminant in question is equal to the discriminant of $\Lambda_{L_i}^\# \cap L_i^\perp$ and hence $\asymp \disc_Q(L_i)$. 
For the second way, observe that by \eqref{eq:proofdyngrids2} the discriminant of the lattice spanned by $w^{i}_{k+1}, \ldots, w^{i}_n$ is given by $\alpha_i^{-1}$ multiplied by the determinant of the matrix with entries\footnote{
One conceptual way to see this is the following: the bilinear form $\langle \cdot,\cdot \rangle_Q$ induces a bilinear form $\langle \cdot,\cdot \rangle_{\bigwedge^{n-k}Q}$ on the wedge-product $\bigwedge^{n-k}\Q^n$ by defining it on pure wedges through
\begin{align*}
\langle v_1\wedge \ldots \wedge v_{n-k},w_1\wedge \ldots \wedge w_{n-k} \rangle_{\bigwedge^{n-k}Q} = \det(\langle v_i,w_j\rangle_Q).
\end{align*}
This definition asserts that the discriminant of a rank $n-k$ lattice is the quadratic value of the wedge product of any of its bases. Equation~\ref{eq:proofdyngrids3} is then obtained by replacing one of the wedges in $\langle w_{k+1}^i\wedge \ldots \wedge w_n^i,w_{k+1}^i\wedge \ldots \wedge w_n^i \rangle_{\bigwedge^{n-k}Q}$ via \eqref{eq:proofdyngrids2}.
}
\begin{equation}\label{eq:proofdyngrids3}
\langle g_{L_i} \eta_i e_j, w_{m}^i \rangle_Q \quad \text{with} \ j,m > k.
\end{equation}
To compute this determinant, write $\eta_i e_j = \sum_{\ell} a_{\ell j}^i e_\ell$ for all $j > k$ so that
\begin{align*}
g_{L_i} \eta_i e_j = \sum_{\ell} a_{\ell j}^i v_{\ell}^i.
\end{align*}
Using that $\{w^{i}_l\}$ are dual vectors to $\{v^{i}_l\}$ we compute
\[  \langle g_{L_i} \eta_i e_j, w_{m}^i \rangle_Q = \sum_{\ell} a_{\ell j}^i \langle v_{\ell}^i,w_{m}^i\rangle_Q =  a_{m j}^i \]
for all $m,j > k$. This implies that the determinant of the matrix with entries \eqref{eq:proofdyngrids3} is equal to the determinant of the lower right block of the matrix $\eta_i$. The latter being equal to one, we conclude that the discriminant of the lattice spanned by $w^{i}_{k+1}\wedge \ldots \wedge w^{i}_n$ is equal to $\alpha_i^{-1}$.

To summarize, we have established the following identity:
\[ \disc_Q(\Lambda_{L_i}^\# \cap L_i^\perp) = \alpha_i^{-1} \]
Since the left-hand side of this identity goes to infinity as $i \to \infty$ (because $\asymp \disc_Q(L_i)$) while the right-hand side is bounded, we have reached a contradiction. 
It follows that $\mathbf{M} = \SG$ and hence the proof of Theorem~\ref{thm:dyn-simplyconnected} is complete.

%% file: Sections/dynamical2.tex
\section{The dynamical version of the theorem in codimension 2} \label{SEC: two_dim_thm}
In the following, we prove the analogue of Theorem~\ref{thm:dyn-simplyconnected} for the case $k =2 $ and $n-k \geq 3$ (i.e.~$n \geq 5$) ignoring the unipotent radical (cf.~Remark~\ref{rem:gridsIlya}); the case $n-k =2$, $k \geq 3$ is completely analogous and can be deduced by passing to the orthogonal complement.
Contrary to cases treated in \S\ref{sec:dimatleast4}, the groups whose dynamics we use are not semisimple and have non-trivial central torus (see also Remark~\ref{rem:congcond}).

Recall the following notation (for $k=2$):
\begin{itemize}
\item $\bSG = \Spin_Q \times \SL_2 \times \SL_{n-2}$ (here the ambient group).
\item $\bStDiag_L = \{ (h,\psi_{1,L}(h),\psi_{2,L}(h)): h \in \StSpin_L\}$ (here the acting group) for any $L \in \Gr{k}(\Q)$ where $\psi_{1,L}$ resp.~$\psi_{2,L}$ is roughly the restriction of the action of $h$ to $L$ resp.~$L^\perp$ (cf.~\S\ref{sec:diagemb}).
\item For any $L \in \Gr{2}(\Q)$ a choice of intermediate lattice $\Z^n \subset \Lambda_L \subset (\Z^n)^\#$ (also implicit in the definition of $\bStDiag_L $). For simplicity, we assume here in addition that $\Lambda_L \cap L = L(\Z)$ and $\Lambda_L^\# \cap L^\perp = L^\perp(\Z)$; such a choice will be constructed later (cf.~Proposition~\ref{prop:constructionLambda}).
Again, if $Q$ is unimodular, $\Lambda_L = \Z^n$ satisfies this property.
\end{itemize}


\begin{Thm}\label{thm:dyn dim2}
Let $L_i \in \Gr{2}(\Q)$ for $i \geq 1$ be an admissible sequence of rational subspaces and let $g_i \in \bSG(\R)$ be such that $g_i \bStDiag_{L_i}(\R) g_i^{-1} = \bStDiag_{L_0}(\R)$.
Let $\mu_i$ be the Haar probability measure on the closed orbit
\begin{align*}
g_i\bStDiag_{L_i}(\A) \bSG(\Q) \subset \rquot{\bSG(\A)}{\bSG(\Q)}.
\end{align*}
Then $\mu_i$ converges to the Haar probability measure on $\rquot{\bSG(\A)}{\bSG(\Q)}$ as $i \to \infty$.
\end{Thm}

We will structure the proof somewhat differently as equidistribution in the first component turns out to be the most difficult challenge in the proof.
We fix an admissible sequence of subspaces $L_i$ and a prime $p$ as in the definition of admissibility.

Recall (cf.~\S\ref{sec:stabilizersintro}) that for any $L \in \Gr{2}(\Q)$ the group $\StSpin_L$ is not semisimple but only reductive.
Let us describe the center as well as the commutator subgroup of~$\StSpin_L$.
Define the pointwise stabilizer subgroup
\begin{align*}
\StSpinpt_{L} = \{g \in \Spin_Q: g.v = v \text{ for all }v \in L\}.
\end{align*}
The center of $\StSpin_L$ is equal to $\StSpinpt_{L^\perp}$ which we denote by $\mathbf{T}_L$ for simplicity as it is abelian in this case. The commutator subgroup of $\StSpin_L$ is the semisimple group $\StSpinpt_L$ and $\StSpin_L$ is isogenous to $\StSpinpt_L \times \mathbf{T}_L$ (see Remark~\ref{rem:isogenytostabprod}).
As in \S\ref{sec:dimatleast4}, one can use the measure rigidity result of Gorodnik and Oh \cite{gorodnikoh}, this time for subgroups of the form $\StSpinpt_L$.
These are however non-maximal so that we need to put extra effort to rule out intermediate groups\footnote{Roughly speaking, the obstacle to overcome are 'short vectors' in $L$. Ellenberg and Venkatesh \cite{localglobalEV} prove the theorem we are alluding to here assuming that $L$ does not contain 'short vectors' -- see also Proposition~\ref{prop:pointwisestab}.}.
Here, we use an averaging procedure involving the torus $\mathbf{T}_L$ as well as Duke's theorem \cite{duke88} to show that these obstructions typically do not occur.

Let us outline the structure of the proof:
\begin{itemize}
\item In \S\ref{sec:reductiontoindivid}, we show (in Lemma~\ref{lem:2dim indiv->equi}) that it is sufficient to prove equidistribution in each of the factors of $\bSG$, that is, to show equidistribution of the projections of the packets in Theorem~\ref{thm:dyn dim2} to
\begin{align}\label{eq:listoffactors}
\rquot{\Spin_Q(\A)}{\Spin_Q(\Q)},\ \rquot{\SL_2(\A)}{\SL_2(\Q)},\ \rquot{\SL_{n-2}(\A)}{\SL_{n-2}(\Q)}.
\end{align}
As mentioned in Remark~\ref{rem:disjointness}, we use the elementary fact that ergodic systems are disjoint from trivial systems for this reduction (see Lemma~\ref{lem:disjointnessfromtrivial}).
\item To prove equidistribution in each of the factors of $\bSG$, we first note that equidistribution in the third factor can be verified as in \S\ref{sec:dimatleast4}, Proposition~\ref{prop:projontofactors of Levi}. Equidistribution in the second factor turns out to be a variant of Duke's theorem~\cite{duke88} which we discuss in \S\ref{sec:2dim 2nd}. 
\item Due to the difficulties described above, equidistribution in the first factor of $\bSG$, is the hardest to prove (cf.~\S\ref{sec:firstfactor}) and implies Theorem~\ref{thm:dyn dim2} by the first two items in this list.
In \S\ref{sec:cor2ndfactor}, we collect a useful corollary of the above variant of Duke's theorem which we then use in Lemma~\ref{lem:ugly lemma} to prove that the subspaces in the packet do not contain short vectors on average.
\end{itemize} 

\subsection{Reduction to individual equidistribution}\label{sec:reductiontoindivid}
As explained, we begin by reducing Theorem~\ref{thm:dyn dim2} to the corresponding equidistribution statement in each of the factors of $\bSG$.
To this end, we will use the following elementary fact from abstract ergodic theory.

\begin{Lem}\label{lem:disjointnessfromtrivial}
Let $\mathsf{X}_1 = (X,\mathcal{B}_1,\mu_1,T_1)$ and $ \mathsf{X}_2 = (X_2,\mathcal{B}_2,\mu_2,T_2)$ be measure-preserving systems.
Suppose that $\mathsf{X}_1$ is ergodic and that $\mathsf{X}_2$ is trivial (i.e.~$T_2(x) = x$ for $\mu_2$-almost every $x \in X_2$).
Then the only joining of $\mathsf{X}_1$ and $\mathsf{X}_2$ is $\mu_1 \times \mu_2$.
\end{Lem}

\begin{proof}
Let $\nu$ be a joining and let $A_1 \times A_2 \subset X_1 \times X_2$ be measurable. It suffices to show that $\nu(A_1 \times A_2) = \mu_1(A_1)\mu_2(A_2)$.
By $T_1 \times T_2$-invariance of $\nu$, we have
\begin{align*}
\nu(A_1 \times A_2)
&= \int_{X_1 \times X_2} 1_{A_1}(x_1) 1_{A_2}(x_2) \de \nu(x_1,x_2) \\
&= \frac{1}{M} \sum_{m=0}^{M-1} \int_{X_1 \times X_2} 1_{A_1}(T_1^{m}x_1) 1_{A_2}(T_2^{m}x_2) \de \nu(x_1,x_2)
\end{align*}
As $\mathsf{X}_1$ is ergodic, there is a $\mu_1$-conull set $B_1 \subset X_1$ with
\begin{align*}
\frac{1}{M} \sum_{m=0}^{M-1} 1_{A_1}(T_1^m(x)) \to \mu_1(A_1)
\end{align*}
for every $x \in B_1$ by Birkhoff's ergodic theorem.
As $\mathsf{X}_2$ is trivial, there is a $\mu_2$-conull set $B_2$ with $T_2(x) = x$ for all $x \in B_2$. We let $B = B_1 \times B_2$ and note that $B$ has full measure as it is the intersection of the full-measure sets $B_1 \times X_2$ and $X_1 \times B_2$ (we use here that $\nu$ is a joining).
Therefore,
\begin{align*}
\nu(A_1 \times A_2) 
&=\frac{1}{M} \sum_{m=0}^{M-1} \int_{B} 1_{A_1}(T_1^{m}x_1) 1_{A_2}(T_2^{m}x_2) \de \nu(x_1,x_2)\\
&= \frac{1}{M} \sum_{m=0}^{M-1} \int_{B} 1_{A_1}(T_1^{m}x_1) 1_{A_2}(x_2) \de \nu(x_1,x_2) \\
&=  \int_{B} \frac{1}{M} \sum_{m=0}^{M-1} 1_{A_1}(T_1^{m}x_1) 1_{A_2}(x_2) \de \nu(x_1,x_2)\\
& \to \int_{B} \mu_1(A_1)  1_{A_2}(x_2) \de \nu(x_1,x_2) = \mu_1(A_1) \mu_2(A_2)
\end{align*}
as claimed.
\end{proof}

We aim to apply Lemma~\ref{lem:disjointnessfromtrivial} to any weak${}^\ast$-limit $\mu$ of the measures in Theorem~\ref{thm:dyn dim2}.
Thus, we need to establish some invariance of the latter.
Let $p$ be as in the definition of admissibility.

\begin{Lem}\label{lem:invariance of the limit}
There exists $g \in \GL_n(\Q_p)$ with the following property: Let $L \in \Gr{2}(\Q_p)$ be the subspace spanned by the first two columns of $g$.
Then $\mu$ is invariant under the subgroup of $\bStDiag_L(\Q_p) \subset \bSG(\Q_p)$ where
\begin{align*}
\bStDiag_L = \{ (h, \pi_1(g^{-1}\isog(h)g), \pi_2(g^{-1}\isog(h^{-1})g)^t) : h \in \StSpin_{L} \}.
\end{align*}
Moreover, the $\Q_p$-group $\bStDiag_L$ is strongly isotropic.
\end{Lem}

\begin{proof}
First of all, we prove that there exists a compact subset $K \subset \GL_n(\Q_p)$ such that $g_{L_i} \in K$ for all $i \in \N$. 
Recall that $g_{L_i}$ consists of a basis of an intermediate lattice $\Z^n \subseteq \Lambda_{L_i} \subseteq (\Z^n)^\#$ (cf.~\S\ref{sec:diagemb}).
The set $K$ of elements $g \in \GL_n(\Q_p)$ with $\Z_p^n \subset g\Z_p^n \subset (\Z_p^n)^\#$ is compact (in fact, it consists of finitely many cosets modulo $\GL_n(\Z_p)$ on the right).


By compactness of $K$ we may assume (by passing to a subsequence) that the sequence $(g_{L_i})_{i \in \N}$ converges to some $g \in K$. Let $L$ denote the $\Q_p$-plane spanned by the first two columns of $g$. Note that $\mu$ is $\bStDiag_L(\Q_p)$-invariant because each $\mu_i$ is $\bStDiag_{L_i}(\Qp)$-invariant. Therefore, we are left to show that $L$ is non-degenerate and $\bStDiag_L(\Q_p)$ is strongly isotropic.

We first observe that $L$ and $L^{\perp}$ are non-degenerate. Indeed, since $g_{L_i} \rightarrow g$, there exist $\Z_p$-bases of the subspaces $L_i$ which converge towards a basis of $L$.
Taking discriminants of $L_i$ and $L$ with respect to these bases, we obtain
\begin{align*}
\disc_{p,Q}(L_i) \rightarrow \disc_{p,Q}(L).
\end{align*}
Since $\Z_p^{\times}/(\Z_p^{\times})^2$ is discrete, $\disc_{p,Q}(L_i)$ is eventually constant and therefore $\disc_{p,Q}(L) = \disc_{p,Q}(L_i)$ for $i$ large enough; non-degeneracy of $L$ follows. In particular, $L^\perp$ is non-degenerate.
%

We may now use Corollary \ref{cor: strong-isotropicity stabilizer iff subspace} to show $\bStDiag_L$ or equivalently $\StSpin_{L}$ is strongly isotropic. 
Since $\StSpin_{L_i}$ is strongly isotropic at $p$, the quadratic spaces $(Q|_{L_i}, L_i)$ and $(Q|_{L_i^{\perp}}, L_i^{\perp})$ are isotropic over $\Q_p$.
By isotropy of the spaces $(Q|_{L_i}, L_i)$, we have a sequence of non-zero primitive vectors $v_i \in L_i(\Z_p)$ such that $Q(v_i)=0$ (after multiplying with denominators). By compactness of $\Z_p^n \setminus p\Z_p^n$,
the sequence $v_i$ admits a limit $v\in \Z_p^n \setminus p\Z_p^n$ after passing to a subsequence. This limit clearly satisfies $v \in L(\Z_p)$ and $Q(v) =0$, so $(Q|_L, L)$ is isotropic.
An identical argument proves that $(Q|_{L^{\perp}}, L^{\perp})$ is also isotropic, which proves (cf.~Corollary \ref{cor: strong-isotropicity stabilizer iff subspace}) that $\StSpin_{L}$ is a strongly isotropic group. The proof is complete.
\end{proof}

Recall that $\psi_{1,L},\psi_{2,L}$ denote the epimorphisms $\StSpin_{L_i} \to \SO_{q_{L_i}}$, $\StSpin_{L_i} \to \SO_{q_{L_i^\perp}}$ respectively.

\begin{Lem}\label{lem:2dim indiv->equi}
Suppose that individual equidistribution holds i.e.~that
\begin{enumerate}
\item $g_{i,1}\StSpin_{L_i}(\A) \Spin_Q(\Q)$ is equidistributed in $\rquot{\Spin_Q(\A)}{\Spin_Q(\Q)}$,
\item $g_{i,2}\psi_{1,L}(\StSpin_{L_i}(\A)) \SL_2(\Q)$ is equidistributed in $\rquot{\SL_2(\A)}{\SL_2(\Q)}$, and
\item $g_{i,3}\psi_{2,L}(\StSpin_{L_i}(\A)) \SL_{n-2}(\Q)$ is equidistributed in $\rquot{\SL_{n-2}(\A)}{\SL_{n-2}(\Q)}$.
\end{enumerate}
Then Theorem~\ref{thm:dyn dim2} holds.
\end{Lem}

\begin{proof}
Let $\mu$ be a weak${}^\ast$-limit and choose $L$ as in Lemma~\ref{lem:invariance of the limit}. 
By assumption, $\mu$ is a joining with respect to the Haar measures on each factor.
We proceed in two steps and apply Lemma~\ref{lem:disjointnessfromtrivial} once in each step.

For the first step, we choose $h \in \StSpin_{L}(\Q_p)$ which acts trivially on $L$ but non-trivially on $L^\perp$. As $\StSpin_{L}(\Q_p)$ is strongly isotropic we can choose $h$ so that it is unipotent and not contained in any normal subgroup of $\Spin_Q(\Q_p)$. Since $\Spin_Q$ is simply connected and $\Spin_Q(\Q_p)$ is isotropic, $\Spin_Q$ has strong approximation with respect to $\{p\}$ (see for example \cite[Thm.~7.12]{platonov}).
In particular, $\Spin_Q(\Q_p)$ acts ergodically on $X_1 = \rquot{\Spin_Q(\A)}{\Spin_Q(\Q)}$ with respect to the Haar measure on $X_1$.
By Mautner's phenomenon (see \cite[\S2]{MT2} for this instance), $h$ also acts ergodically.
Embedding $h$ diagonally (using the embedding in Lemma \ref{lem:invariance of the limit}), we can apply Lemma~\ref{lem:disjointnessfromtrivial} for $X_1$ as above and $X_2 = \rquot{\SL_2(\A)}{\SL_2(\Q)}$ and obtain that the pushforward of $\mu$ to $X_1 \times X_2$ is the Haar measure.

For the second step, we proceed similarly. Choose $h \in \StSpin_{L}(\Q_p)$ which acts trivially on $L^\perp$ but non-trivially on $L$. One checks that $h$ acts ergodically on $X_1 \times X_2$ (via $\pi_2(g^{-1}\isog(h^{-1})g)^t$ on the second factor cf. Lemma \ref{lem:invariance of the limit}). Applying Lemma~\ref{lem:disjointnessfromtrivial} again for $X_1 \times X_2$ and for $X_3 = \rquot{\SL_{n-2}(\A)}{\SL_{n-2}(\Q)}$ we obtain the claim.
\end{proof}

We prove the conditions of Lemma~\ref{lem:2dim indiv->equi} in an order which is potentially peculiar at first sight.
The third assertion can be proven exactly as in \S\ref{sec:dimatleast4} by applying \cite{gorodnikoh} (see Proposition~\ref{prop:projontofactors of Levi}) so we omit it here.

\subsection{Individual equidistribution in the second factor}\label{sec:2dim 2nd}
The aim of this second section is to prove the second assertion of Lemma~\ref{lem:2dim indiv->equi}.
As we shall see, it follows from Duke's theorem \cite{duke88} and its generalizations -- see e.g.~\cite{Dukeforcubic,harcosmichelII}.
Note that
\begin{align*}
g_{i,2}\psi_{1,L}(\StSpin_{L_i}(\A)) \SL_2(\Q) \subset g_{i,2}\SO_{q_{L_i}}(\A) \SL_2(\Q).
\end{align*}
While the right-hand side is equidistributed by Duke's theorem (specifically for instance by \cite[Thm.~4.6]{Dukeforcubic} or -- as we assume a splitting condition -- by \cite{W-linnik}), one needs to verify that the left-hand side has sufficiently large 'volume'.

\begin{Prop}\label{prop:squares}
For $L \in \Gr{2}(\Q)$ and any field $K$ of characteristic zero the image $\psi_{1,L}(\StSpin_{L}(K))$ contains the group of squares in the abelian group $\SO_{q_{L}}(K)$.
\end{Prop}

\begin{proof}
The proof is surprisingly involved.
Observe first that $\psi_{1,L}(\StSpin_L(K))$ contains $\psi_{1,L}(\mathbf{T}_L(K))$ which we now identify as the set of squares in $\SO_{q_{L}}(K)$.

We identify the torus $\mathbf{T}_{L}$ in terms of the Clifford algebra. Denote by $\mathcal{C}$ resp.~$\mathcal{C}^0$ the Clifford algebra of $Q$ resp.~the even Clifford algebra of $Q$.
Let $v_1,v_2$ be an orthogonal basis of $L$ and complete it into an orthogonal basis of $\Q^n$. Consider $X = v_1v_2 \in (\mathcal{C}^0)^\times$ ($L$ is non-degenerate) which satisfies the relations
\begin{align}\label{eq:proofsquares1}
Xv_i = v_i X,\ X v_1 = -Q(v_1) v_2 = -v_1 X,\ Xv_2 = -v_2 X
\end{align}
for all $i>2$. Moreover, $X^2 = -Q(v_1)Q(v_2) \in \Q^\times$.
Denote by $\sigma$ the standard involution on $\mathcal{C}$. Then $\sigma(X) = v_2 v_1 = -X$.

It follows directly from \eqref{eq:proofsquares1} that for all $a,b\in K$ the element $t = a+bX$ satisfies $tv_i = v_i t$ for $i>2$. Also,
\begin{align*}
t v_1 \sigma(t) 
&= (a+bX)v_1(a-bX) = a^2v_1 +ab Xv_1 -abv_1 X -b^2 Xv_1 X\\
&=a^2 v_1 -2abQ(v_1) v_2 - b^2Q(v_1)Q(v_2) v_1 \in L
\end{align*}
and similarly for $v_2$. Therefore, $t \in \mathbf{T}_L$ if and only if
\begin{align*}
\sigma(t) t = (a-bX)(a+bX) = a^2-b^2X^2 = a^2-b^2Q(v_1)Q(v_2) = 1.
\end{align*}
We set
\begin{align*}
F = \Q(-Q(v_1)Q(v_2)) = \Q(-\disc_Q(L))
\end{align*}
and embed $F$ into $\mathcal{C}^0$ via $\sqrt{-\disc_Q(L)} \mapsto X$.
The non-trivial Galois automorphism on $F$ is then given by $\sigma|_F$.
To summarize, we obtain
\begin{align*}
\mathbf{T}_L(K) = \{ t \in F \otimes K: \sigma(t) t = 1\}.
\end{align*}
Also, recall that the special Clifford group surjects onto $\SO_Q$ so that one may show quite analogously
\begin{align*}
\SO_{q_L}(K) = (F \otimes K)^\times/K^\times.
\end{align*}
The proposition then follows from Hilbert's theorem 90 as in the proof of \cite[Lemma 7.2]{W-linnik}.
\end{proof}

\begin{Cor}
The orbits
\begin{align*}
g_{i,2}\psi_{1,L}(\StSpin_{L_i}(\A)) \SL_2(\Q) \subset \rquot{\SL_2(\A)}{\SL_2(\Q)}
\end{align*}
equidistribute as $i \to \infty$.
\end{Cor}

\begin{proof}
We deduce the corollary from existing literature and Proposition~\ref{prop:squares}.
We first claim that as $i \to \infty$ the sets 
\begin{align}\label{eq:adaptedorbit}
g_{i,2}\SO_{q_{L_i}}(\widehat{\Z})\psi_{1,L}(\StSpin_{L_i}(\A)) \SL_2(\Q)
\end{align}
are equidistributed. By Proposition~\ref{prop:squares} the abelian group $\SO_{q_{L_i}}(\widehat{\Z})\psi_{1,L}(\StSpin_{L_i}(\A))$ contains the group $\SO_{\qsubscr{L_i}}(\widehat{\Z})\SO_{\qsubscr{L_i}}(\A)^2$ where $\SO_{\qsubscr{L_i}}(\A)^2$ denotes the group of squares. 

The orbit \eqref{eq:adaptedorbit} is then a union of suborbits of the same form associated to these subgroups.
Any sequence of such suborbits are equidistributed, for instance, by \cite{harcosmichelII} as the volume is of size\footnote{Since the $2$-torsion of the Picard group of the order of discriminant $\disc_Q(L_i)$ has size $\disc_Q(L_i)^{o(1)}$ (see e.g.~\cite[p.~342]{Cassels}), the squares form a subgroup of size $\disc_Q(L_i)^{\frac{1}{2}+o(1)}$.} $\disc_Q(L_i)^{\frac{1}{2}+o(1)}$.
We note that the result in \cite{harcosmichelII} allows for smaller volumes (where the exponent $\frac{1}{2}$ can be replaced by $\frac{1}{2}-\eta$ for some not too large $\eta >0$).
In the case needed here one can also apply Linnik's ergodic method as we assume a splitting condition at a fixed prime -- see \cite[\S7]{W-linnik}.
By averaging, the claim in \eqref{eq:adaptedorbit} follows.
The corollary is implied by \eqref{eq:adaptedorbit} and ergodicity of the Haar measure on $\SL_2(\A)/\SL_2(\Q)$ under any diagonal flow.
\end{proof}

\subsection{Individual equidistribution in the first factor}\label{sec:firstfactor}

In view of the discussion in \S\ref{sec:2dim 2nd} and Lemma~\ref{lem:2dim indiv->equi}, it suffices to show equidistribution of the packets
\begin{align*}
g_{i,1}\StSpin_{L_i}(\A) \Spin_Q(\Q) \subset \rquot{\Spin_Q(\A)}{\Spin_Q(\Q)}
\end{align*}
to prove Theorem~\ref{thm:dyn dim2}. We proceed in several steps.

\subsubsection{An equidistribution theorem for the pointwise stabilizers}

We first establish the following proposition which shows that either orbits of the pointwise stabilizer are equidistributed or there is some arithmetic obstruction.

\begin{Prop}\label{prop:pointwisestab}
Let $(L_i)_i$ be a sequence of $2$-dimensional rational subspaces such that there exists a prime $p$ for which $\StSpinpt_{L_i}(\Q_p)$ is strongly isotropic for all $i$. Let $g_i \in \G(\R)$ and assume that $\disc_Q(L_i)\to \infty$ as $i \to \infty$.
Then one of the following statements is true:
\begin{enumerate}
\item The packets $g_i\StSpinpt_{L_i}(\A)\Spin_Q(\Q)$ are equidistributed in $\rquot{\Spin_Q(\A)}{\Spin_Q(\Q)}$ as $i \to \infty$.
\item There exists a rational vector $v \in \Q^n\setminus\{0\}$ and lattice elements $\delta_i \in \Spin_Q(\Q)$ such that 
\begin{align*}
\Q v = \bigcap_i \delta_i^{-1}.L_i(\Q).
\end{align*}
The lattice elements additionally satisfy that there exist $h_i \in \StSpinpt_{L_i}(\A)$ such that the sequence $g_i h_i\delta_i$ is convergent as $i \to \infty$.
\end{enumerate}
\end{Prop}

\begin{proof}
We prove the proposition in exactly the same way we proved the first case in Proposition~\ref{prop:projontofactors of Levi}; we will thus be rather brief.
Let $\delta_i \in \Spin_Q(\Q)$ and a connected $\Q$-group $\mathbf{M} < \bSG$ be as in Theorem~\ref{Theorem: GorodnikOh}. In particular,
\begin{align*}
\delta_i^{-1} \StSpinpt_{L_i} \delta_i < \mathbf{M}
\end{align*}
and it suffices for equidistribution to verify that $\mathbf{M} = \Spin_Q$. 
One can see that $\mathbf{M}$ strictly contains $\delta_i^{-1} \StSpinpt_{L_i} \delta_i$ for all $i$ by using $\disc_Q(L_i) \to \infty$ and repeating the proof of the first case in Proposition~\ref{prop:projontofactors of Levi}.

Contrary to the case treated in Proposition~\ref{prop:projontofactors of Levi} the groups $\StSpinpt_{L_i}$ are non-maximal.
The intermediate groups can however be understood explicitly: they are of the form $\StSpinpt_{W}$ where $W$ is a rational line contained in $\delta_i^{-1}.L_i$ for all $i$.
For a proof of this fact we refer to \cite[Prop.~4]{localglobalEV}, see also the arXiv version of the same paper where the authors give an elementary proof in the case $n-2 \geq 7$.
This concludes the proof of the proposition.
\end{proof}

\begin{Cor}
Let the notation and the assumptions be as in Proposition \ref{prop:pointwisestab} and suppose that the second case holds. Then
\begin{align*}
\min_{w \in L_i(\Z)\setminus\{0\}} Q(w) = \min_{w \in \Z^2\setminus\{0\}} q_{L_i}(w)
\end{align*}
is bounded as $i \to \infty$.
\end{Cor}

\begin{proof}
Let $v \in \Q^n$ be as in Proposition~\ref{prop:pointwisestab} and suppose without loss of generality that $v$ is integral and primitive. Suppose also that $g_i h_i \delta_i \to g' \in \Spin_Q(\A)$ and write $g_i h_i \delta_i = \varepsilon_i g'$ where $\varepsilon_i \to e$.
Let $i_0$ be large enough so that $\varepsilon_i \in \bSG(\RZ)$ for all $i \geq i_0$ and let $N \in \N$ be the smallest integer such that $Ng'_p$ is integral for all primes $p$.

We claim that $v_i:= N \delta_i.v \in L_i(\Z)$. To see this, first note that $v_i\in L_i(\Q)$.
Furthermore, for any prime $p$ the vector $v_i$ is contained in $L(\Z_p)$.
Indeed, $h_{i,p} \in \StSpinpt_{L_i}(\Q_p)$ necessarily fixes $v_i$ and, as $g_{i,p} = e$, 
\begin{align*}
v_i = h_{i,p}.v_i = N h_{i,p}\delta_i.v = N\varepsilon_{i,p} g_{p}.v \in \Z_p^n.
\end{align*}
This proves the claim and hence the corollary as $Q(N\delta_i.v) = N^2 Q(v)$.
\end{proof}

\subsubsection{A corollary of equidistribution in the second factor}\label{sec:cor2ndfactor}

In the following we would like to give an estimate on the measure of the set of points in $g_{i,1}\mathbf{T}_{L_i}(\A)\Spin_Q(\Q)$ whose associated point in $\rquot{\SL_2(\A)}{\SL_2(\Q)}$ is 'close' to the cusp.
This will allow to 'wash out' the effect of the obstructions in Proposition~\ref{prop:pointwisestab} on average across the full stabilizer group.
To obtain said estimate, we introduce a height function that suits our needs.

Let $\shapesp{2}$ be the space of positive definite real binary quadratic forms up to similarity\footnote{Two positive definite binary real quadratic forms $Q_1,Q_2$ are similar if there exist $\lambda>0$ and $g \in \GL_2(\Z)$ with $Q_2(\cdot) = \lambda Q_1(g\cdot)$. Note that the space $\shapesp{2}$ will be discussed in more detail in \S\ref{sec:homspaces}.} and write $[q]$ for the similarity class of a binary form $q$.
We define for $\varepsilon>0$
\begin{align*}
\shapesp{2}(\varepsilon) = \big\{ [q]\in \shapesp{2}: \min_{w \in \Z^2 \setminus\{0\}} q(w) > \varepsilon \sqrt{\disc(q)}\big\}.
\end{align*}
Note that the condition is independent of the choice of representative of $[q]$.

By Mahler's compactness criterion \cite{Mahlercompactness}, these are compact subsets of $\shapesp{2}$ and any compact subset is contained in $\shapesp{2}(\varepsilon)$ for some $\varepsilon>0$.
Furthermore, one can show that the Haar measure of $\shapesp{2} \setminus \shapesp{2}(\varepsilon)$ is $\ll \varepsilon$ by direct integration of the hyperbolic area measure on that region.

We define $K_\varepsilon \subset \rquot{\SL_2(\A)}{\SL_2(\Q)}$ to be the preimage of $\shapesp{2}(\varepsilon)$ under the composition
\begin{align*}
\rquot{\SL_2(\A)}{\SL_2(\Q)} \to \rquot{\SL_2(\R)}{\SL_2(\Z)} \to \rquot{\PGL_2(\R)}{\PGL_2(\Z)} \to \shapesp{2}.
\end{align*}
By the previous discussion, this is a compact set whose complement has Haar measure $\ll \varepsilon$.
For $x \in \rquot{\SL_2(\A)}{\SL_2(\Q)}$ we  shall call the supremum over all $\varepsilon>0$ with $x \in K_\varepsilon$ the \emph{minimal  quadratic value} for $x$.

The following is a direct corollary of equidistribution in the second factor.
\begin{Cor}\label{cor:non-div from equidistr}
For any $\varepsilon\in (0,1)$ there exists $i_0 \geq 1$ so that the measure of the set of points $t \in \rquot{\mathbf{T}_{L_i}(\A)}{\mathbf{T}_{L_i}(\Q)}$ for which $g_{2}\psi_{1,L_i}(t)\SL_2(\Q) \not\in K_\varepsilon$ is $\ll \varepsilon$ for all $i \geq i_0$.
\end{Cor}

\subsubsection{Using the shape in the subspace}

In the following we identify the minimal quadratic value for the points on the orbits in the context of proving Theorem~\ref{thm:dyn dim2}.
As $\Spin_Q(\RZ)$ is a compact open subgroup, it has finitely many orbits on $\Spin_Q(\A)/\Spin_Q(\Q)$ (these correspond to the spin genus of the quadratic form $Q$).
We choose a finite set of representatives $\mathcal{R} \subset \Spin_Q(\A_f)$ such that
\begin{align}\label{eq:spin genus}
\Spin_Q(\A)/\Spin_Q(\Q) = \bigsqcup_{\mathsf{r} \in \mathcal{R}} \Spin_Q(\RZ) \mathsf{r}\, \Spin_Q(\Q).
\end{align}
Note that in $\SL_2$ (or $\SL_{n-2}$) any $g \in \SL_2(\A)$ can be written as $g = b \gamma$ where $b \in \SL_2(\RZ)$ and $\gamma \in \SL_2(\Q)$. 

\begin{Lem}\label{lem:ugly lemma}
Let $h \in \bStDiag_{L_i}(\A)$ and write $h \gamma = b \mathsf{r}$ for some $\gamma \in \bSG(\Q)$, $b \in \bSG(\R \times \widehat{\Z})$, and $\mathsf{r} \in \mathcal{R}$.
Then $\gamma_1^{-1}.L_i$ is a rational subspace of discriminant $\asymp D$.  Furthermore, the minimum
\begin{align*}
\min_{w \in \Z^2 \setminus\{0\}} \frac{q_{\gamma_1^{-1}.L_i}(w)}{\sqrt{\disc_Q(\gamma_1^{-1}.L_i)}}
\end{align*}
is comparable to the minimal quadratic value for $g_{i,2}\psi_{1,L}(h)\SL_2(\Q)$.
\end{Lem}

Note that a lemma in this spirit will later on be used to deduce the main theorems from their dynamical counterparts (cf.~Proposition~\ref{Proposition: packets are in RkD}).
The statement here is more technical in nature (as it needs to treat different genera) and the reader is encouraged to return to the proof after reading Proposition~\ref{Proposition: packets are in RkD}.
We note that such a treatment has appeared in different context in the literature \cite{localglobalEV,ALMW-elliptic}.

\begin{proof}
The ingredients for this proof are all contained in the proof of Proposition~\ref{Proposition: packets are in RkD}; so we will be brief.
Write $L = L_i$ for simplicity.
Note that $h_{1,p} \gamma_1 = b_{1,p}\mathsf{r}_p$ and hence
\begin{align*}
\disc_{p,Q}(\gamma_1^{-1}.L)
&= \disc_{p,Q}(\gamma_1^{-1}h_{1,p}^{-1}.L)
= \disc_{p,Q}(\mathsf{r}_p^{-1}b_{1,p}.L)\\
&\asymp_{\mathsf{r}} \disc_{p,Q}(b_{1,p}.L) = \disc_{p,Q}(L).
\end{align*}
As the discriminant is a product of the local discriminants \eqref{eq:disc loctoglob}, this proves the first claim.

For the second claim, we let $L' = \gamma_1^{-1}.L$ and first consider $m = g_{L'}^{-1}\isog(\gamma_1^{-1})g_L \gamma_2\in \GL_n(\Q)$.
Observe that
\begin{align*}
m L_0 = g_{L'}^{-1}\isog(\gamma_1^{-1})g_L L_0 = g_{L'}^{-1}\isog(\gamma_1^{-1})L = g_{L'}^{-1}L' = L_0.
\end{align*}
As we will now see, $m$ is 'almost integral' and invertible. For this, compute
\begin{align*}
\isog(\gamma_1^{-1})g_L \gamma_2 = \isog(\gamma_1^{-1} h_{1,p}^{-1})g_L h_{2,p} \gamma_2 = \isog(\mathsf{r}_p^{-1} b_{1,p}^{-1})g_L b_{2,p}.
\end{align*} 
This implies that there exists some $N \in\N$ independent of $L$ such that $N\isog( \gamma_1^{-1})g_L \gamma_2$ and $N\isog(\gamma_1^{-1})g_L^{-1} \gamma_2$ are integral. 
Recall that $\disc(Q)g_{L'},\disc(Q)g_{L'}^{-1}$ are integral so that $N\disc(Q) m$ and $N \disc(Q)m^{-1}$ are integral.
This discussion implies that for any two positive definite real binary quadratic forms $q,q'$ with the property that $\pi_1(m)q$ and $q'$ are similar we have
\begin{align*}
\min_{w \in \Z^2\setminus\{0\}} \frac{q(w)}{\sqrt{\disc(q)}} 
\asymp 
\min_{w \in \Z^2\setminus\{0\}} \frac{q'(w)}{\sqrt{\disc(q')}}.
\end{align*}
Here, recall that $\GL_2(\R)$ acts on binary forms via $gq(x) = q(g^t x)$.

Now note
\begin{align*}
[q_{L'}] = [Q(g_L' \cdot)]
= [Q(\isog(\gamma_1) g_{L'}\cdot)]
\end{align*} 
while the similarity class belonging to $g_{2}\psi_{1,L}(t)\SL_2(\Q)$ is
\begin{align*}
[\gamma_2^{-1}q_{L}]= [Q(g_L \gamma_2 \cdot)] = [Q(\isog(\gamma_1) g_{L'}m \cdot)]
\end{align*}
The claim follows.
\end{proof}

\subsubsection{Proof of Theorem~\ref{thm:dyn dim2}}

As explained, it now suffices to prove that the packets for $L_i$
\begin{align*}
g_{i,1} \StSpin_{L_i}(\A) \Spin_Q(\Q) \subset \rquot{\Spin_Q(\A)}{\Spin_Q(\Q)}
\end{align*}
equidistribute as $\disc_Q(L_i) \to \infty$. 
Similarly to the situation in the proof of Theorem~\ref{thm:dyn-simplyconnected}, we need to circumvent the problem that $\StSpin_L$ for $L \in \Gr{2}(\Q)$ is not exactly isomorphic to $\StSpinpt_L \times \mathbf{T}_L$ -- see Remark~\ref{rem:isogenytostabprod} for a more careful discussion.
Denote by $\StSpin_L(\A)^\star$ the image of $\StSpinpt_L(\A) \times \mathbf{T}_L(\A) \to \StSpin_L(\A)$; this is a normal subgroup of $\StSpin_L(\A)$ with the property that $K_L:= \rquot{\StSpin_L(\A)}{\StSpin_L(\A)^\star}$ is compact and abelian.
By an argument as in the beginning of the proof of Theorem~\ref{thm:dyn-simplyconnected} it suffices to show that for any $k_i \in K_{L_i}$ the orbits
\begin{align*}
g_{i,1} k_i\StSpin_{L_i}(\A)^\star \Spin_Q(\Q) \subset \rquot{\Spin_Q(\A)}{\Spin_Q(\Q)}
\end{align*}
are equidistributed as $i \to \infty$.
We let $\mu_i$ be the Haar measure on the $i$-th such orbit and let
\begin{align*}
\mu_{i,1} = m_{\StSpinpt_{L_i}(\A)\Spin_Q(\Q)},\quad 
\mu_{i,2} = m_{\mathbf{T}_{L_i}(\A)\Spin_Q(\Q)}
\end{align*}
be the Haar measure on the closed orbits of $\StSpinpt_{L_i}(\A)$ resp.~$\mathbf{T}_{L_i}(\A)$.
Then we have for any function $f \in C_c(\rquot{\Spin_Q(\A)}{\Spin_Q(\Q)})$
\begin{align}\label{eq:TxH disintegration}
\int f \de \mu_i = \int \int f(g_{i,1} k_i h t) \de \mu_{1,i}(h) \de \mu_{i,2}(t).
\end{align}
In the following, we identify $k_i$ with a representative in a fixed bounded region of $\StSpin_{L_i}(\A)$.

For a fixed $t_i\in \mathbf{T}_{L_i}(\A)$, the inner integral is the integral over the orbit
\begin{align*}
g_{i,1}k_i \StSpinpt_{L_i}(\A)t_i \Spin_Q(\Q) 
= g_{i,1}k_i t_i\StSpinpt_{L_i}(\A) \Spin_Q(\Q).
\end{align*}
Writing $t_i\gamma_i = b_i\mathsf{r}$ as in \eqref{eq:spin genus}, we see that
\begin{align*}
g_{i,1}k_i t_i\StSpinpt_{L_i}(\A) \Spin_Q(\Q) 
&= g_{i,1}k_ib_i\mathsf{r} \gamma_i^{-1} \StSpinpt_{L_i}(\A) \Spin_Q(\Q)\\
&= g_{i,1}k_ib_i\mathsf{r} \StSpinpt_{\gamma_i^{-1}.L_i}(\A) \Spin_Q(\Q)
\end{align*}
which is equidistributed if and only if $\StSpinpt_{\gamma_i^{-1}.L_i}(\A) \Spin_Q(\Q)$ is equidistributed (as $g_{i,1}k_ib_i$ is bounded).
By Proposition~\ref{prop:pointwisestab} and its corollary it suffices to show that the minimal non-zero value of $q_{\gamma_i^{-1}.L_i}$ goes to infinity.
This minimum is comparable to the minimal quadratic value for $g_{i,2} \psi_{1,L}(t_i) \SL_2(\Q)$ by Lemma~\ref{lem:ugly lemma}.

Motivated by this observation, we define for $\varepsilon>0$
\begin{align*}
\mathcal{B}_i(\varepsilon) = \{t \mathbf{T}_{L_i}(\Q): g_{i,2} \psi_{1,L}(t)\SL_2(\Q) \in K_{\varepsilon}\} \subset \rquot{\mathbf{T}_{L_i}(\A)}{\mathbf{T}_{L_i}(\Q)}
\end{align*}
so that the complement of $\mathcal{B}_i(\varepsilon)$ has $\mu_{i,2}$-measure $\ll \varepsilon$ for all $i$ large enough (depending on $\varepsilon$) by Corollary~\ref{cor:non-div from equidistr}.
In view of \eqref{eq:TxH disintegration}, this implies that
\begin{align*}
\int f \de \mu_i 
= \tfrac{1}{\mu_{i,2}(\mathcal{B}_i(\varepsilon))} 
\int_{\mathcal{B}_i(\varepsilon)} \int f(g_{i,1}k_i h t) \de \mu_{1,i}(h) \de \mu_{i,2}(t) + \mathcal{O}(\varepsilon)
\end{align*}
By the previous paragraph, the orbits $g_{i,1}k_i \StSpinpt_{L_i}(\A)t_i \Spin_Q(\Q)$ are equidistributed for any sequence $t_i \in \mathcal{B}_i(\varepsilon)$. The integral on the right-hand side is a convex combination of such orbital integrals and hence must converge to the integral of $f$ over the Haar measure.
Letting $\mu$ be any weak${}^*$-limit of the measures $\mu_i$ we obtain
\begin{align*}
\int f \de \mu = \int f \de m_{\rquot{\Spin_Q(\A)}{\Spin_Q(\Q)}} + \mathcal{O}(\varepsilon).
\end{align*}
As $\varepsilon$ is arbitrary, this implies the claim.

%% file: Sections/quadraticforms.tex
\section{Discriminants and glue groups} \label{SEC:quadratic_forms}

Recall that $Q$ is a positive definite integral quadratic form on $\Q^n$ and $\langle \cdot,\cdot\rangle_Q$ is its symmetric bilinear form.
By integrality we mean that $\langle \cdot,\cdot\rangle_Q$ takes integer values on $\Z^n\times \Z^n$.
The goal of this section is to prove the following proposition:

\begin{Prop}\label{prop:disccomparison}
For any subspace $L \subset \Q^n$ there exist two positive divisors $m_1,m_2$ of $\disc(Q)$ with 
\begin{align*}
\disc_Q(L^\perp) = \frac{m_1}{m_2}\disc_Q(L).
\end{align*}
In particular,
\[ \frac{1}{\disc(Q)} \disc_Q(L) \leq \disc_Q(L^\perp) \leq\disc(Q) \disc_Q(L). \]
\end{Prop}

To that end, we will use the notion of glue groups defined in \S\ref{sec:gluedef} and, in fact, prove a significantly finer statement in Proposition~\ref{Prop:Isomorphism between the quotients of piL(Rn) and piL perp (Rn)} below.

\subsection{Definitions}\label{sec:gluedef}
For any $\Z$-lattice $\Gamma \subset \Q^n$ we define the dual lattice
\begin{align*}
    \Gamma^\# = \{x \in \Gamma \otimes \Q: \langle x,y\rangle_Q \in \Z \text{ for all } y \in \Gamma\}.
\end{align*}
If $\Gamma \subset \Z^n$ (or more generally if $\langle \cdot,\cdot\rangle_Q$ takes integral values on $\Gamma\times \Gamma$), the dual lattice $\Gamma^\#$ contains $\Gamma$.
Note that if $\Gamma_1 \subset \Gamma_2$ are any two $\Z$-lattices then $\Gamma_1^\#\supset \Gamma_2^\#$.

For the purposes of this section, a very useful classical tool is the so-called \emph{glue-group}, which one could see as a concept generalizing the discriminant.
We introduce only what is needed here; for better context we refer to \cite{conwaysloane,McMullenglue} (in particular, we do not introduce the fractional form).
We define the \emph{glue group} of a rational subspace $L$ (or of the lattice $L(\Z)$) as
\begin{align*}
\mathcal{G}(L) = L(\Z)^\# / L(\Z).
\end{align*}
Note that $L(\Z)^\#$ contains $L(\Z)$ by integrality. The glue-group is a finite abelian group whose cardinality is exactly the discriminant (see e.g.~\cite[\S5.1]{kitaoka}).
We remark that the glue group may be computed from local data -- this is made explicit in \S\ref{sec:local glue groups} of the appendix.

\begin{remark}
For each discriminant $D$, one may consider the collection of subspaces $L \in \Gr{k}(\Q)$ with discriminant $D$ and glue group a fixed abelian group of order~$D$. In principle, the results of the current article should carry over to prove equidistribution of these subspaces together with their shapes (cf.~\cite{2in4}). However, it is not clear when one expects such collections to be non-empty, even when $Q$ is the sum of squares.
\end{remark}

\subsection{The glue group of the orthogonal complement}\label{sec:gluegroups}

We study the relation between the glue group of a subspace and that of its orthogonal complement.
Any subspace $L\subset \Q^n$ contains various lattices which are (potentially) of interest and are natural:
\begin{itemize}
\item the intersections $L(\Z) = L(\Q) \cap \Z^n$ and $L(\Q) \cap (\Z^n)^\#$,
\item the dual lattice $L(\Z)^\#$, and
\item the projection lattices $\pi_L(\Z^n)$ and $\pi_L((\Z^n)^\#)$ where $\pi_L: \Q^n \to L$ denotes the orthogonal projection.
\end{itemize}

\begin{Lem}[Elementary properties]\label{lem:imageoforthproj}
The following relations between the aforementioned lattices hold:
\begin{enumerate}[(i)]
\item \label{item:ElemProp ProjvsDualA 1} $L(\Z)^\# = \pi_L((\Z^n)^\#)$ and $(L \cap (\Z^n)^\#)^\# = \pi_L(\Z^n)$.
\item \label{item:ElemProp ProjvsDualA 2}
$(L \cap (\Z^n)^\#) / L(\Z) \simeq L(\Z)^\# / \pi_L(\Z^n)$.
\end{enumerate}
\end{Lem}

\begin{proof}
We prove \ref{item:ElemProp ProjvsDualA 1} first. 
Since the proofs of the two assertions in  \ref{item:ElemProp ProjvsDualA 1} are similar, we only detail the first.
Let $v_1, \ldots, v_k$ be a $\Z$-basis of $L(\Z)$.
Moreover, let $w_1, \ldots, w_k \in L$ be the dual basis to  $v_1, \ldots, v_k$.
Extend $v_1, \ldots, v_k$ to a basis $v_1, \ldots, v_n$ of $\Z^n$ and consider $y_1, \ldots, y_n$ the dual basis to $v_1, \ldots, v_n$.
Then $\pi_L(y_i) = w_i$ for any $i \leq k$ as
\begin{align*}
\langle \pi_L(y_i),v_j \rangle_Q = \langle y_i,v_j\rangle_Q = \delta_{ij}
\end{align*}
whenever $j \leq k$. Moreover, $y_i \in L^\perp$ for $i >k$ by construction. Thus, 
\[ \pi_L((\Z^n)^{\#}) = \pi_L(\spn_{\Z}(y_1, \ldots, y_n)) = \spn_{\Z}(w_1, \ldots ,w_k) = L(\Z)^\#  \]
as claimed.
The proof of the second equality is analogous.

For \ref{item:ElemProp ProjvsDualA 2}, note that for any two lattices $\Lambda_1 \subset \Lambda_2$ in $L$ one has
\begin{align}\label{eq:quotduals}
\Lambda_2/\Lambda_1 \simeq \Lambda_1^\# /\Lambda_2^\#,
\end{align}
so \ref{item:ElemProp ProjvsDualA 2} follows from \ref{item:ElemProp ProjvsDualA 1}.
To construct such an isomorphism one proceeds as follows. Fix a basis $v_1, \ldots, v_n$ of $\Lambda_2$ such that $d_1v_1, \ldots, d_nv_n$ is a basis\footnote{Such a basis is sometimes called 'adapted basis' (in geometry of numbers). The existence can be easily seen using Smith's normal form.} of $\Lambda_1$ with $d_i \in \Z$ and let $w_1, \ldots, w_n$ be the dual basis to $v_1, \ldots, v_n$. Then, the map
\[ f: \Lambda_2 \rightarrow \Lambda_1^\#, \ v_i \mapsto d_i^{-1}w_i \]
induces the desired isomorphism.
\end{proof}

\begin{Prop}
\label{Prop:Isomorphism between the quotients of piL(Rn) and piL perp (Rn)}
We have an isomorphism
\[ \pi_L(\Z^n)/L(\Z) \rightarrow \pi_{L^{\perp}}(\Z^n)/L^{\perp}(\Z). \]
\end{Prop}

When $Q$ is unimodular, i.e.~$\disc(Q) = 1$, this together with Lemma~\ref{lem:imageoforthproj} shows that the glue-groups of $L$ and $L^\perp$ are isomorphic. 
Indeed, in this case $(\Z^n)^\# = \Z^n$ and hence $\pi_L(\Z^n) = L(\Z)^\#$.
In particular, $L$ and $L^\perp$ have the same discriminant.
When $Q$ is not unimodular, the proposition gives an isomorphism between subgroups of the respective glue-groups.

\begin{proof}
We define a map $f$ from $\pi_L(\Z^n)$ to $\pi_{L^{\perp}}(\Z^n)/L^{\perp}(\Z)$ as follows.
For $x \in \pi_L(\Z^n)$ choose a lift $\hat{x} \in \Z^n$ of $x$ for the projection $\pi_L$ and define 
\begin{align*}
f(x) = \pi_{L^{\perp}}(\hat{x}) + L^\perp(\Z).
\end{align*}
Note that $f$ is well-defined since if $\hat{x}, \hat{y} \in \Z^n$ are two lifts of $x\in \pi_L(\Z^n)$, then $\hat{x} - \hat{y} \in L^{\perp}(\Z)$ which implies $\pi_{L^{\perp}}(\hat{x}) + L^\perp(\Z) = \pi_{L^{\perp}}(\hat{y}) + L^\perp(\Z)$.

We show that $\ker(f) = L(\Z)$.
Obviously, $L(\Z) \subset \ker(f)$ since for any $x \in L(\Z)$ we can choose $x$ itself as lift.
On the other hand, if $x \in \ker(f)$ there is a lift $\hat{x} \in \Z^n$ of $x$ for $\pi_L$ such that $\pi_{L^{\perp}}(\hat{x}) \in L^{\perp}(\Z)$. 
In particular, 
\[ x = \pi_L(\hat{x}) = \pi_L(\hat{x}) - \pi_L(\pi_{L^{\perp}}(\hat{x})) = \pi_L(\hat{x} -\pi_{L^{\perp}}(\hat{x})) = \hat{x} -\pi_{L^{\perp}}(\hat{x})  \in L(\Z) .\]
We deduce that $\ker(f) \subset L(\Z)$ and hence equality. This proves the proposition. 
\end{proof}

\begin{proof}[Proof of Proposition~\ref{prop:disccomparison}]
By Proposition~\ref{Prop:Isomorphism between the quotients of piL(Rn) and piL perp (Rn)}
\begin{align*}
\disc_Q(L) = |\mathcal{G}(L)| &= |L(\Z)^\#/\pi_L(\Z^n)| \cdot |\pi_L(\Z^n)/L(\Z)|\\
&= |L(\Z)^\#/\pi_L(\Z^n)| \cdot |\pi_{L^\perp}(\Z^n)/L^\perp(\Z)|\\
&= \frac{|L(\Z)^\#/\pi_L(\Z^n)|}{|L^\perp(\Z)^\#/\pi_{L^\perp}(\Z^n)|} |\mathcal{G}(L^\perp)|.
\end{align*}
Using Lemma~\ref{lem:imageoforthproj} note that the finite group $L(\Z)^\#/\pi_L(\Z^n) = \pi_L((\Z^n)^\#)/\pi_L(\Z^n)$ is a quotient of $(\Z^n)^\#/\Z^n$ and hence $|L(\Z)^\#/\pi_L(\Z^n)|$ is a divisor of $\disc(Q)=|(\Z^n)^\#/\Z^n|$.
As the analogous statement holds for $L^\perp$, the proposition follows.
\end{proof}

\begin{remark}
    When $\disc(Q) = 1$, Proposition~\ref{Prop:Isomorphism between the quotients of piL(Rn) and piL perp (Rn)} states that $\mathcal{G}(L) \simeq \mathcal{G}(L^\perp)$.
    Apart from the discriminants of $L$ and $L^\perp$ being the same, this includes information about the local coefficients of the quadratic forms on $L$ and $L^\perp$. This is exploited e.g.~in Proposition~\ref{prop:primitivity}.
    When $k=n-k$, one can ask whether this implies that $Q|_{L(\Z)}$ and $Q|_{L^\perp(\Z)}$ are in the same genus.
\end{remark}

%% file: Sections/moduli.tex
\section{Moduli spaces}\label{sec:moduli}

In this section we study the moduli space $\Grids$ of basis extensions which was introduced in \S\ref{sec:modulispacetheorem} consisting of (certain) homothety classes $[L,\Lambda]$ where $L$ is a $k$-dimensional subspace, $\Lambda$ is a full rank lattice in $\R^n$ and  $L \cap \Lambda$ is a lattice in $L$.
We also discuss a slight refinement of
Theorem~\ref{thm:equidistrmoduli} (Theorem~\ref{thm:equidistrmoduli-better} below) and see how it implies Theorem~\ref{thm:jointwithshapes}.
%
%
%

\subsection{Oriented subspaces}\label{sec:orientedsubsp}

For the purposes of proving the main theorems from their dynamical analogues, it is convenient to work with subspaces with an orientation.
In fact, the main theorems may be refined to include orientation.

\newcommand{\pure}{\mathcal{P}}
Oriented subspaces of dimension $k$ form an affine variety $\Gror{k}$ (defined over $\Q$) with a morphism (of algebraic varieties) $\Gror{k} \to \Gr{k}$ where the preimage of any point consists of two points corresponding to two choices of orientation.

\begin{remark}
    To construct $\Gror{k}$ explicitly, observe that the positive definite form $Q$ induces a rational form $\disc_Q$ on the exterior product $\bigwedge^k \Q^n$ via
\begin{align*}
\disc_Q(v_1\wedge \ldots \wedge v_k) = \det \begin{pmatrix}
\langle v_1,v_1\rangle_Q & \cdots & \langle v_1,v_k\rangle_Q \\
\vdots & & \vdots \\
\langle v_k,v_1\rangle_Q & \cdots & \langle v_k,v_k\rangle_Q
\end{pmatrix}.
\end{align*}
Note that this merely extends the previous definition of discriminant.
The variety $\Gror{k}$ is then as the subvariety of the variety of pure wedges $\pure$ satisfying the additional equation $\disc_Q(v_1\wedge \ldots \wedge v_k) = 1$.
Note that rational subspaces with an orientation do \emph{not} correspond to rational points of $\Gror{k}$ but rather to primitive integer points of the variety of pure wedges $\mathcal{P}$.
In that sense, it is often more natural to work with $\mathcal{P}$ instead of $\Gror{k}$.
\end{remark}

The orthogonal group $\SO_Q$ (and hence also $\Spin_Q$) acts on oriented subspaces. For an oriented rational subspace $L$ the stabilizer group in $\Spin_Q$ under this action is exactly equal to the stabilizer group $\StSpin_L$ defined in \S\ref{sec:stabilizersintro}.
Moreover, the action of $\Spin_Q(\R)$ on $\Gror{k}(\R)$ is transitive (as the action of $\SO_Q(\R)$ is).

\begin{remark}[Orientation on the orthogonal complement]
For any oriented $k$-dimensional subspace $L$ the orthogonal complement inherits an orientation: if $v_1,\ldots,v_k$ is an oriented basis of $L$ then a basis $v_{k+1},\ldots,v_n$ of $L^\perp$ is oriented if $\det(v_1,\ldots,v_n) >0$. The orthogonal complement yields an isomorphism $\Gror{k} \to \Gror{n-k}$ which is explicitly realizable in Pl\"ucker coordinates at least when $\disc(Q) = 1$ \cite[\S1]{Schmidt-heights}.
\end{remark}

\subsection{Quotients of homogeneous spaces}\label{sec:homspaces}

\subsubsection{The moduli space of oriented basis extensions}\label{sec:modulioriented}
We extend the definition of the moduli space of basis extensions to include orientation.
Consider the pairs $(L,\Lambda)$ where $L$ is an oriented subspace, $\Lambda \subset \R^n$ is a full rank lattice, and $L \cap \Lambda$ is a lattice in $L$.
Two such pairs $(L,\Lambda),(L',\Lambda')$ are equivalent if $L=L'$ (including orientation) and if there exists $g \in \GL_n(\R)$ which acts by positive scalar multiplication of $L$ and $L^\perp$ such that $g \Lambda = \Lambda'$. 
The moduli space of oriented basis extensions $\Gridsor$ is defined to be the set of such equivalence classes $[L,\Lambda]$.
There exists a natural map $\Gridsor \to \Grids$ (simply by forgetting orientation). 

We begin by realizing $\Gridsor$ as a double quotient of a Lie group.
We use the following notation:
\begin{itemize}
\item The groups $\SPnk$ and $\SG$ as defined in \S\ref{sec:groupsdef}:
\begin{align*}
\SPnk &= \Big\{ \begin{pmatrix}
A & B \\
0 & D
\end{pmatrix} \in \SL_n : \det(A) = \det(D)=1 \Big\},\\
\SG &= \Spin_Q \times \SPnk.
\end{align*}
\item The reference subspace $L_0$ spanned by the first $k$ standard basis vectors \eqref{eq:refsubspace} as well as the 'standardization' $\gQ$ defined in \eqref{eq:defetaQ}. Note that $L_0$ is oriented using the standard basis.
\item For any oriented subspace $L \subset \Q^n$ we let $\StSpin_L<\Spin_Q$ be the stabilizer group of $L$.
\item The subgroup $\StSpin_{L_0}<\Spin_Q$ maps to a subgroup of $\SPnk$ under the (spin) isogeny $\isog$; we again denote by $\StDiag_{L_0}<\SG$ the diagonally embedded group (this agrees with the definition in \S\ref{sec:diagemb} with the choice of the standard basis).
\end{itemize}

\begin{Lem}\label{lem:ident for grids}
There is an identification
\begin{align*}
\Gridsor \simeq \lrquot{\StDiag_{L_0}(\R)}{\SG(\R)}{\SPnk(\Z)}.
\end{align*}
\end{Lem}

By Lemma~\ref{lem:ident for grids}, we may pull back the Haar quotient probability measure on the right-hand side to $\Gridsor$ (and by push-forward on $\Grids$).

\begin{proof}
The above identification runs as follows. If $(g_1,g_2)\in \SG(\R)$ is given, we set $L = \isog(g_1^{-1})g_2L_0(\R) = g_1^{-1}.L_0(\R)$ and $\Lambda =\isog(g_1^{-1})g_2 \Z^n$.
Clearly, $\Lambda$ intersects $L$ in the lattice $\isog(g_1^{-1})g_2L_0(\Z)$.
As any element of $\SPnk(\Z)$ stabilizes $L_0(\R)$ and $\Z^n$ and as $\StDiag_{L_0}(\R)$ is diagonally embedded, we obtain a well-defined map
\[  \lrquot{\StDiag_{L_0}(\R)}{\SG(\R)}{\SPnk(\Z)} \rightarrow \Grids.\]
The injectivity of this map is clear from the definition of $\StDiag_{L_0}(\R)$, so let us argue for the surjectivity.

Let $[L,\Lambda] \in \Grids$. By choosing the representative correctly, we may assume that $\Lambda$ as well as $L \cap \Lambda$ are unimodular.
Choose $g_1 \in \Spin_Q(\R)$ such that $g_1.L = L_0$. Then $L_0(\R)$ is $g_1.\Lambda$-rational.
Pick a basis $v_1,\ldots,v_k$ of $g_1.\Lambda \cap L_0(\R)$ and complete it into a basis $v_1,\ldots,v_n$ of $g_1.\Lambda$. Set 
\begin{align*}
g_2 = (v_1\mid\ldots\mid v_n) \in \{g \in \SL_n(\R): gL_0(\R) = L_0(\R)\}.
\end{align*}
As $g_1.\Lambda \cap L_0(\R)$ is unimodular, we have that $g_2 \in \SPnk(\R)$.
Under these choices we have $\isog(g_1^{-1})g_2L_0(\R) = L$ and $\isog(g_1^{-1})g_2 \Z^n = \Lambda$; surjectivity follows.
\end{proof}

\begin{remark}[Action of $\Spin_Q(\Z)$]\label{rem:action SpinZ}
Note that $\Spin_Q(\Z)$ acts on $\Gridsor$ via $g[L,\Lambda] = [g.L,g.\Lambda]$.
In view of the identification in Lemma~\ref{lem:ident for grids} (and its proof) this action of $\Spin_Q(\Z)$ corresponds to the $\Spin_Q(\Z)$-action from the right on the double quotient $\lrquot{\StDiag_{L_0}(\R)}{\SG(\R)}{\SPnk(\Z)}$.
In particular,
\begin{align*}
\lquot{\Spin_Q(\Z)}{\Gridsor}
\simeq \lrquot{\StDiag_{L_0}(\R)}{\SG(\R)}{\SG(\Z)}.
\end{align*}
\end{remark}

Recall from the introduction that $\shapesp{k}$ is the space of positive definite real quadratic forms in $k$ variables up to similarity. 
Here, we say that two forms $q,q'$ in $k$-variables are \emph{equivalent} if there is $g \in \GL_k(\Z)$ such that $gq = q'$
and \emph{similar} if $q$ is equivalent to a multiple of $q'$.
We may identify $\shapesp{k}$ with
\begin{align}\label{eq:ident shapesp}
\lrquot{\Orth_k(\R)}{\PGL_k(\R)}{\PGL_k(\Z)}.
\end{align}
Indeed, to any point $\Orth_k(\R)g \PGL_k(\Z)$ one associates the similarity class of the form represented by $g^tg$.
Conversely, given the similarity class of a form $q$ and a matrix representation $M$ of $q$ one can write $M = g^t g$ for some $g \in \GL_k(\R)$.
Another way to view the quotient in \eqref{eq:ident shapesp} is as the space of lattices in $\R^k$ up to isometries and homothety. For a lattice $\Gamma \subset \R^k$, we denote by $\langle\Gamma\rangle$ its equivalence class.
The map 
\begin{equation}
\label{eq: map from lattice viewpoint to qform viewpoint}
\langle\Gamma\rangle \mapsto [Q_0|_{\Gamma}]
\end{equation}
is the desired bijection. In words, the class of lattices $\langle\Gamma\rangle$ is associated to the similarity class of the standard form $Q_0$ represented in a basis of the lattice $\Gamma$.

Note that we have a map $[L,\Lambda] \in \Grids \mapsto [Q|_{L \cap \Lambda}] \in \shapesp{k}$ already alluded to in the introduction.
It is natural to ask what equivalence class of lattices corresponds to the similarity class (or shape) $[Q|_{L\cap \Lambda}]$ from the introduction under the identification~\eqref{eq: map from lattice viewpoint to qform viewpoint}.
To answer this question, choose a rotation $k_L \in \SO_Q(\R)$ with $k_L L(\R) = L_0(\R)$.
Apply $\gQ$ to the lattice $k_L (L\cap \Lambda) \subset L_0(\R)$.
Recall that $\gQ$ was chosen in \S\ref{sec:quadratic forms} to preserve $L_0(\R)$ so that $\gQ k_L (L \cap \Lambda) \subset L_0(\R)$.
Since 
\begin{align*}
Q_0|_{\gQ k_L (L\cap\Lambda)} \simeq Q|_{L\cap \Lambda},
\end{align*}
the equivalence class of the lattice $\gQ k_L (L \cap \Lambda)$ corresponds to the similarity class or shape $[Q|_{L\cap \Lambda}]$.
As we did in the introduction, we will also write $[L\cap \Lambda]$ for that shape.

\begin{Lem}\label{lem:torus bundle}
There is a surjective map
\begin{align*}
\Gridsor \to \Grk(\R) \times \shapesp{k} \times \shapesp{n-k}
\end{align*}
given explicitly by $[L,\Lambda] \mapsto (L, [L\cap \Lambda],[L^\perp\cap \Lambda^\#])$.
Moreover, the pushforward of the Haar (quotient) probability measure is the Haar probability measure on the target.
\end{Lem}

\begin{proof}
Recall that $\StSO_{L_0}$ is the stabilizer of $L_0$ in $\SO_Q$. Over $\R$, we have $\StSO_{L_0}(\R) = \isog(\StSpin_{L_0}(\R))$.
Consider the (surjective) composition
\begin{align*}
\Gridsor &\to \lrquot{\StDiag_{L_0}(\R)}{\SG(\R)}{\SPnk(\Z)}\\
&\to \lquot{\StSpin_{L_0}(\R)}{\Spin_Q(\R)} \times
\lrquot{\StSO_{L_0}(\R))}{\SPnk(\R)}{\SPnk(\Z)}\\
&\to \lquot{\StSpin_{L_0}(\R)}{\Spin_Q(\R)} \times
\lrquot{\gQ\StSO_{L_0}(\R)\gQ^{-1}}{\SPnk(\R)}{\SPnk(\Z)}
\end{align*}
where the first map is the identification in Lemma~\ref{lem:ident for grids}, the second map is the quotient map and the third map is multiplication by $\gQ$ in the second factor.
First, observe that $\lquot{\StSpin_{L_0}(\R)}{\Spin_Q(\R)}$ is identified with $\Gror{k}(\R)$ via $\StSpin_{L_0}(\R)g_0 \mapsto g_0^{-1}.L_0(\R)$.
Note also that $\gQ\StSO_{L_0}(\R))\gQ^{-1}$ is equal to the group $\SO_k(\R) \times \SO_{n-k}(\R)$ embedded block-diagonally. We apply projections onto the blocks ($\pi_1,\pi_2$ defined in \S\ref{sec:groupsdef}) as well as inverse-transpose in the second block to obtain a surjective map 
\begin{align*}
\lrquot{\gQ\StSO_{L_0}(\R)\gQ^{-1}}{\SPnk(\R)}{\SPnk(\Z)}
\to \shapesp{k} \times \shapesp{n-k}.
\end{align*}
Overall, we have a surjection $\phi:\Gridsor \to \Gr{k}(\R) \times \shapesp{k} \times \shapesp{n-k}$.

It remains to verify that this surjection is the map from the lemma.
Let $[L,\Lambda] \in \Gridsor$ and let $(g_1,g_2) \in \SG(\R)$ be a representative of its double coset in Lemma~\ref{lem:ident for grids}. 
It is clear from the proof of Lemma~\ref{lem:ident for grids} that $\phi([L,\Lambda])_1 = g_1^{-1}.L_0(\R) = L(\R)$.
For the second component, note that using $g_1^{-1}.L_0(\R) = L(\R)$
\begin{align*}
[Q|_{L \cap \Lambda}]
= [Q_0|_{\gQ \isog(g_1)(L\cap \Lambda)})]
= [Q_0|_{\gQ(L_0 \cap g_2\Z^n)}]
= [Q_0|_{\pi_1(\gQ g_2) \Z^k}] = \phi([L,\Lambda])_2.
\end{align*}
For the third component, we first observe that $L^\perp(\R) = g_1^{-1}.L_0(\R)^\perp$ as well as $\Lambda^\# = \isog(g_1^{-1})(g_{2}^{-1})^t\Z^n$.
Hence, 
\begin{align*}
[Q|_{L^\perp \cap \Lambda^\#}]
&= [Q_0|_{\gQ \isog(g_1)(L^\perp \cap \Lambda^\#)}]
= [Q_0|_{\gQ (L_0^\perp \cap (g_{2}^{-1})^t\Z^n})]
= [Q_0|_{\pi_2(\gQ (g_{2}^{-1})^t) \Z^k}] \\
&= \phi([L,\Lambda])_3
\end{align*}
which concludes the lemma.
\end{proof}

\subsection{A construction of an intermediate lattice}

As was already observed in Remark~\ref{rem:comparisonconj}, equidistribution of the tuples $[L,\Z^n]$ for $L \in \HkQD$ (Conjecture~\ref{conj:equidistrmoduli}) does not necessarily imply equidistribution of the tuples $(L,[L(\Z)],[L^\perp(\Z)])$ when $Q$ is not unimodular (Conjecture~\ref{conj:ideal equidistribution result}).
Indeed, one can see from Lemma~\ref{lem:torus bundle} that it implies equidistribution of the tuples $(L,[L(\Z)],[L^\perp\cap (\Z^n)^\#])$ for $L \in \HkQD$.
Here, we construct for every $L$ a full rank sublattice $\Lambda_L\subset \Q^n$ so that equidistribution of the tuples $[L,\Lambda_L]$ does have this desired implication.
For any subspace $L \subset \Q^n$ write $\pi_L$ for the orthogonal projection onto $L$.

\begin{Prop}\label{prop:constructionLambda}
For any subspace $L \in \Gr{k}(\Q)$ there exists a full rank $\Z$-lattice $\Lambda_L \subset \Q^n$ with the following properties:
\begin{enumerate}
\item\label{item:Lambda1} $\Z^n \subset \Lambda_L \subset (\Z^n)^\#$
\item\label{item:Lambda2} We have
\begin{align*}
L \cap \Lambda_L = L(\Z), \quad 
\pi_{L^\perp}(\Lambda_L) = L^\perp(\Z)^\#, \quad \text{and} \quad
L^\perp(\Z) = \Lambda_L^\# \cap L^\perp.
\end{align*}
\item\label{item:Lambda3} Suppose that $L'$ satisfies that there are $\gamma \in \Spin_Q(\Q)$ and $k_p \in \Spin_Q(\Z_p)$ for every prime $p$ such that $\gamma. L = L'$ and $k_p. L(\Z_p) = L'(\Z_p)$. Then
\begin{align*}
\Lambda_{L'} = \bigcap_{p} k_p.(\Lambda_L \otimes \Z_p) \cap \Q^n.
\end{align*}
\end{enumerate}
\end{Prop}

We remark that if $Q$ is unimodular, one may simply take $\Lambda_L = \Z^n$.
For $Q$ not unimodular, this choice generally satisfies \eqref{item:Lambda1} and \eqref{item:Lambda3} but not necessarily \eqref{item:Lambda2}.

\begin{remark}[Equivalence relation]\label{rem:equivrelation1}
We write $L \sim L'$ for rational subspaces $L,L'$ of dimension $k$ if there are $\gamma \in \Spin_Q(\Q)$ and $k_p \in \Spin_Q(\Z_p)$ for every prime $p$ such that $\gamma. L = L'$ and $k_p .L(\Z_p) = L'(\Z_p)$.
This defines an equivalence relation.
As $L,L'$ are locally rotated into each other, they have the same discriminant (see Equation (\ref{eq:disc loctoglob})).
\end{remark}

\begin{proof}[Proof of Proposition~\ref{prop:constructionLambda}]
In view of Remark~\ref{rem:equivrelation1} and the required property in \eqref{item:Lambda3} we first observe that if $L'$ is equivalent to $L$ and $L$ satisfies \eqref{item:Lambda1},\eqref{item:Lambda2} then $L'$ also does so. 
We may hence split $\Gr{k}(\Q)$ into equivalence classes, choose a representative $L$ in each equivalence class, and construct $\Lambda_L$ with the properties in \eqref{item:Lambda1} and \eqref{item:Lambda2} ignoring \eqref{item:Lambda3}.

So let $L \in \Gr{k}(\Q)$ be such a representative.
Choose a basis $v_1,\ldots,v_k$ of $L(\Z)$. We consider the $\Z$-module $\rquot{(\Z^n)^\#}{L(\Z)}$ which fits into the following exact sequence
\begin{align}\label{eq:shortexact}
0 \to \rquot{L \cap (\Z^n)^\#}{L(\Z)} \to \rquot{(\Z^n)^\#}{L(\Z)} \to \rquot{(\Z^n)^\#}{L \cap (\Z^n)^\#} \to 0.
\end{align}
As $L \cap (\Z^n)^\#$ is primitive\footnote{A sublattice $\Gamma$ of a lattice $\Lambda \subset \Q^n$ is primitive if it is not strictly contained in any sublattice of the same rank.} in $(\Z^n)^\#$, the module on the very right is free of rank $n-k$.
We choose a basis of it as well as representatives $v_{k+1},\ldots,v_n \in (\Z^n)^\#$ of these basis elements.
Define
\begin{align*}
\Lambda_L = \Z v_1 + \ldots + \Z v_n.
\end{align*}
It is not very hard to see that this lattice contains $\Z^n$ and is contained in $(\Z^n)^\#$ so that \eqref{item:Lambda1} is satisfied.

Suppose that 
\begin{align*}
v = \sum_i \alpha_i v_i \in L \cap \Lambda_L.
\end{align*}
This implies that $\sum_{i>k}\alpha_i v_i \in L$ and so $\sum_{i>k}\alpha_i v_i = 0$ by linear independence. 
The identity $L \cap \Lambda_L = L(\Z)$ follows.

By Lemma~\ref{lem:imageoforthproj}, the projection $\pi_{L^\perp}:(\Z^n)^\# \to L^\perp(\Z)^\#$ is surjective.
Clearly, the kernel is $L\cap (\Z^n)^\#$ and hence by construction of $\Lambda_L$ we have $\pi_{L^\perp}(\Lambda_L) = \pi_{L^\perp}((\Z^n)^\#) = L^\perp(\Z)^\#$.

It remains to prove the last identity.
As $\Lambda_L^\# \supset \Z^n$ we have $\Lambda_L^\# \cap L^\perp \supset L^\perp(\Z)$ so it suffices to show that
\begin{align*}
L^\perp(\Z)^\# = \pi_{L^\perp}(\Lambda_L) \subset (\Lambda_L^\# \cap L^\perp)^\#.
\end{align*}
For $v = \pi_{L^\perp}(v') \in \pi_{L^\perp}(\Lambda_L)$ and $w \in L^\perp \cap \Lambda_L^\#$ we have $\langle v,w\rangle = \langle v',w\rangle\in \Z$ proving the remaining claim.
\end{proof}

\begin{remark}
Observe that $\Lambda_L$ constructed above depends on the choice of basis for the free module $\rquot{(\Z^n)^\#}{L \cap (\Z^n)^\#}$ which forms the 'free part' of $\rquot{(\Z^n)^\#}{L(\Z)}$ in the sense of \eqref{eq:shortexact}.
But the short exact sequence \eqref{eq:shortexact} does not split in general so that the basis elements have no canonical lifts to $\rquot{(\Z^n)^\#}{L(\Z)}$; different choices yield different lattices $\Lambda_L$.
This dependency is inconsequential as the set of lattices $\Lambda$ with $\Z^n \subset \Lambda \subset (\Z^n)^\#$ is finite.
\end{remark}

\subsection{A refinement of Theorem \ref{thm:equidistrmoduli}}\label{sec:refinement moduli}
We now present a refinement of Theorem  \ref{thm:equidistrmoduli} which is necessary in order to deduce the desired equidistribution theorem of shapes (i.e.~Theorem \ref{thm:jointwithshapes}).

%
%

\begin{Thm}\label{thm:equidistrmoduli-better}
Let $k\geq 3$ with $k \leq n-k$ and let $p$ be a prime with $p \nmid2\disc(Q)$. 
Let $L \in \Grk(\Q) \mapsto \Lambda_L$ satisfy conditions \eqref{item:Lambda1} and \eqref{item:Lambda3} from Proposition~\ref{prop:constructionLambda}.
Suppose that $D_i \in \N$ is a sequence of integers with $\kfree{D_i}{k}\to \infty$, $\mathcal{H}_{Q}^{n,k}(D_i)\neq \emptyset$ as well $p \nmid D_i$ if $k \in \{3,4\}$.
Then the sets
\begin{align}\label{eq:orientedcoll}
\{([L,\Lambda_L]: L \subset \Q^n \text{ oriented }, \disc_Q(L) = D_i, \dim(L)=k\}
\end{align}
equidistribute in $\Gridsor$ as $i \to \infty$
\end{Thm}

We observe that the special case $\Lambda_L= \Z^n$ for every $L \in \Grk(\Q)$ in Theorem~\ref{thm:equidistrmoduli-better} implies Theorem~\ref{thm:equidistrmoduli} after projection $\Gridsor\to \Grids$.

\begin{proof}[Proof of Theorem~\ref{thm:jointwithshapes} from Theorem~\ref{thm:equidistrmoduli-better} when $k \geq 3$]
Let $\Lambda_L$ for $L\in \Gr{k}(\Q)$ be defined as in Proposition~\ref{prop:constructionLambda}.
Let $p$ be a prime and $D_i\geq 1$ be a sequence of discriminants as in Theorem~\ref{thm:jointwithshapes}.
Then Theorem~\ref{thm:equidistrmoduli-better} is applicable and the sets in \eqref{eq:orientedcoll} are equidistributed in $\Gridsor$ when $i \to \infty$.
By construction of $\Lambda_L$, the image of these sets under the map in Lemma~\ref{lem:torus bundle} is exactly
\begin{align*}
\{(L,[L(\Z)],[L^\perp(\Z)]): L \in \mathcal{H}_{Q}^{n,k}(D_i)\}.
\end{align*}
These images are equidistributed with respect to the pushforward measure, which is the Haar probability measure on $\Gr{k}(\R)\times \shapesp{k} \times \shapesp{n-k}$.
\end{proof}

\begin{remark}[Theorem~\ref{thm:jointwithshapes} for oriented subspaces]
Let $\mathcal{X}_k$ be the space of positive definite real quadratic forms in $k$ variables up to \emph{proper} similarity. 
Observe that the shape of an oriented $k$-dimensional subspace makes sense as a point in $\mathcal{X}_k$.
Very much related to this is the fact that the proof of Lemma~\ref{lem:torus bundle} actually establishes a surjective map $\Gridsor \to \Gror{k}(\R)\times \mathcal{X}_k \times \mathcal{X}_{n-k}$. Theorem~\ref{thm:jointwithshapes} may thus be generalized to this latter space. 
For $k=1$, this oriented version already appears in the works \cite{AES-dim3,AES-higherdim}.
\end{remark}

%% file: Sections/equidistr_moduli.tex
\section{Proof of the main theorems from the dynamical versions}\label{sec:proof arithmetic}

The aim of this section is to prove Theorem~\ref{thm:equidistrmoduli-better} and Theorem~\ref{thm:jointwithshapes} for $k=2$.
We remark that any possible future upgrades to the dynamical versions (in regard to the congruence conditions at fixed primes) imply the analogous upgrades to the arithmetic versions.

\subsection{Notation}
We recall and introduce here some notation used throughout this Section~\ref{sec:proof arithmetic}. In the following, $L \subset \Q^n$ is an arbitrary $k$-dimensional oriented subspace unless specified otherwise:
\begin{itemize}
\item $\Gridsor$ is the moduli space of oriented basis extensions defined in \S\ref{sec:modulioriented} (see also \S\ref{sec:modulispacetheorem}).
Recall that $\Spin_Q(\Z)$ acts on $\Gridsor$ via $g[L,\Lambda] = [g.L,g.\Lambda]$. Moreover, by Lemma~\ref{lem:ident for grids} and the subsequent Remark~\ref{rem:action SpinZ}, we have
\begin{align}\label{eq:recapLemmaId}
\Gridsor&\simeq \lrquot{\StDiag_{L_0}(\R)}{\SG(\R)}{\SPnk(\Z)},\\
\lquot{\Spin_Q(\Z)}{\Gridsor}&\simeq \lrquot{\StDiag_{L_0}(\R)}{\SG(\R)}{\SG(\Z)}
\end{align}
where $L_0 = \Q^k \times \{(0,\ldots,0)\} \subset \Q^n$ is the fixed reference subspace (cf.~\eqref{eq:refsubspace}) and $\SG = \Spin_Q \times \SPnk$ (cf.~\ref{sec:groupsdef}).
\item The subgroup $\StSpin_L < \Spin_Q$ is the identity component of the stabilizer group of $L$ (cf. \S\ref{sec:stabilizersintro} and see also \S\ref{sec:orientedsubsp}).
\item We fix a full-rank lattice $\Z^n \subset \Lambda_L \subset (\Z^n)^\#$ satisfying (1) and (3) in Proposition~\ref{prop:constructionLambda}.
The reader is encouraged to keep in mind the case $\disc(Q) = 1$ where one may take $\Lambda_L = \Z^n$ for all $L$.
\item We fix an oriented basis of $\Lambda_L$ where the first $k$ vectors are an oriented basis of $L\cap \Lambda_L$. Let $g_L \in \GL_n(\Q)$ be the element whose columns consist of this basis.
\item The subgroup $\StDiag_L < \SG$ is defined as in \S\ref{sec:diagemb} using the basis in $g_L$.
\item For any $[L,\Lambda] \in \Gridsor$ (where $L$ is not necessarily rational) we write to shorten notation $[L, \Lambda]_{\star}$ for the equivalence class $\Spin_Q(\Z)[L, \Lambda] \in \lquot{\Spin_Q(\Z)}{\Gridsor}$.
\item Let $\mathsf{s}_L \in \SG(\R)$ be the representative of the double coset of $[L,\Lambda_L]$ defined using $g_L$ (see also the proof of Lemma~\ref{lem:ident for grids}).
\item For any $D \in \N$ with $\HkQD \neq \emptyset$ we consider the finite set $\RkD\subset \Gridsor$ consisting of classes $[L,\Lambda_L]$ where $L$ runs over all oriented $k$-dimensional subspaces $L \subset \Q^n$ with $\disc_Q(L)=D$ -- see also \eqref{eq:orientedcoll}.
The action of $\Spin_Q(\Z)$ on $\Gridsor$ leaves $\RkD$ invariant.
\end{itemize}

\subsection{Outline of the proof}
Let $\mathcal{U} = \SG(\R\times \widehat{\Z})\SG(\Q)  \subset \rquot{\SG(\A)}{\SG(\Q)}$ be the principal genus\footnote{The genera (i.e.~orbits of $\SG(\R)\times \SG(\widehat{\Z})$) correspond to classes in the spinor genus of $Q$. Recall here that if $Q$ is the sum of squares in $\leq 8$ variables, then the spinor genus consists of one class (cf.~\cite[p.232]{Cassels}) and hence $\mathcal{U} = \lquot{\SG(\A)}{\SG(\Q)}$.}.
There is a natural map
\begin{align}\label{eq:projectionontoreal}
\rquot{\SG(\A)}{\SG(\Q)} \supset \mathcal{U} \to \lquot{\Spin_Q(\Z)}{\Gridsor}
\end{align}
given by taking the quotient on the left of $\SG(\A)/\SG(\Q)$ by the maximal compact open subgroup $\SG(\widehat{\Z})$ and $\StDiag_{L_0}(\R)$.
Consider now an oriented subspace $L$ of discriminant $D$ and the orbit $\mathsf{s}_L \StDiag_L(\A)\SG(\Q)$.
For any $L \in \HkQD$, the image of the intersection of $\mathsf{s}_L \StDiag_L(\A)\SG(\Q)$ with $\U$ under \eqref{eq:projectionontoreal} is a subset of the collection $\lquot{\Spin_Q(\Z)}{\RkD}$ and contains $[L,\Lambda_L]$ -- see Proposition~\ref{Proposition: packets are in RkD}. 
In other words, we have a commutative diagram
\begin{center}
\begin{tikzcd}
    \mathsf{s}_L \StDiag_L(\A)\SG(\Q)\cap \U \arrow[hookrightarrow]{r}{} \arrow{d}{} & \U \arrow{d}{} \\
   \lquot{\Spin_Q(\Z)}{\RkD} \arrow[hookrightarrow]{r}  & \lquot{\Spin_Q(\Z)}{\Gridsor}.
  \end{tikzcd}
\end{center}
Assuming here $k\geq 3$, the intersection $\mathsf{s}_L \StDiag_L(\A)\SG(\Q)\cap \U$ is equidistributed in $\U$ with respect to the normalized restriction of the Haar measure (along any sequence of admissible subspaces).
This immediately implies equidistribution of the pushforwards under the map in \eqref{eq:projectionontoreal}.

It remains to compare the pushforward of the Haar measure on the orbit to the measure on $\lquot{\Spin_Q(\Z)}{\RkD}$ induced by the normalized counting measure on $\RkD$. (This technical argument constitutes a large part of this section \S\ref{sec:proof arithmetic}.)
To this end, we first note that the projection $\Pack(L)$ of $\mathsf{s}_L \StDiag_L(\A)\SG(\Q)\cap \U$ is not surjective but $\lquot{\Spin_Q(\Z)}{\RkD}$ may be decomposed into such images for different subspaces $L$ -- see Remark~\ref{rem:equivrelation2}.
Thus, it is enough to determine the weights of individual points in $\Pack(L)$ -- see Lemmas~\ref{Lemma: computation of fibers of Phi} and \ref{lem:weights1}.

\subsection{Generating integer points from the packet}\label{sec:generating}
As a first step towards the proof of Theorem \ref{thm:equidistrmoduli-better}, we illustrate a general technique for generating points in $\RkD$ from a given point in $\RkD$.
This kind of idea appears in many recent or less recent articles in the literature -- see for example \cite[Thm.~8.2]{platonov}, \cite{localglobalEV}, \cite{AES-dim3}, \cite{AES-higherdim}, and \cite{2in4}.

For $g \in \SG = \Spin_Q \times \SPnk$ we write $g=(g_1,g_2)$ where $g_1$ is the first (resp.~$g_2$ is the second) coordinate of $g$.
Consider the open subset (principal genus)
\begin{align*}
\U = \SG(\R\times \widehat{\Z}) \SG(\Q) \subset \rquot{\SG(\A)}{\SG(\Q)}
\end{align*}

On $\U$, there is a projection map
\begin{equation}
\label{Equation: map Phi in introduction}
\Phi:\U \to \rquot{\SG(\R)}{\SG(\Z)} \to \lrquot{\StDiag_{L_0}(\R)}{\SG(\R)}{\SG(\Z)}
\simeq \lquot{\Spin_Q(\Z)}{\Gridsor}
\end{equation}
where the first map takes for any point $x \in \U$ a representative in $\SG(\R\times \widehat{\Z})$ and projects onto the real component.
Note that the first map is clearly $\SG(\R)$-equivariant.
For $L \in \Grk(\Q)$ we define
\begin{align}\label{eq:defP(L)}
\Pack(L) := \Phi(\mathsf{s}_L \StDiag_L(\A)\SG(\Q) \cap \U).
\end{align}

\begin{Prop}
\label{Proposition: packets are in RkD}
For any oriented $k$-dimensional subspace $L\subset \Q^n$ of discriminant $D$ we have
\[ \Pack(L) \subset \Spin_Q(\Z)\setminus \RkD. \]
\end{Prop}

\begin{proof}
Fix a coset $b\SG(\Q) \in \StDiag_L(\A)\SG(\Q) \cap \U$ and a representative $b = (b_1, b_2) \in \SG(\R \times \widehat{\Z})$.
By definition of $\Phi$
\begin{align*}
\Phi(\mathsf{s}_L b \SG(\Q)) = \StDiag_{L_0}(\R)\mathsf{s}_L b_\infty \SG(\Z).
\end{align*}
Note that since $b\G(\Q) \in \StDiag_L(\A)\G(\Q)$ there exists $h \in \StDiag_L(\A)$ and $\gamma \in \SG(\Q)$ such that $b= h \gamma$. 
By definition of $\StDiag_L$ we have $h_2 = g_L^{-1}\isog(h_1)g_L$.
We first show that the point in $\lquot{\Spin_Q(\Z)}{\Gridsor}$ corresponding to $\Phi(\mathsf{s}_L b \SG(\Q))$ lies above a rational subspace under the natural map $\Grids \to \Gror{k}(\R)$.
Note that by definition of the maps in \eqref{eq:recapLemmaId} the subspace attached to $\Phi(\mathsf{s}_L b \G(\Q))$ is $\isog(b_{1,\infty}^{-1})\rho_L^{-1}L_0 = b_{1,\infty}^{-1}.L$.
But
\begin{align}\label{eq: Proof packet is in RkD 1}
b_{1,\infty}^{-1}.L= \gamma_1^{-1} h_{1,\infty}^{-1}.L = \gamma_1^{-1}.L\subset \Q^n.
\end{align}
Next, we show that $\gamma_1^{-1}.L$ has discriminant $D$. To this end, note that by an analogous argument as in \eqref{eq: Proof packet is in RkD 1} for a prime $p$ we have $b_{1,p}^{-1}.L = \gamma_1^{-1}.L$ so that
\begin{align*}
\disc_{p,Q}(L) = \disc_{p,Q}(b_{1,p}^{-1}.L) = \disc_{p,Q}(\gamma_1^{-1}.L)
\end{align*}
where we used that $b_{1,p} \in \Spin_Q(\Z_p)$ preserves the local discriminant at $p$. Thus, $\disc_Q(\gamma_1^{-1}.L) = D$ by \eqref{eq:disc loctoglob}.

It remains to show that $\Phi(\mathsf{s}_L b \G(\Q))$ corresponds to $[\gamma_1^{-1}.L,\Lambda_{\gamma_1^{-1}.L}]_{\star}$.
For this, notice first that under \eqref{eq:recapLemmaId}
\begin{align*}
\Phi(\mathsf{s}_L b \G(\Q))
= [\gamma_1^{-1}.L, \isog(b_{1,\infty}^{-1})g_L b_{2,\infty} \Z^n]_{\star}
\end{align*}
by definition of the equivalence relation. Now,
\begin{align*}
\isog(b_{1,\infty}^{-1})g_L b_{2,\infty} =\isog(\gamma_1^{-1}h_1^{-1})g_L h_2 \gamma_2 = \isog(\gamma_1^{-1})g_L \gamma_2.
\end{align*}
Quite analogously, we have $\isog(\gamma_1^{-1})g_L \gamma_2 = \isog(b_{1,p}^{-1})g_L b_{2,p}$ so that
\begin{align*}
\isog(\gamma_1^{-1})g_L \gamma_2\Z_p^n 
= \isog(b_{1,p}^{-1})g_L \Z_p^n = b_{1,p}^{-1}.(\Lambda_L\otimes \Z_p).
\end{align*}
This shows that
\begin{align*}
\isog(\gamma_1^{-1})g_L \gamma_2\Z^n = \bigcap_p (\isog(\gamma_1^{-1})g_L \gamma_2 \Z_p^n) \cap \Q^n
= \bigcap_p b_{1,p}^{-1}.(\Lambda_L\otimes \Z_p) \cap \Q^n
= \Lambda_{\gamma_1^{-1}.L}
\end{align*}
by the third property of $\Lambda_L$ in Proposition~\ref{prop:constructionLambda}.
This shows that 
\begin{align*}
\Phi(\mathsf{s}_L b \G(\Q)) = [ \gamma_1^{-1}.L,\Lambda_{\gamma_1^{-1}.L}]_{\star}
\end{align*}
and hence the proposition follows.
\end{proof}

\begin{remark}[Equivalence class induced by packets]\label{rem:equivrelation2}
Note that for any two $L,L'$ of discriminant $D$ the sets $\Pack(L),\Pack(L')$ are either equal or disjoint. Indeed, these sets are equivalence classes for an equivalence relation which is implicitly stated in the proof of Proposition~\ref{Proposition: packets are in RkD}; see also Remark~\ref{rem:equivrelation1}.
\end{remark}

We analyze the fibers of the map $\Phi$ when restricted to the piece of the homogeneous set $\mathsf{s}_L \StDiag_L(\A)\G(\Q)$ in the open set $\U$.
We set for any $L\in \Grk(\Q)$
\begin{align*}
\bStDiagcpt_L = \{h \in \StDiag_L(\A): h_1 \in \StSpin_L(\R \times \widehat{\Z})\}.
\end{align*}
We remark that $\bStDiagcpt_L$ is not equal to $\StDiag_L(\R \times \widehat{\Z})$ as $g_L$ can have denominators (cf.~\eqref{eq:defisogenystab}).

\begin{Lem}
\label{Lemma: computation of fibers of Phi}
Let $x,y \in \StDiag_L(\A)\G(\Q) \cap \U$. Then
\[ \Phi(\mathsf{s}_L x) = \Phi(\mathsf{s}_L y) 
\iff y \in \bStDiagcpt_L x.\]
\end{Lem}

\begin{proof}
We fix representatives $b^x \in \SG(\RZ)$ of $x$ and $b^y \in \SG(\RZ)$ of $y$.
Moreover, we write $b^x = h^x\gamma^x$ and $b^y = h^y \gamma^y$ with $h^x,h^y \in \StDiag_L(\A)$ and $\gamma^x,\gamma^y \in \SG(\Q)$. 
The direction ''$\Leftarrow$'' is straightforward to verify; we leave it to the reader.

Assume that $\Phi(\mathsf{s}_L x) = \Phi(\mathsf{s}_L y)$.
We recall from Proposition~\ref{Proposition: packets are in RkD} and its proof that
\begin{align*}
\Phi(\mathsf{s}_Lx) = [(\gamma_1^x)^{-1}.L,\Lambda_{(\gamma_1^x)^{-1}.L}]_{\star}
\end{align*}
and similarly for $\Phi(\mathsf{s}_L y)$.
By assumption, we have that there exists $\eta \in \Spin_Q(\Z)$ such that $\eta(\gamma_1^x)^{-1}.L = (\gamma_1^y)^{-1}.L$. Therefore, $ \gamma_1^y \eta (\gamma_1^x)^{-1} \in \StSpin_L(\Q)$ and we obtain that
\begin{align*}
\Spin_Q(\R\times\widehat{\Z}) 
\ni b_1^x \eta (b_1^y)^{-1} = h_1^x \gamma_1^x \eta (\gamma_1^y)^{-1} (h_1^y)^{-1} \in \StSpin_L(\A).
\end{align*}
The element $h = (h_1, g_L^{-1}\isog(h_1)g_L) \in \bStDiagcpt_L$ corresponding to $h_1 = b_1^x \eta (b_1^y)^{-1} \in \StSpin_L(\RZ)$ satisfies $hy = x$.
To see this, note that
\begin{align*}
hy = hb^{y}\G(\Q) 
= (b_1^x \eta (b_1^y)^{-1}b_1^y\Spin_Q(\Q), g_L^{-1}\isog(b_1^x \eta (b_1^y)^{-1})g_Lb_2^y\SPnk(\Q)).
\end{align*}
For the first component we have $b_1^x \eta (b_1^y)^{-1}b_1^y\Spin_Q(\Q) = b_1^x \Spin_Q(\Q)$ because $\eta \in \Spin_Q(\Z)$. 
For the second component, we first recall that
\begin{align*}
b_2^y = h_2^y\gamma_2^y 
= g_L^{-1}\isog(t_1^y)g_L\gamma_2^y \quad \text{and} \quad b_1^x \eta (b_1^y)^{-1} 
= h_1^x \gamma_1^x \eta (\gamma_1^y)^{-1} (t_1^y)^{-1}.
\end{align*}
We may therefore rewrite
\begin{align*}
g_L^{-1}\isog(b_1^x \eta (b_1^y)^{-1})g_Lb_2^y\SPnk(\Q) = g_L^{-1}\isog(h_1^x \gamma_1^x \eta (\gamma_1^y)^{-1})g_L\gamma_2^y\SPnk(\Q).
\end{align*}
Using that $\gamma_2^y \in \SPnk(\Q)$ and $h_2^x = g_L^{-1}\isog(h_1^x)g_L$ we obtain:
\begin{align*}
 g_L^{-1}\isog(h_1^x \gamma_1^x \eta (\gamma_1^y)^{-1})g_L\gamma_2^y\SPnk(\Q) = h_2^xg_L^{-1}\isog(\gamma_1^x\eta(\gamma_1^y)^{-1})g_L \SPnk(\Q).
\end{align*}
Finally, $g_L^{-1}\isog(\gamma_1^x\eta(\gamma_1^y)^{-1})g_L \in \SPnk(\Q)$ because $\gamma_1^x\eta(\gamma_1^y)^{-1}$ stabilizes $L$, thus
\begin{align*}
h_2^xg_L^{-1}\isog(\gamma_1^x\eta(\gamma_1^y)^{-1})g_L \SPnk(\Q) = h_2^x \SPnk(\Q) = b_2^x\SPnk(\Q).
\end{align*}
It follows that $hx =y$ and the proof is complete.
\end{proof}

\subsection{The correct weights}\label{sec:weights}

Let $\mu_L$ be the Haar probability measure on the orbit $\mathsf{s}_L \StDiag_L(\A) \SG(\Q)\subset \SG(\A)/\SG(\Q)$ and let $\mu_L|_\U$ be the normalized restriction\footnote{Note that the normalized restriction is well-defined (i.e.~$\mu_L(\U)\neq 0$) as the intersection $\mathsf{s}_L \StDiag_L(\A) \SG(\Q)\cap \U$ contains $\mathsf{s}_L (\StDiag_L(\A) \cap \SG(\R\times \widehat{\Z}))\SG(\Q)$ which is open in $\mathsf{s}_L \StDiag_L(\A) \SG(\Q)$.} to $\U$.


We compute the measure of a fiber through any point $x\in \U$ in the packet.
\begin{Lem}
\label{lem:weights1}
Let $x \in \StDiag_L(\A) \G(\Q) \cap \U$ and write $\Phi(\mathsf{s}_L x) = [\hat{L},\Lambda_{\hat{L}}]_{\star}$.
Then
\begin{align}\label{eq:measurecptorbit}
\mu_L|_\U\big(\mathsf{s}_L\bStDiagcpt_L x\big) 
= \Big(\sum_{[L',\Lambda_{L'}]_{\star} \in \Pack(L)} \frac{|\StSpin_{\hat{L}}(\Z)|}{|\StSpin_{L'}(\Z)|}\Big)^{-1}.
\end{align}
\end{Lem}

\begin{proof}
We must trace through a normalization: let $m$ be the Haar measure on $\StDiag_L(\A)$ induced by requiring that $\mu_L$ is a probability measure and let $C_1 = m(\bStDiagcpt_L)$. 
Then
\begin{align}\label{eq:weight1}
\mu_L\big(\mathsf{s}_L\bStDiagcpt_L x\big) = \frac{C_1}{|\mathrm{Stab}_{\bStDiagcpt_L}(x)|}.
\end{align}
We compute the stabilizer. Write $x = b \SG(\Q)$ for some $b \in \SG(\RZ)$ and observe
\begin{align}\label{eq:weight2}
\mathrm{Stab}_{\bStDiagcpt_L}(x)
= b \mathrm{Stab}_{\bStDiagcpt_{\hat{L}}}(\SG(\Q))b^{-1}
\end{align}
as $\hat{L} = b^{-1}_{1,\infty}.L$. The intersection $\bStDiagcpt_{\hat{L}}\cap \SG(\Q)$ consists of rational elements $g$ of $\StDiag_{\hat{L}}(\Q)$ whose first component $g_1$ is in $\Spin_Q(\RZ)$. Equivalently, it is the subgroup of $\StDiag_{\hat{L}}(\Q)$ of elements $g$ with $g_1 \in \Spin_Q(\Z)$ which is clearly isomorphic to $\StSpin_{\hat{L}}(\Z)$.
In particular,
\begin{align*}
|\mathrm{Stab}_{\bStDiagcpt_L}(x)| = |\StSpin_{\hat{L}}(\Z)|.
\end{align*}

We now use that the one-to-one correspondence between $\Pack(L)$ and $\bStDiagcpt_L$-orbits in $\StDiag_L(\A) \SG(\Q) \cap \U$ (Lemma~\ref{Lemma: computation of fibers of Phi}). By summing \eqref{eq:measurecptorbit} over all such orbits we obtain
\begin{align*}
\mu_L(\U) = \sum_{[L',\Lambda_{L'}]_{\star} \in \Pack(L)} \frac{C_1}{|\StSpin_{L'}(\Z)|}
\end{align*}
which determines $C_1$. This concludes the lemma as by \eqref{eq:weight1} and \eqref{eq:weight2}
\begin{align*}
\mu_L|_\U\big(\mathsf{s}_L\bStDiagcpt_L x\big) = C_1 \mu_L(\U)^{-1} |\StSpin_{\hat{L}}(\Z)|^{-1}.
\end{align*}
\end{proof}

\subsubsection{Measures on $\lquot{\Spin_Q(\Z)}{\RkD}$}\label{sec:measures}
We have different measures on the set of cosets $\lquot{\Spin_Q(\Z)}{\RkD}$:
\begin{itemize}
\item $\nu_D$ is the pushforward of the normalized sum of Dirac measures on $\RkD$.
\item For any $L\subset \Q^n$ oriented $k$-dimensional with $\disc_Q(L) = D$ the measure $\nu_{\Pack(L)}$ is the pushforward of $\mu_L|_{\U}$ under the map~$\Phi$ defined in \eqref{Equation: map Phi in introduction}.
Here, the collection $\Pack(L)$ is defined in \eqref{eq:defP(L)}.
\end{itemize}
We claim that $\nu_D$ is a convex combination of the measures $\nu_{\Pack(L)}$ for $L$ varying with discriminant $D$.
The weights of the above measures may be computed explicitly.
Beginning with the former, note that the mass $\nu_D$ gives to a point $[\hat{L},\Lambda_{\hat{L}}]_{\star} \in \lquot{\Spin_Q(\Z)}{\RkD}$ is, up to a fixed scalar multiple, the number of preimages of $[\hat{L},\Lambda_{\hat{L}}]_{\star}$ under the quotient map $\RkD \rightarrow \lquot{\Spin_Q(\Z)}{\RkD}$. In other words, it is a constant times
\begin{align*}
\#\{[g.\hat{L},\Lambda_{g.\hat{L}}]: g \in \Spin_Q(\Z)\}
= \#\{g.\hat{L} : g \in \Spin_Q(\Z)\} 
= \frac{|\Spin_Q(\Z)|}{|\StSpin_{\hat{L}}(\Z)|}.
\end{align*}
By the same argument as in Lemma~\ref{lem:weights1}, we have (as $|\Spin_Q(\Z)|$ cancels out)
\begin{align}\label{eq:measureonquotofR(D)}
\nu_D\big([\hat{L},\Lambda_{\hat{L}}]_{\star}\big)  =  \Big(\sum_{[L',\Lambda_{L'}]_{\star} \in \lquot{\Spin_Q(\Z)}{\RkD}} \frac{1}{|\StSpin_{L'}(\Z)|}\Big)^{-1}\frac{1}{|\StSpin_{\hat{L}}(\Z)|}
\end{align}

On the other hand, the measure $\nu_{\Pack(L)}$ satisfies for any $[\hat{L},\Lambda_{\hat{L}}]_{\star} \in \Pack(L)$
\begin{align}\label{eq:measure on pushedpacket}
\nu_{\Pack(L)}\big([\hat{L},\Lambda_{\hat{L}}]_{\star}\big) 
= \Big(\sum_{[L',\Lambda_{L'}]_{\star} \in \Pack(L)} \frac{1}{|\StSpin_{L'}(\Z)|}\Big)^{-1}\frac{1}{|\StSpin_{\hat{L}}(\Z)|}
\end{align}
by Lemma~\ref{lem:weights1}.

Thus, the relative weights the measures $\nu_D$ and $\nu_{\Pack(L)}$ assign agree.
It follows from Remark~\ref{rem:equivrelation2} and \eqref{eq:measure on pushedpacket} and  \eqref{eq:measureonquotofR(D)} that $\nu_D$ is a convex combination of the measures $\nu_{\Pack(L)}$ as claimed.

\subsection{Conclusion}\label{sec:conclusion}

We now prove the remaining theorems.
We proved in \S \ref{sec:refinement moduli} that Theorem~\ref{thm:equidistrmoduli-better} implies Theorem~\ref{thm:jointwithshapes} when $k>2$ and Theorem~\ref{thm:equidistrmoduli}. So it is left to prove Theorem~\ref{thm:equidistrmoduli-better} and Theorem~\ref{thm:jointwithshapes} when $k=2$. 

\begin{proof}[Proof of Theorem~\ref{thm:equidistrmoduli-better}]
The key insight is that $\nu_{D_i}$ is a convex combination of measures that are equidistributed along any sequence of admissible subspaces. The assumption of $D_i$ to be $k$-power free implies admissibility.

Let $p$ be an odd prime not dividing $\disc(Q)$ and let $D_i \to \infty$ be a sequence of integers as in the assumptions of the theorem for the prime $p$.
We first claim that any sequence $L_i \in \mathcal{H}_{Q}^{n,k}(D_i)$ is admissible (cf.~\S\ref{sec:dimatleast4}).
Observe first that Condition \eqref{item: c1 admissibility} is automatic.
Also, the assumption $\kfree{D_i}{k}\to \infty$ implies Condition \eqref{item: c2 admissibility}. By Proposition~\ref{prop:disccomparison} and $n-k \geq k$ we have
\begin{align*}
\kfree{\disc(L_i^\perp)}{n-k} \geq \kfree{\disc(L_i^\perp)}{k} \asymp_Q \kfree{D_i}{k}
\end{align*}
proving Condition \eqref{item: c3 admissibility}.
Lastly, Condition~\eqref{item: c4 admissibility} follows from Propositions~\ref{prop:disccomparison} and \ref{Proposition: Strong Isotropicity criterion for HL } (where the former implies $p \nmid \disc_Q(L^\perp)$).

For any sequence $L_i$ as above together with an additional given orientation the measures $\nu_{\Pack(L_i)}$ equidistribute to the Haar measure on $\lquot{\Spin_Q(\Z)}{\Gridsor}$. Indeed, by admissibility the measures $\mu_{L_i}$ converge to the Haar measure $\mu$ on $\rquot{\SG(\A)}{\SG(\Q)}$ by Theorem~\ref{thm:dyn-simplyconnected}.
In particular, as $\U$ is compact open we have $\mu_{L_i}|_\U \to \mu|_\U$.
Taking the pushforward under $\Phi$ yields $\nu_{\Pack(L_i)} \to \nu$ where $\nu$ is the Haar measure on $\lquot{\Spin_Q(\Z)}{\Gridsor}$.

The fact that $\nu_{D_i}$ is a convex combination of the measures $\nu_{\Pack(L_i)}$ finally implies Theorem~\ref{thm:equidistrmoduli-better}. 
\end{proof}

\begin{proof}[Proof of Theorem~\ref{thm:jointwithshapes} for $k=2$]
Let $\bar{\U}$ be the principal genus of $\rquot{\bSG(\A)}{\bSG(\Q)}$. The following diagram commutes by construction:
\begin{center}
\begin{tikzcd}
    \rquot{\SG(\A)}{\SG(\Q)}\supset \U \arrow[twoheadrightarrow]{r}{} \arrow{d}{} & \bar{\U} \subset \rquot{\bSG(\A)}{\bSG(\Q)} \arrow{d}{} \\
   \lquot{\Spin_Q(\Z)}{\Gridsor} \arrow[twoheadrightarrow]{r}  & \lquot{\Spin_Q(\Z)}{\Gr{2}(\R)}\times \shapesp{2}\times \shapesp{n-2}.
  \end{tikzcd}
\end{center}
By Theorem~\ref{thm:dyn dim2},
the images of $\mathsf{s}_{L_i} \StDiag_{L_i}(\A)\SG(\Q)\cap \U$ in $\bar{\U}$ along any admissible sequence of subspaces $L_i$ are equidistributed.
On the other hand, by the above commutative diagram these images are given by the images of $\Pack(L_i)$ under the bottom map.
The rest of the argument is analogous to the case $k>2$.
\end{proof}

%% file: Sections/non-empty.tex
\section{Non-emptiness for the sum of squares}\label{sec:appendix}

In this section, we discuss non-emptiness conditions for the set $\HkQD$ when $Q$ is the sum of squares.
To simplify notation, we write $\mathcal{H}^{n,k}(D)$.
Note that we have a bijection
\begin{align*}
L \in \mathcal{H}^{n,k}(D) \mapsto L^\perp \in \mathcal{H}^{n,n-k}(D)
\end{align*}
as $Q$ is unimodular (see Proposition~\ref{Prop:Isomorphism between the quotients of piL(Rn) and piL perp (Rn)} and its corollary).
In view of our goal, we will thus assume that $k \leq n-k$ throughout.
We will also suppose that $n-k \geq 2$.

The question of when $\mathcal{H}^{n,k}(D)$ is non-empty is a very classical problem in number theory, particularly if $k=1$.
Here, note that $\mathcal{H}^{n,1}(D)$ is non-empty if and only if there exists a primitive vector $v \in \Z^n$ with $Q(v) =D$ (i.e.~$D$ is primitively represented as a sum of $n$ squares).
\begin{itemize}
\item For $n=3$, Legendre proved, assuming the existence of infinitely many primes in arithmetic progression, that $\mathcal{H}^{3,1}(D)$ is non-empty if and only if $D \not \equiv 0,4,7 \mod 8$.
A complete proof was later given by Gauss \cite{gauss}; we shall nevertheless refer to this result as Legendre's three squares theorem.
\item For $n=4$, Lagrange's four squares theorem states that $\mathcal{H}^{4,1}(D)$ is non-empty if and only if $D \not\equiv 0 \mod 8$.
\item For $n\geq 5$, we have $\mathcal{H}^{5,1}(D) \neq \emptyset$ for all $D \in \N$ as one can see from Lagrange's four square theorem. Indeed, if $D \not\equiv 0 \mod 8$ the integer $D$ is primitively represented as a sum of four squares and hence also of $n$ squares (by adding zeros). If $D \equiv0 \mod 8$ one can primitively represent $D-1$ as sum of four squares which yields a primitive representation of $D$ as sum of five squares.
\end{itemize}
When $k=2$, this question has been studied by Mordell \cite{mordell1,mordell2} and Ko \cite{ko}. In \cite{2in4}, the first and last named authors have shown together with Einsiedler that
\begin{align}\label{eq:non-emptyfrom2in4}
\mathcal{H}^{4,2}(D) \neq \emptyset \iff D \not\equiv 0,7,12,15 \mod 16.
\end{align}
This concludes all cases with $n \in \{3,4\}$. In this appendix we show by completely elementary methods the following.

\begin{Prop}\label{prop:non-empty}
Suppose that $n \geq 5$. Then $\mathcal{H}^{n,k}(D)$ is non-empty.
\end{Prop}

We first claim that it suffices to show that $\mathcal{H}^{5,2}(D)$ is non-empty. For this, observe that there exists for any $(n,k)$ injective maps
\begin{align}\label{eq:(5,2) enough}
\mathcal{H}^{n,k}(D) \hookrightarrow \mathcal{H}^{n+1,k}(D),\quad
\mathcal{H}^{n,k}(D) \hookrightarrow \mathcal{H}^{n+1,k+1}(D).
\end{align}
The first map is given by viewing $L \in \mathcal{H}^{n,k}(D)$ as a subspace of $\Q^{n+1}$ via $\Q^n \to \Q^{n}\times \{0\} \subset \Q^{n+1}$.
The second map associates to $L = \Q v_1 \oplus \ldots \Q v_k \in \mathcal{H}^{n,k}(D)$ the subspace $\Q (v_1,0) \oplus \ldots \oplus \Q (v_k,0) \oplus \Q e_{n+1} \in \mathcal{H}^{n+1,k+1}(D)$.
In particular, Proposition~\ref{prop:non-empty} for $(n,k) = (5,2)$ implies Proposition~\ref{prop:non-empty} for $(n,k) = (6,2),(6,3)$. One then proceeds inductively to verify the claim.

\subsection{A construction of Schmidt}
Though it is not, strictly speaking, necessary, we introduce here a conceptual construction of Schmidt \cite{Schmidt-count} that captures what can be done with inductive arguments as in~\eqref{eq:(5,2) enough}.
As before, we identify $\Q^n$ with a subspace of $\Q^{n+1}$ via $\Q^n \simeq \Q^{n}\times \{0\}$.
Given any $L \in \mathrm{Gr}_{n+1,k}(\Q)$ we have that either the intersection $L \cap \Q^{n}$ is $(k-1)$-dimensional or $L$ is contained in $\Q^n$.
In particular, we can write
\begin{align*}
\mathcal{H}^{n+1,k}(D) = \mathcal{H}^{n,k}(D) \sqcup \mathcal{H}_{\mathrm{nd}}^{n+1,k}(D)
\end{align*}
where $\mathcal{H}_{\mathrm{nd}}^{n+1,k}(D)$ denotes the subspaces $L\in \mathcal{H}^{n+1,k}(D)$ for which $L \not\subset \Q^n$.
We also let $\mathrm{Gr}_{n+1,k}^{\mathrm{nd}}(\Q)$ be the subspaces $L\in \mathrm{Gr}_{n+1,k}(\Q)$ for which $L \not\subset \Q^n$.
Here, '$\mathrm{nd}$' stands for 'non-degenerate'.

We now associate to $L \in \mathrm{Gr}_{n+1,k}^{\mathrm{nd}}(\Q)$ three quantities. Let $L' = L \cap \Q^n$. Furthermore, note that the projection of $L(\Z)$ onto the $x_{n+1}$-axis consists of multiples of some vector $(0,\ldots,0,h_L)$ where $h_L \in \N$.
As $(0,\ldots,0,h_L)$ comes from projection of $L(\Z)$, there exists some vector $(u_L,h_L) \in L(\Z)$. We define $v_L$ to be the projection of $u_L$ onto the orthogonal complement of $L'$ inside $\Q^n$.

\begin{Prop}[{\cite[\S5]{Schmidt-count}}]
\label{prop:Schmidtinductiveprocedure}
The following properties hold:
\begin{enumerate}[(i)]
\item For any $L \in \mathrm{Gr}_{n+1,k}^{\mathrm{nd}}(\Q)$ the pair $(h_L,v_L)$ is relatively prime in the following sense: there is no integer $d >1$ such that $d^{-1}h_L \in \N$ and $d^{-1}v_L \in \pi_{L'^{\perp}}(\Z^{n-1})$. 
\item
Let $(h,\bar{L},v)$ be any triplet with $h \in \N$, $\bar{L} \in \Gr{k-1}(\Q)$ and $v \in \pi_{\bar{L}}(\Z^{n-1})$ such that $(h_L,v_L)$ is relatively prime. 
Then there exists a unique $L \in \mathrm{Gr}_{n+1,k}^{\mathrm{nd}}(\Q)$ with $(h,\bar{L},v) = (h_L,L',v_L)$.
\item
We have
\[ \disc(L) = \disc(L')(h_L^2 + Q(v_L)). \]
\end{enumerate}
\end{Prop}

We remark that the construction in (ii) is quite explicit: If $u \in\Z^{n-1}$ satisfies $\pi_{\bar{L}}(u) = v$, one defines $L$ to be the span of $\bar{L}$ and the vector $(u,h)$. 

To illustrate this construction we show the direction in \eqref{eq:non-emptyfrom2in4} that we will need for Proposition~\ref{prop:non-empty}.

\begin{Lem}\label{lem:(4,2)non-empty}
If $D \in \N$ satisfies $D \not\equiv 0,7,12,15 \mod 16$, then $\mathcal{H}^{4,2}(D)$ is non-empty.
\end{Lem}

\begin{proof}
By Legendre's three squares theorem and \eqref{eq:(5,2) enough}, we have
\begin{align*}
D \not\equiv 0,4,7 \mod 8 \implies \mathcal{H}^{4,2}(D) \neq \emptyset.
\end{align*}
So suppose that $D$ is congruent to $4,8$ modulo $16$. In view of Proposition~\ref{prop:Schmidtinductiveprocedure}, we let $L'$ be the line through $(1,-1,0)$ so that $\disc(L') = 2$. Thus, it remains to find relatively prime $h \in \N$ and $v \in \pi_{L'}(\Z^3)$ with $\frac{D}{2} = h^2 + Q(v)$. Note that
\begin{align*}
\pi_{L'}(\Z^3) = \Z \frac{e_1+e_2}{2} + \Z e_3
\end{align*}
so that we may choose $v= a \frac{e_1+e_2}{2} + be_3$ for $a,b \in \Z$. Hence, we need to find a solution to
\begin{align*}
\frac{D}{2} = h^2 + \frac{a^2}{4}+ \frac{a^2}{4} + b^2 = h^2 + \frac{a^2}{2} + b^2
\end{align*}
such that $(h,a,b)$ is primitive.

Equivalently, this corresponds to finding a primitive representation of $D$ by the ternary form $x_1^2 + 2x_2^2 + 2x_3^2$.
This is again a very classical problem and has been settled by Dickson~\cite{Dickson}; as the argument is very short and elementary we give it here.
Note that $\frac{D}{4}$ is congruent to $1$ or $2$ modulo $4$ and hence there is $(x,y,z) \in \Z^3$ primitive with $x^2+y^2+z^2 = \frac{D}{4}$.
As $\frac{D}{4}\equiv 1,2 \mod 4$ at least one and at most two of the integers $x,y,z$ must be even. Suppose without loss of generality that $x$ is even and $y$ is odd. One checks that 
\begin{align*}
D = 2(x+y)^2 + 2(x-y)^2 + (2z)^2
\end{align*}
and observing that $(x+y,x-y,2z)$ is primitive as $x+y$ is odd, the claim follows in this case.
\end{proof}

\begin{proof}[Proof of Proposition~\ref{prop:non-empty}]
As explained, it suffices to consider the case $(n,k) = (5,2)$.
In view of Lagrange's four squares theorem and \eqref{eq:(5,2) enough}, we may suppose that $D \equiv 0 \mod 8$.
Moreover, we can assume that $D \equiv 0,7,12,15 \mod 16$ by \eqref{eq:(5,2) enough} and Lemma~\ref{lem:(4,2)non-empty}.
To summarize, we only need to consider the case $D \equiv 0 \mod 16$.

We again employ the technique in Proposition~\ref{prop:Schmidtinductiveprocedure}.
Consider the subspace $L' \subset \Q^4$ spanned by the vector $(1,-1,0,0)$ which has discriminant $2$.
Then
\begin{align*}
\pi_{L'}(\Z^4) = \Z \frac{e_1+e_2}{2} + \Z e_3 + \Z e_4
\end{align*}
and as in the proof of Lemma~\ref{lem:(4,2)non-empty} we need to find a primitive representation $(h,a,b,c)$ of $\frac{D}{2}$ as
\begin{align*}
\frac{D}{2} = h^2 + \frac{a^2}{2} + b^2 + c^2.
\end{align*}
Setting $a=2$ and observing that $\frac{D}{2}-2 \equiv 6 \mod 8$ the claim follows from Legendre's three squares theorem.
\end{proof}

%% file: Sections/primitive.tex
\section{More results around discriminants and induced forms}\label{sec:appendixprimitive}
%
%
%

The contents of this section of the appendix are of elementary nature and complement the results in \S\ref{SEC:quadratic_forms}.

\subsection{Local glue groups}\label{sec:local glue groups}

In this section we briefly explain how to compute the glue group in terms of local data.
This is largely in analogy to the local formula for the discriminant \eqref{eq:disc loctoglob}.
Define for any prime $p$
\begin{align*}
\mathcal{G}_p(L) = L(\Z_p)^\# / L(\Z_p)
\end{align*}
where we recall that $L(\Z_p) = L(\Q_p) \cap \Z_p^n$ and 
\begin{align*}
L(\Z_p)^\# = \{v \in L(\Q_p): \langle v,w\rangle \in \Z_p\}.
\end{align*}
Observe that $\mathcal{G}_p(L)$ is trivial for all but finitely many $p$. Indeed, $\mathcal{G}_p(L)$ is trivial if $L$ is $p$-unimodular for an odd prime $p$, i.e.~$p\nmid \disc_Q(L)$ (see also Remark~\ref{rem:glueloc} for a much finer statement).
Also, it is easy to adapt Lemma~\ref{lem:imageoforthproj} and Proposition~\ref{Prop:Isomorphism between the quotients of piL(Rn) and piL perp (Rn)} to their local analogues.
Here, we prove the following:

\begin{Lem}
We have
\begin{align}\label{eq:glue loctoglob}
\mathcal{G}(L) \simeq \prod_p \mathcal{G}_p(L).
\end{align}
\end{Lem}

Taking cardinalities, \eqref{eq:glue loctoglob} encodes the (obvious) local product formula for discriminants~\eqref{eq:disc loctoglob}.

\begin{proof}
The image of the natural inclusion $L(\Z) \hookrightarrow L(\Z_p)$ is dense for every $p$.
In particular, the image of $L(\Z)^\#$ under $L(\Q) \hookrightarrow L(\Q_p)$ lies in $L(\Z_p)^\#$ and is dense therein. We obtain a homomorphism $\iota:\mathcal{G}(L) \to \prod_p \mathcal{G}_p(L)$.
We prove that $\iota$ is the desired isomorphism. Let $(v_i)_i$ be an integral basis of $L(\Z)$.

Let $v+L(\Z)$ be in the kernel of $\iota$. Then $v \in L(\Z_p)$ for every $p$ or, equivalently, the coordinates of $v$ in the $\Z$-basis $(v_i)_i$ of $L(\Z)$ have no denominators in $p$ for every $p$. Hence $v \in L(\Z)$ and $\iota$ is injective.

As $\mathcal{G}_p(L)$ is trivial for all but finitely many $p$, it suffices to find for any $v \in L(\Z_p)^\#$ an element $w \in L(\Z)^\#$ with $w+L(\Z_p) = v+L(\Z_p)$ and $w \in L(\Z_q)$ for any $q \neq p$.
Let $v \in L(\Z_p)^\#$ and write $v = \sum_{i}\alpha_i v_i$ where $\alpha_i \in \Q_p$. For every $i$ let $\beta_i \in \Z[\frac{1}{p}]$ be such that $\alpha_i \in \beta_i + \Z_p$ and set $w = \sum_i \beta_i v_i\in L(\Q)$ as well as $u = w-v\in L(\Z_p)$.
Then clearly for every $i$
\begin{align*}
\langle w, v_i\rangle = \langle v,v_i\rangle + \langle u,v_i\rangle \in \Z_p,
\end{align*}
that is, $w \in L(\Z_p)^\#$, and $\langle w,v_i\rangle \in \Z[\frac{1}{p}]$. But $\Z_p \cap \Z[\frac{1}{p}] = \Z$ and hence $w \in L(\Z)^\#$.
Observe also that by construction $w \in L(\Z_q)$ for every prime $q \neq p$.
Hence $\iota$ is surjective.
\end{proof}

\begin{remark}\label{rem:glueloc}
The isomorphism in \eqref{eq:glue loctoglob} is particularly handy when one tries to explicitly compute glue-groups.
Indeed, recall that for any odd prime $p$ an integral quadratic form $q$ over $\Z_p$ is diagonalizable \cite[Ch.~8]{Cassels}.
For
\begin{align*}
q(x_1,\ldots,x_k) = \alpha_1p^{\ell_1}x_1^2 + \ldots + \alpha_k p^{\ell_k} x_k^2
\end{align*}
with units $\alpha_i \in \Z_p^\times$ and $\ell_i \geq 0$,
the glue-group is
\begin{align*}
\Z/p^{\ell_1}\Z \times \ldots \times \Z/p^{\ell_k}\Z.
\end{align*}
For $p=2$ an integral quadratic form $q$ need not be diagonalizable over $\Z_2$. However, by \cite[Lemma 4.1]{Cassels} we may write $q$ as a (direct) sum of forms of the following types in distinct variables:
\begin{equation} \label{eq: types of forms for p=2}
2^{\ell}\alpha x_1^2, \quad 2^{\ell} (2x_1x_2) \quad \text{and} \quad 2^{\ell}  (2x_1^2 + 2x_1x_2 + 2x_2^2)
\end{equation}
with $\ell\geq 0$ and $\alpha \in \Z_2^{\times}$.
An elementary computation leads to observing that the glue groups of the quadratic forms in \eqref{eq: types of forms for p=2} are respectively:
\begin{equation} \label{eq: types of glue groups for p=2}
\Z/ 2^{\ell}\Z \quad \Z/ 2^{\ell}\Z \times \Z/ 2^{\ell}\Z  \quad \text{and} \quad \Z/2^{\ell}\Z \times \Z/ 2^{\ell}\Z. 
\end{equation}
It follows that the glue group has essentially the same structure as in the case of $p$ odd.
More precisely, assume that
\[ q(x_1,\ldots,x_k) = q_1 + \cdots + q_m \]
where the $q_{i}$'s are forms as in \eqref{eq: types of forms for p=2} with exponents $\ell = \ell_i$ satisfying $\ell_1 \leq \ldots \leq \ell_m$, Then the glue group is a product of groups as in \eqref{eq: types of glue groups for p=2} with exponents $\ell_1 \leq \ldots \leq \ell_m$.
\end{remark}

\subsection{Indices of projected lattices}\label{sec:indices}

For any subspace $L\subset \Q^n$ we denote the index of $L(\Z)$ in $L \cap (\Z^n)^\#$ by $i(L)$.
Then the proof of Proposition~\ref{prop:disccomparison} and Lemma~\ref{lem:imageoforthproj}
shows that
\begin{align*}
\disc_Q(L^\perp) = \frac{i(L^\perp)}{i(L)} \disc_Q(L).
\end{align*}
The following proposition establishes a fundamental relation between the indices for $L$ and $L^\perp$.

\begin{Prop}
\label{Appendix. Proposition: short exact sequence with L cap R dual and L(R) dual}
Let $L\subset \Q^n$ be a subspace.
The sequence
\[ 0 \rightarrow (L^{\perp} \cap (\Z^n)^{\#})/L^{\perp}(\Z) \rightarrow  (\Z^n)^{\#}/\Z^n \rightarrow L(\Z)^{\#}/ \pi_L(\Z^n) \rightarrow 0. \]
obtained by inclusion and projection is exact.
In particular,
\begin{align*}
i(L) i(L^\perp) = \disc(Q).
\end{align*}
Similarly, for any prime $p$
\begin{align*}
[L(\Q_p)\cap (\Z_p^n)^\#:L(\Z_p)]\cdot [L^\perp(\Q_p)\cap (\Z_p^n)^\#:L^\perp(\Z_p)] = p^{\nu_p(\disc(Q))}.
\end{align*}
\end{Prop}

\begin{proof}
By Lemma~\ref{lem:imageoforthproj}, the orthogonal projection $\pi_L$ defines a surjective morphism
\[ f: (\Z^{n})^{\#} \rightarrow L(\Z)^{\#}/ \pi_L(\Z^n). \]
The kernel of this morphism can be described by:
\begin{equation}
\label{Appendix. Equation: characterization of kernel in proof of short exact sequence}
\ker(f) = \{ v \in (\Z^{n})^{\#} : \text{there exists} \ w \in \Z^n \ \text{such that} \ v-w \in L^{\perp} \}.
\end{equation}
Clearly $L^{\perp} \cap (\Z^n)^{\#}   \subset \ker(f)$. We claim that the inclusion of $L^{\perp} \cap (\Z^n)^{\#}$ into $\ker(f)$ induces an isomorphism:
\[ L^{\perp} \cap (\Z^n)^{\#} / L^{\perp}(\Z) \rightarrow \ker(f)/\Z^n. \]
The fact that the map $L^{\perp} \cap (\Z^n)^{\#} \rightarrow \ker(f)/\Z^n$ induced by the inclusion is surjective follows immediately from the characterization of $\ker(f)$ in \eqref{Appendix. Equation: characterization of kernel in proof of short exact sequence}. Since the kernel of this map is clearly $L^{\perp}(\Z)$, the claim is proven.
It follows that
\[ 0 \rightarrow L^{\perp} \cap (\Z^n)^{\#}/L^{\perp}(\Z) \rightarrow  (\Z^n)^{\#}/\Z^n \rightarrow L(\Z)^{\#}/ \pi_L(\Z^n) \rightarrow 0. \]
is a short exact sequence.
The local analogue follows similarly.
\end{proof}

\begin{remark}
It would be interesting to see statistical results regarding these indices. To give a concrete example, suppose $\disc(Q) =2$. Then clearly $i(L) \in \{1,2\}$ for any subspace $L$ and one can ask what the proportion of subspaces $L$ with $i(L) =1$ (or $i(L^\perp) =2$) is. If $k=n-k$, Proposition~\ref{Appendix. Proposition: short exact sequence with L cap R dual and L(R) dual} shows that the number of subspaces with $i(L) = 1$ and $i(L) =2$ is the same.
\end{remark}

\subsection{Primitive forms}

Here, we study to what extent the induced forms $q_L,q_{L^\perp}$ (defined in \S\ref{sec:notation-induced qf} up to equivalence) for a given subspace $L \in \Gr{k}(\Q)$ need to be primitive.
For instance, we establish that for $k<n-k$ the form $q_{L^\perp}$ needs to be essentially primitive (while $q_{L}$ does not).
First, observe that indeed the form $q_L$ need not be primitive:

\begin{Ex}\label{ex:non-prim}
Let $n \geq 6$, let $(e_i)_{i=1}^n$ denote the standard basis vectors of $\Q^n$ and suppose that $Q = Q_0$ is the standard positive definite form.
Let $(v_1,v_2) \in \Z^2$ be a primitive vector.
Then the integer lattice in the subspace 
\begin{align*}
 L = \mathrm{span}_{\Q} \{ v_1e_1+v_2e_2, v_1e_3+v_2e_4, v_1e_5+v_2e_6\}
\end{align*}
is spanned by $v_1e_1+v_2e_2, v_1e_3+v_2e_4, v_1e_5+v_2e_6$ which are orthogonal vectors. 
In this basis,
\begin{align*}
q_L(x_1,x_2,x_3) = (v_1^2+v_2^2) (x_1^2+x_2^2+x_3^2)
\end{align*}
which is a highly non-primitive form.
Similarly, $L^\perp(\Z)$ is spanned by the integer vectors $v_2e_1-1v_1e_2, v_2e_3-v_1e_4, v_2e_5-v_1e_6,e_{7},\ldots,e_n$ and hence in this basis
\begin{align*}
q_{L^\perp}(x_1,\ldots,x_{n-3}) = (v_1^2+v_2^2) (x_1^2+x_2^2+x_3^2) + x_4^2+\ldots + x_{n-3}^2
\end{align*}
In particular, $q_{L^\perp}$ is primitive if $n>3$ and otherwise $\gcd(q_{L^\perp})= \gcd(q_{L})$ (as $q_{L^\perp}=q_L$ in this specific example).
This type of behavior is generally true as established below.
For more examples we also refer to \cite[Example 2.4]{2in4}.
\end{Ex}

\begin{Prop}\label{prop:primitivity}
Let $L \in \Grk(\Q)$. If $k > n-k$, $\gcd(q_L)$ divides $\disc(Q)$ and
\begin{align*}
\disc(\tilde{q}_L) \asymp_Q \disc_Q(L).
\end{align*}
Conversely, if $k < n-k$, $\gcd(q_{L^\perp})$ divides $\disc(Q)$ and $\disc(\tilde{q}_{L^\perp}) \asymp_Q \disc_Q(L)$.

Moreover, if $k=n-k$ we have $\gcd(q_L) \asymp_Q \gcd(q_{L^\perp})$ and 
\begin{align*}
\disc(\tilde{q}_L) \asymp_Q \disc(\tilde{q}_{L^\perp}).
\end{align*}
\end{Prop}

For convenience of the reader, we provide two proofs of the first claim in the proposition; the second uses glue groups and generalizes to $k=n-k$.

\begin{proof}[First proof for $k \neq n-k$]
Fix a basis $v_1, \ldots , v_k$ of $L(\Z)$ and complete it into a basis $v_1, \ldots ,v_n$ of $\Z^n$. 
Let $v_1^{*}, \ldots , v_n^{*}$ be its dual basis. Since $k > n-k$, without loss of generality we may assume $v_1 \in \spn_{\R}(v_{k+1}, \ldots ,v_n)^{\perp}$.
Note that $v_1^{*} \in (\Z^n)^{\#}$ and so $\disc(Q)v_1^{*} \in \Z^n$. In particular, we may write
\[ \disc(Q)v_1^{*} = \sum_{s \leq n} a_s v_s \ \text{with} \ a_s \in \Z. \]
By our choice of $v_1$ we have
\[ \disc(Q) = \langle \disc(Q)v_1^{*}, v_1\rangle_Q = \sum_{s \leq k} a_s \langle v_s,v_1\rangle_Q \]
and the first claim follows as $\gcd(q_L)$ divides the right-hand side.
\end{proof}

\begin{proof}
Given a prime $p$ we write $\ord_p(q_L)$ for the largest integer $m$ with $p^m\mid\gcd(q_L)$.
Note that $\ord_p(q_L)$ can be extracted from the glue-group of $L$ whenever $p \mid \gcd(q_L)$ -- see Remark~\ref{rem:glueloc}.

To begin the proof, fix $p$ and note that $a_L := \ord_p(q_L)$ can be characterized as follows: it is the smallest integer $m$ so that there exists a primitive vector $v \in L(\Z_p)^\#$ with $p^m v \in L(\Z_p)$.
To see this, first assume $p$ is an odd prime. 
Then, as in Remark~\ref{rem:glueloc} (after possibly changing the basis), we may write
\begin{align*}
q_L(x_1,\ldots,x_k) = \alpha_1p^{\ell_1}x_1^2 + \ldots + \alpha_k p^{\ell_k} x_k^2
\end{align*}
with $\ell_1 \leq \ell_2 \leq \ldots \leq \ell_k$.
If $v$ is a vector as above, the expression for the glue-group in Remark~\ref{rem:glueloc} as well as primitivity imply that $m \geq \ell_1$. Conversely, it is easy to see that the first vector $v$ in the above (implicit) choice of basis of $L(\Z_p)$ satisfies $p^{-\ell_1}v \in L(\Z_p)^{\#}$ and is primitive. For $p=2$ the proof above can be adapted using Remark \ref{rem:glueloc}.

Define $a_{L}'$ as the smallest integer $m$ so that there exists a primitive vector $v' \in \pi_L(\Z_p^n)$ with $p^m v' \in L(\Z_p)$. We argue that $a'_L \leq a_L$. Let $v$ be as in the above definition of $a_L$.
Then, there exists an integer $i \leq a_L$ such that $p^iv \in \pi_L(\Z_p^n)$ and $p^iv$ is primitive in $\pi_L(\Z_p^n)$. 
For this integer $i$ set $v':=p^iv$ and observe that $p^{a_L-i}v' = p^{a_L}v \in L(\Z_p)$. Therefore, $a_L' \leq a_L-i \leq a_L$ as claimed. In analogous fashion, one argues that $a_L \leq a_L' + \ord_p(i_p(L))$ so that the following inequalities hold:
\begin{align*}
a_L' \leq a_L \leq a_L' + \ord_p(i_p(L)).
\end{align*}

Suppose now that $k>n-k$. Applying Proposition~\ref{Prop:Isomorphism between the quotients of piL(Rn) and piL perp (Rn)} we see that there exists $v' \in \pi_L(\Z_p^n)$ primitive with $v'\in L(\Z_p)$. Indeed, as $\pi_{L^\perp}(\Z_p^n)/L^\perp(\Z_p)$ is a product of at most $k$ non-trivial cyclic groups, the same is true for $\pi_{L}(\Z_p^n)/L(\Z_p)$ implying the claim.
Therefore, $a_L' = 0$ and hence $a_L \leq \ord_p(i_p(L))$. This shows that $\gcd(q_L) \mid i(L)$ which proves a sharpened version of the first part of the proposition (cf.~Proposition~\ref{Appendix. Proposition: short exact sequence with L cap R dual and L(R) dual}).

Suppose now that $k=n-k$. We show first that $a_L'=a_{L^\perp}'$. If $a_L'=0$, $\pi_{L}(\Z_p^n)/L(\Z_p)$ is a product of at most $k-1$ cyclic groups and hence the same is true for $\pi_{L^\perp}(\Z_p^n)/L^\perp(\Z_p)$ by Proposition~\ref{Prop:Isomorphism between the quotients of piL(Rn) and piL perp (Rn)}. This implies that $a_{L^\perp}'=0$.
If $a_L'\neq 0$, the number $a_L'$ is exactly the smallest order of a non-trivial element in $\pi_{L}(\Z_p^n)/L(\Z_p)$.  Applying the same for $L^\perp$ yields $a_L'=a_{L^\perp}'$ in all cases.
In particular,
\begin{align*}
a_L \leq a_{L^\perp}'+ \ord_p(i_p(L)) \leq a_{L^\perp}+ \ord_p(i_p(L)).
\end{align*}
Varying the prime $p$ we obtain that
\begin{align*}
\gcd(q_L) \mid \gcd(q_{L^\perp})i(L)
\end{align*}
and conversely. This finishes the proof of the proposition.
\end{proof}